\documentclass[11pt]{article}
\usepackage{amsmath}
\usepackage[utf8]{inputenc}
\usepackage{tikz}
\usepackage{amsthm}
\usepackage{thmtools}
\usepackage{hyperref}
\usepackage{cleveref}
\usepackage[a4paper, total={8in, 10in}]{geometry}
\usepackage{blindtext}
\usepackage{comment}
\usepackage{geometry}
\usepackage{amsmath}
\NewDocumentEnvironment{alignb}{b}{%
	\begin{align*}
		\refstepcounter{equation} #1 \tag{\theequation}
	\end{align*}
}{}
\allowdisplaybreaks
\usepackage{enumerate}
\makeatletter
\newcommand{\myitem}[1]{%
	\item[#1]\protected@edef\@currentlabel{#1}%
}
\makeatother
\usepackage{clipboard}
\declaretheorem[numberwithin=section]{theorem, definition}
\declaretheorem{lemma, proposition, remark}[style=plain,
numberwithin=section]
\numberwithin{equation}{section}
\newtheorem{mytheorem}{Theorem}

\numberwithin{mytheorem}{section}
\usepackage{amsbsy}
\usepackage{graphicx}
\usepackage[T1]{fontenc}
\usepackage{lmodern}
\usepackage{imakeidx}
\usepackage{amssymb}
\usepackage{thmtools}
\usepackage{cite}
\newenvironment{myproof}[2] {\paragraph{Proof of {#1} {#2} :}}{\hfill$\square$}

\usepackage[english]{babel}
\title{Multiplicity of solutions with prescribed mass for a quasilinear critical Choquard equation driven by a local-nonlocal operator}
\author{J. Giacomoni\footnote{LMAP (UMR E2S UPPA CNRS 5142) Bat. IPRA, Avenue de l'Universit\'{e}, 64013 Pau, France. e-mail: jacques.giacomoni@univ-pau.fr}\;, Nidhi Nidhi\footnote{Department of Mathematics, Indian Institute of Technology, Delhi, Hauz Khas, New Delhi-110016, India. e-mail: nidhi.nidhi@maths.iitd.ac.in; nidhi.kaushik2809@gmail.com}\; and K. Sreenadh\footnote{Department of Mathematics, Indian Institute of Technology, Delhi, Hauz Khas, New Delhi-110016, India. e-mail: sreenadh@maths.iitd.ac.in}\;}
\date{}
\begin{document}
	\maketitle
	\begin{abstract}
		\noindent In this paper we study the normalized solutions of the following critical growth Choquard equation with mixed local and non-local operators: 
		\begin{equation*}
			\begin{array}{rcl}
				-\Delta_p u +(-\Delta_p)^s u & = & \lambda |u|^{p-2}u +\mu |u|^{q-2}u +(I_{\alpha}*|u|^{p^*_{\alpha}})|u|^{p^*_{\alpha}-2}u \text{ in } \mathbb{R}^N,\\
				\left\| u \right\|_p & =  & \tau.
			\end{array}
		\end{equation*}
		Here, $N\geq 3$, $2 \le p<N$, $\tau>0$, $I_{\alpha}$ is the Riesz potential of order $\alpha\in (\max\{0,N-2p\}, N)$, $p^*_{\alpha}=\frac{p}{2}\left(\frac{N+\alpha}{N-p}\right)$ is the critical exponent corresponding to the Hardy Littlewood Sobolev inequality, $(-\Delta_p)^s$ is the non-local fractional p-Laplacian operator with $s\in (0,1)$, $\mu>0$ is a parameter and $\lambda$ appears as a Lagrange multiplier. We show the existence of at least two distinct solutions in the presence of a mass subcritical perturbation, $\mu |u|^{q-2}u$ with $p<q<p+\frac{sp^2}{N}$ under some conditions on $p,N$ and $s$.
	\end{abstract}
	\section{Introduction}
	This article concerns the existence of multiple normalized solutions to the following quasilinear critical growth Choquard equation involving mixed diffusion-type operator:
	\begin{equation}\label{prob}
		\begin{array}{rcl}
			-\Delta_p u +(-\Delta_p)^s u & = & \lambda |u|^{p-2}u +\mu |u|^{q-2}u +(I_{\alpha}*|u|^{p^*_{\alpha}})|u|^{p^*_{\alpha}-2}u \text{ in } \mathbb{R}^N,\\
			\left\| u \right\|_p & =  & \tau,
		\end{array}
	\end{equation}
	where $N\geq 3$, $\tau>0$, $2\leq p<N$, $p<q<p+\frac{sp^2}{N}$, $\mu>0$ is a parameter and $\lambda$ appears as a Lagrange multiplier.
	The operators $p$-{L}aplacian ($\Delta_p$) and fractional $p$-{L}aplacian $(-\Delta_p)^s$ are defined as:
	$$\Delta_pu=div(|\nabla u|^{p-2}\nabla u),$$
	and
	$$(-\Delta_p)^su(x)=\lim_{\epsilon\rightarrow 0}\int_{\mathbb{R}^N\setminus B_{\epsilon}(0)}\frac{|u(x)-u(y)|^{p-2}(u(x)-u(y))}{|x-y|^{N+sp}}dy \text{ for } s\in (0,1).$$
	Here, $I_{\alpha}$ is the Riesz potential of order $\alpha\in (\max\{0,N-2p\},N)$ given by
	\begin{equation}\label{A_alpha}
		I_{\alpha}(x)=\frac{A_{N,\alpha}}{|x|^{N-\alpha}}\text{ with } A_{N,\alpha}=\frac{\Gamma(\frac{N-2}{2})}{\pi^{\frac{N}{2}}2^{\alpha}\Gamma(\frac{\alpha}{2})}\text{ for every }x\in\mathbb{R}^N\setminus \{0\},
	\end{equation}
	and $p^*_{\alpha}=\left(\frac{p}{2}\right)\left(\frac{N+\alpha}{N-p}\right)$, is the critical exponent with respect to the following well known Hardy-Littlewood-Sobolev(HLS) inequality \cite[Theorem~4.3]{Lieb2001Analysis}:
	\begin{proposition}\label{prop1}
		Let $t,r>1$ and $0<\alpha <N$ with $1/t+1/r=1+\alpha/N$, $f\in L^t(\mathbb{R}^N)$ and $h\in L^r(\mathbb{R}^N)$. There exists a sharp constant $C(t,r,\alpha,N)$ independent of $f$ and $h$, such that
		\begin{equation}\label{HLS}
			\int_{\mathbb{R}^N}\int_{\mathbb{R}^N}\frac{f(x)h(y)}{|x-y|^{N-\alpha}}\, {dx dy}\leq C(t,r,\alpha,N) \|f\|_{L^t(\mathbb{R}^N)}\|h\|_{{L^r(\mathbb{R}^N)}}.
		\end{equation}
		If $t=r=2N/(N+\alpha)$, then
		\begin{align}\label{C_alpha}
			C(t,r,\alpha,N)=C(N,\alpha)= \pi^{\frac{N-\alpha}{2}}\frac{\Gamma(\frac{\alpha}{2})}{\Gamma(\frac{N+\alpha}{2})}\left\lbrace \frac{\Gamma(\frac{N}{2})}{\Gamma(N)}\right\rbrace^{-\frac{\alpha}{N}}.
		\end{align}
		Equality holds in  \eqref{HLS} if and only if ${f}/{h}\equiv constant$ and
		$\displaystyle h(x)= A(\gamma^2+|x-a|^2)^{-(N+\alpha)/2}$
		for some $A\in \mathbb{C}, 0\neq \gamma \in \mathbb{R}$ and $a \in \mathbb{R}^N$.
	\end{proposition}
	\noindent From the inequality \eqref{HLS}, it follows that for any $u\in W^{1,p}(\mathbb{R}^N)$
	\begin{align*}
		{\mathcal{A}_r(u):=	\int_{\mathbb R^N}\int_{\mathbb R^N}\frac{A_{N,\alpha}|u(x)|^{r}|u(y)|^{r}}{|x-y|^{N-\alpha}}}{ dx dy}
	\end{align*}
	is well defined if $p_*^{\alpha}:=\left(\frac{p}{2}\right)\left(\frac{N+\alpha}{N}\right)\leq r \leq \left(\frac{p}{2}\right)\left(\frac{N+\alpha}{N-p}\right)=p^*_{\alpha}$.\\
	The exponent $r = p^*_{\alpha}$ is known as the Hardy-Littlewood-Sobolev critical exponent and similar to the usual critical exponent, $W^{1,p}_0(\Omega)\ni\, u\mapsto \mathcal{A}_{p^*_\alpha}(u)$ is continuous for the norm topology but not for the weak topology (see \cite{Moroz2017guide}). Thus, the presence of this HLS critical exponent ($p^*_{\alpha}$) makes our problem challenging and intriguing to work on.  	\\
	\noindent Equations involving nonlinearity of the form $(I_{\alpha}*|u|^q)|u|^{q-2}u$ are called {\it Choquard equation}, as in 1976, Choquard, at the Symposium on Coulomb Systems  
	utilised the energy functional associated to equation 
	\begin{equation}\label{choq}
		\left\{ \begin{array}{rl}   	
			& 	-\Delta u  +  u = (I_2*|u|^2)u\;\;\text{in } \mathbb{R}^3,\\
			&  	u \in H^1(\mathbb{R}^3),
		\end{array}
		\right.
	\end{equation}
	to examine a viable approximation to Hartree-Fock theory for a one-component plasma (see \cite{lieb1977existence}). The equation has various other applications in quantum physics, for instance, it is used to characterise an electron confined within its own vacancy, see \cite{penrose1996gravity} and related sources. Several works have ever since conducted research on the existence, multiplicity, and qualitative characteristics of the solution to the Choquard type equations
	as detailed in \cite{filippucci2020singular, Moroz2013groundstates, liu2022another}. \\
	\noindent In particular, for $p=2$, the problem \eqref{prob} becomes the following:
	\begin{equation}\label{p=2}
		\begin{array}{rcl}
			-\Delta u +(-\Delta)^s u & = & \lambda u +\mu |u|^{q-2}u +(I_{\alpha}*|u|^{2^*_{\alpha}})|u|^{2^*_{\alpha}-2}u \text{ in } \mathbb{R}^N\\
			\left\| u \right\|_2 & = & \tau,
		\end{array}
	\end{equation}
	where $2^*_{\alpha}=\frac{N+\alpha}{N-2}$. A solution of \eqref{p=2} gives a standing wave solution to the following Schr\"odinger equation driven by both local and nonlocal operators:
	\begin{equation}\label{Schrodinger_equation}
		i \frac{\partial \psi}{\partial t} = -\Delta \psi+(-\Delta)^s\psi -\mu |\psi|^{p-2} \psi -(I_{\alpha}*|\psi|^{2^*_{\alpha}})|\psi|^{2^*_{\alpha}-2}\psi,
	\end{equation}
	and is called a normalized solution or a solution with prescribed mass. Recently such problems involving fixed mass constraints have attracted many researchers. One can quote for instance the works of \cite{gou2018multiple,bartsch2016normalized,bartsch2018normalized,bartsch2019multiple,noris2014stable,noris2019normalized,Giacomoni2025Normalized,Nidhi2025Existence,Nidhi2025Normalized_Asymp}, where authors studied the existence, multiplicity and regularity of normalized solutions for some nonlinear Schrödinger equations with several local and nonlocal nonlinearities. The general quasilinear case is also studied by various authors: for instance, Feng et al. in \cite{Feng2023Normlaized} studied the problem:
	\begin{equation}
		\left\{
		\begin{array}{ll}
			-\Delta_p u  & = \lambda |u|^{p-2}+\mu |u|^{q-2}u +|u|^{p^*-2}u \text{ in } \mathbb{R}^N,\\
			\left\| u \right\|_p & =c, 
		\end{array}
		\right.
	\end{equation}
	and deduced the existence and multiplicity of normalized solutions for different values for $q$. A similar problem with a general nonlinearity has been tackled by the authors in \cite{Zhang2022normalized}. Also, one can see the work of \cite{Liu2025Solutions}, where the normalized solutions to a critical growth Poisson system involving a p-Laplacian operator have been studied in details. As far as the fractional p-Laplacian operator is concerned, one may refer to the studies by \cite{Cui2025normalized} and \cite{Yu2026Normalized} and references therein. Moreover, the problem involving both p-Laplacian and fractional p-Laplacian is discussed in \cite{Nidhi2025Normlaized_MMA}.
	
	\noindent We dedicated this study to the existence of multiple normalized solutions to a problem concerning both the p-Laplacian and the fractional p-Laplacian operator. The rationale behind this work can be derived from the study of Gou et al. \cite{gou2018multiple}, where they looked at the problem:
	\begin{equation}
		\left\{
		\begin{array}{ll}
			-\Delta u & = \lambda u +\mu |u|^{q-2}u +|u|^{2^*-2}u \text{ in } \mathbb{R}^N,\\
			\left\| u \right\|_2 & = c,
		\end{array}
		\right.
	\end{equation}
	with $2<q<2+\frac{4}{N}$. They demonstrated the existence of a ground-state solution and a second solution with energy strictly less than a fixed constant depending on $c, N$ and the best Sobolev constant $S$. Additionally, X. Sun and Z. Han in \cite{Sun2024Note} studied the same problem with the fractional Laplacian operator and derived analogous results. Further, in \cite{Nidhi2025Existence} and \cite{Nidhi2025Normalized_Asymp}, the case of a mixed local-nonlocal operator and critical Choquard nonlinearity has been discussed, precisely they studied the problem:
	\begin{equation*}
		\left\{
		\begin{array}{rl}
			-\Delta u +(-\Delta)^s u & =\lambda u +\mu |u|^{q-2}u +(I_{\alpha}*|u|^{2^*_{\alpha}})|u|^{2^*_{\alpha}-2}u \text{ in } \mathbb{R}^N, \\
			\left\| u \right\|_2 & = \tau,\\
		\end{array}
		\right.
	\end{equation*}
	and deduced the following:
	\begin{center}
		{\small
			\begin{tabular}{ | p{2.5cm} | p{5cm} | p{4cm} | }
				\hline 
				Range of $q$ & Type of solution & Energy level \\
				\hline 
				$2<q<2+\frac{4s}{N}$ & a local minimizer & $=m_{\tau}<0$  \\
				\hline
				& second solution & $<m_{\tau}+\left(\frac{2^*_{\alpha}-1}{22^*_{\alpha}}\right)S_{\alpha}^{\frac{2^*_{\alpha}}{2^*_{\alpha}-1}}$ \\
				\hline
				$2+\frac{4}{N}\leq q < 2^*$ &  a mountain pass type solution & $<\left(\frac{2^*_{\alpha}-1}{22^*_{\alpha}}\right)S_{\alpha}^{\frac{2^*_{\alpha}}{2^*_{\alpha}-1}}$  \\
				\hline 
		\end{tabular}}
	\end{center}
	where 
	$$S_{\alpha}=\inf_{u\in D^{1,2}(\mathbb{R}^N)\setminus \{0\}}\frac{\left\| \nabla u\right\|_2^2}{\mathcal{A}_{2^*_{\alpha}}(u)^{\frac{1}{2^*_{\alpha}}}}.$$
	{ Here, we would like to extend these results for a generic quasilinear variant of the problem, as stated in \eqref{prob}.  The variational framework implemented for the semilinear setting ($p=2$) relies fundamentally on analyzing the fiber maps $\psi_u(t)=E(t\star u)$ to characterize the geometry of the Poho\v{z}aev manifold, utilizing the Ekeland variational principle to construct minimizing sequences, and recovering compactness below a sharp energy threshold. While the quasilinear setting investigated in this article maintains a structurally parallel uniform scaling profile, extending these arguments to cases where $p \neq 2$ introduces severe analytical hurdles. Firstly, without the quadratic nature of $p=2$, verifying the strict monotonicity, convexity, and structural behavior of the fiber maps becomes significantly more tedious, forcing a reliance on delicate algebraic inequalities rather than clean exact identities (see in particular Lemma \ref{R_1_R_2_R_3} and Lemma \ref{Lemma 4.1}). Moreover, because Talenti-type optimal functions do not exist for Choquard equations in the quasilinear setting (contrary to the case $p=2$), tracking the concentration of minimizing sequences to ensure they remain strictly below the threshold of loss of compactness becomes exceptionally intricate. Consequently, establishing the necessary localized energy $\epsilon$-estimates demands navigating the strong nonlinearity of the critical Sobolev constant $\mathbb{S}$ (see \eqref{S_alpha,p}) against the complex backdrop of non-local quasilinear interactions.}
	
	\noindent	Let us formally initiate our study by discussing the variational framework of \eqref{prob}. Denoting
	$$S(\tau):=\{u\in W^{1,p}(\mathbb{R}^N): \left\| u \right\|_p=\tau\},$$
	we define the weak solution as follows:
	\begin{definition}\label{weak_sol}
		A function $u\in S(\tau)$ is said to be a weak solution to \eqref{prob} if it satisfies the following:
		\begin{equation}\label{sol}
			\int_{\mathbb{R}^N}|\nabla u|^{p-2}\nabla u \nabla v +\ll u,v \gg_{s,p} = \lambda\int_{\mathbb{R}^N}|u|^{p-2}uv+ \mu\int_{\mathbb{R}^N}|u|^{q-2}uv +\int_{\mathbb{R}^N}(I_{\alpha}*|u|^{p^*_{\alpha}})|u|^{p^*_{\alpha}-2}uv, 
		\end{equation}
		for all $v\in W^{1,p}(\mathbb{R}^N)$. Here, we denote 
		\begin{equation*}
			\ll u, v \gg_{s,p} := \int_{\mathbb{R}^N}\int_{\mathbb{R}^N}\frac{|u(x)-u(y)|^{p-2}(u(x)-u(y))(v(x)-v(y))}{|x-y|^{N+sp}}dxdy,
		\end{equation*}
		and the space $W^{1,p}(\mathbb{R}^N)$ is equipped with the Banach norm:
		\begin{equation*}
			\left\| u \right\| = \left(T(u)^p+\left\| u \right\|_p^p\right)^{\frac{1}{p}} \text{ where }T(u)^p=\left\| \nabla u \right\|_p^p+[u]_{s,p}^p, \text{ with } [u]_{s,p}^p=\ll u, u \gg_{s,p}.
		\end{equation*}
	\end{definition}
	\noindent Starting with the regularity of weak solution to \eqref{prob}, our first result is the following:
	\begin{mytheorem}\label{Regularity}
		Suppose $u\in W^{1,p}(\mathbb{R}^N)$ is a weak solution of \eqref{prob}, then $u\in C^{\delta}_{loc}(\mathbb{R}^N)$ for all $0<\delta<\Theta$, where
		$\Theta =\min\left\{1,\frac{sp}{p-1}\right\}>s.$
	\end{mytheorem}
	\noindent Thus, calling the work of Anthal and Garain \cite{Anthal2025Pohozaev} and using the above regularity, it can be established that any weak solution must satisfy the following Pohozaev type identity:
	\begin{equation}\label{Pohozaev_identity}
		\left(\frac{N-p}{p}\right)\left\| \nabla u \right\|_p^p+\left(\frac{N-sp}{p}\right)[u]_{s,p}^p=\frac{\lambda N}{p}\left\| u \right\|_p^p+\frac{\mu N}{q}\left\| u \right\|_q^q +\left(\frac{N-p}{p}\right)A(u),
	\end{equation}
	where $A=\mathcal{A}_{p^*_{\alpha}}$. Equation \eqref{Pohozaev_identity} ensures the non-existence of any solution to \eqref{prob} corresponding to $\lambda<0$ and $q=p^*=\frac{Np}{N-p}$ demonstrated as follows:
	\begin{mytheorem}\label{non_existence}
		Suppose, $q=p^*=\frac{Np}{N-p}$ and $\lambda<0$ then \eqref{prob} does not exhibit any weak solution.
	\end{mytheorem}
	\begin{proof}
		Suppose, \eqref{prob} has a solution, then 
		$$\left\| \nabla u \right\|_p^p +[u]_{s,p}^p = \lambda \tau^p+\mu \left\| u \right\|_q^q +A(u),$$
		and by Pohozaev identity,
		$$\left(\frac{N-p}{p}\right)\left\| \nabla u \right\|_p^p+\left(\frac{N-sp}{p}\right)[u]_{s,p}^p= \frac{N\lambda\tau^p}{p}+\mu \frac{N}{q}\left\| u \right\|_q^q +\left(\frac{N-p}{p}\right)A(u).$$
		By, above two equations, for $q=p^*$ and $\lambda<0$, we get
		$$(1-s)[u]_{s,p}^p =\lambda \tau^p<0.$$
		Hence, by contradiction, \eqref{prob} cannot have solution, for $q=p^*$ and $\lambda<0$.
	\end{proof}
	\noindent  Next, using \eqref{Pohozaev_identity}, it can be seen that a solution to \eqref{prob} lies on the Pohozaev Manifold 
	\begin{equation}
		\mathcal{M}_{\tau}:=\{u\in S(\tau): M(u)=0\},
	\end{equation}
	\begin{equation*}
		\text{ where }M(u)= \left\| \nabla u \right\|_p^p+s[u]_{s,p}^p-\mu\gamma_{p,q}\left\| u \right\|_q^q-A(u) \text{ with }\gamma_{p,q}:=\frac{N(q-p)}{pq}.
	\end{equation*}
	Further using the fibre maps technique in section 3, we subdivided $\mathcal{M}_{\tau}$ into disjoint subsets
	$$\mathcal{M}_{\tau}^0:= \{u\in \mathcal{M}_{\tau} : p \left\| \nabla u \right\|_p^p+ps^2[u]_{s,p}^p=\mu q\gamma_{p,q}^2\left\| u \right\|_q^q+2.p^*_{\alpha}A(u)\},$$
	$$\mathcal{M}_{\tau}^+:= \{u\in \mathcal{M}_{\tau} : p \left\| \nabla u \right\|_p^p+ps^2[u]_{s,p}^p>\mu q\gamma_{p,q}^2\left\| u \right\|_q^q+2.p^*_{\alpha}A(u)\},$$			$$\mathcal{M}_{\tau}^-:= \{u\in \mathcal{M}_{\tau} : p \left\| \nabla u \right\|_p^p+ps^2[u]_{s,p}^p<\mu q\gamma_{p,q}^2\left\| u \right\|_q^q+2.p^*_{\alpha}A(u)\},$$
	and looked for distinct solutions in these disjoint subsets. \\
	\noindent Thanks to symmetric decreasing rearrangement, the Gagliardo-Nirenberg inequality \eqref{G_N_inequality}
	and compact imbedding $W^{1,p}_r(\mathbb{R}^N)\hookrightarrow L^t(\mathbb{R}^N)$ for all $t\in (p,p^*)$, 
	by the Ekeland variational principle, we deduce the existence of a first solution. Precisely, taking 
	$$\tau_1:=\left(\frac{q(2p^*_{\alpha}-p)(2p^*_{\alpha}
		\mathbb{S}^{\frac{2p^*_{\alpha}}{p}}(p-q\gamma_{p,q}))^{\frac{p-q\gamma_{p,q}}{2p^*_{\alpha}-p}}}{\mu C_{N,p,q}(p(2p^*_{\alpha}-q\gamma_{p,q}))^{\frac{2p^*_{\alpha}-q\gamma_{p,q}}{2p^*_{\alpha}-p}}}\right)^{\frac{1}{q(1-\gamma_{p,q})}}$$
	and 
	$$\tau_2:=\left(\frac{(2p^*_{\alpha}-p)}{\mu C_{N,p,q}(2p^*_{\alpha}-q\gamma_{p,q})\gamma_{p,q}^{\frac{q\gamma_{p,q}}{p}}}\left(\frac{q\mathbb{S}^{\frac{2p^*_{\alpha}}{2p^*_{\alpha}-p}}}{p-q\gamma_{p,q}}\right)^{\frac{p-q\gamma_{p,q}}{p}}\right)^{\frac{1}{q(1-\gamma_{p,q})}}$$
	we have the following:
	\begin{mytheorem}\label{Theorem 1}
		For $N\geq 3$, $s\in (0,1)$, $p<q<p+\frac{p^2s}{N}$ and $0<\tau<\min\{\tau_1,\tau_2\}$, there exists a radially decreasing positive function $u_{\tau}^+\in W^{1,p}(\mathbb{R}^N)$ that attains $m_{\tau}^+:=\displaystyle \inf_{u\in \mathcal{M}_{\tau}^+}E(u)$, that is, $E(u_{\tau}^+)=m_{\tau}^+<0$. Moreover, $u_{\tau}^+$ solves \eqref{prob} corresponding to some $\lambda=\lambda_{\tau}^+<0$, for sufficiently large $\mu>0$.
	\end{mytheorem}
	\noindent Moreover, if $p,N$ and $s$ satisfies the following:
	\begin{equation}\label{Conditions_p,N,s}
		\left\{		\begin{array}{cl}
			\text{ for } p^2>N, & \text{ either }\frac{N-p}{p(p-1)}<p(1-s)<\frac{N-p}{p-1} \text{ or } p(1-s) \geq \frac{N-p}{p-1},\\
			\text{ for } p^2 \leq N, & N<\min\{\frac{p^3+p}{2}, p^2(1-s)(p-1)+p\}, 
		\end{array}
		\right.
	\end{equation}
	then, we prove the existence of a second solution to \eqref{prob}. Precisely, we have the following:
	\begin{mytheorem}\label{Theorem 2}
		Let
		$$\tau_3=\left(\frac{(2p^*_{\alpha}-p)}{\mu p(2p^*_{\alpha}-q\gamma_{p,q})C_{N,p,q}}\right)^{\frac{1}{q(1-\gamma_{p,q})}}.$$
		For, $2\leq p <N$ and $s\in(0,1)$, assume $N,p$ and $s$ satisfies \eqref{Conditions_p,N,s}, then for all $\tau<\min\{\tau_1,\tau_2,\tau_3\}$, \eqref{prob} admits a second solution $(u_{\tau}^-,\lambda_{\tau}^-)\in W^{1,p}(\mathbb{R}^N)\times \mathbb{R}$, with $u_{\tau}^-\in \mathcal{M}_{\tau}^-$.
	\end{mytheorem}
	\noindent 
	Denoting $S$ be the best constant corresponding to the imbedding $W^{1,p}(\mathbb{R}^N)\hookrightarrow L^{p^*}(\mathbb{R}^N)$, it is well known that:
	\begin{equation}\label{S}
		S=\inf_{W^{1,p}(\mathbb{R}^N)\setminus \{0\}}\frac{\left\| \nabla v \right\|_p^p}{\left\| u \right\|_{p^*}^p}.
	\end{equation}
	Furthermore, it is achieved by the family of extremal functions:
	\begin{equation}\label{U_epsilon}
		U_{\epsilon}(x)=\frac{K_{N,p}\epsilon^{\frac{(N-p)}{p(p-1)}}}{(\epsilon^{\frac{p}{p-1}}+|x|^{\frac{p}{p-1}})^{\frac{N-p}{p}}} \text{ with } \epsilon>0\text{ and } K_{N,p}=\left(N\left(\frac{N-p}{p-1}\right)^{p-1}\right)^{\frac{(N-p)}{p^2}}.
	\end{equation}
	Now, let us define:
	\begin{equation}\label{S_alpha,p}
		\mathbb{S}:=\inf_{W^{1,p}(\mathbb{R}^N)\setminus \{0\}}\frac{\left\| \nabla u \right\|_p^p}{A(u)^{\frac{p}{2p^*_{\alpha}}}}.
	\end{equation}
	By \autoref{prop1} we have:
	\begin{eqnarray*}
		A(u) & = & \int_{\mathbb{R}^N}\int_{\mathbb{R}^N}\frac{A_{N,\alpha}|u(x)|^{p^*_{\alpha}}|u(y)|^{p^*_{\alpha}}}{|x-y|^{N-\alpha}}
		\leq  A_{N,\alpha}C_{N,\alpha,p} 
		\left(\int_{\mathbb{R}^N}|u|^{p^*_{\alpha}\frac{2N}{N+\alpha}}\right)^{\frac{N+\alpha}{N}}\\
		& = & A_{N,\alpha}C_{N,\alpha,p}\left(\left\| u \right\|_{p^*}^{p^*}\right)^{\frac{N+\alpha}{N}}=A_{N,\alpha}C_{N,\alpha,p}\left(\left\| u \right\|_{p^*}^p\right)^{\frac{N+\alpha}{N-p}}\leq A_{N,\alpha}C_{N,\alpha,p}\left(\frac{\left\| \nabla u \right\|_p^p}{S}\right)^{\frac{2p^*_{\alpha}}{p}}.
	\end{eqnarray*}
	Thus,
	$$\frac{\left\| \nabla u \right\|_p^{2p^*_{\alpha}}}{A(u)}\geq \frac{S^{\frac{2p^*_{\alpha}}{p}}}{A_{N,\alpha}C_{N,\alpha,p}} \text{ for all } u\in W^{1,p}(\mathbb{R}^N)\setminus \{0\}.$$
	Hence, one gets
	\begin{equation}\label{S_and_S_alpha,p}
		\mathbb{S} \geq \frac{S}{(A_{N,\alpha}C_{N,\alpha,p})^{\frac{p}{2p^*_{\alpha}}}}.
	\end{equation}
	The relation \eqref{S_and_S_alpha,p} played an important role to prescribe the energy levels of Palais-Smale sequences and then to establish the existence of a second solution to \eqref{prob}.
	
	\noindent Another important result that was frequently used in our analysis, is the following Gagliardo-Nirenberg inequality:
	\begin{proposition}
		For any $u\in W^{1,p}(\mathbb{R}^N)$, we have:
		\begin{equation}\label{G_N_inequality}
			\left\| u \right\|_{\beta} \leq C_{N,p,\beta} \left\| \nabla u \right\|_p^{\theta} \left\| u \right\|_p^{1-\theta} \text{ for all } p\leq \beta \leq p^*,
		\end{equation}
		with $\theta=\frac{N(\beta-p)}{\beta p}$.
	\end{proposition}
	\begin{proof}
		Clearly, \eqref{G_N_inequality} is trivial for $\beta=p$, and the case of $\beta=p^*$ is actually the Sobolev inequality. Now, for any $\beta\in(p,p^*)$, since we can find $t\in (0,1)$ such that $\beta= tp^*+(1-t)p$, by H$\ddot{\text{o}}$lder's inequality and \eqref{S} we have:
		\begin{eqnarray*}
			\int_{\mathbb{R}^N} |u|^{\beta} & = & \int_{\mathbb{R}^N}|u|^{tp^*}|u|^{(1-t)p} \leq \left(\int_{\mathbb{R}^N}|u|^{p^*}\right)^{t}\left(\int_{\mathbb{R}^N}|u|^p\right)^{1-t}= \left\| u \right\|_{p^*}^{tp^*}\left\| u \right\|_p^{(1-t)p}\\
			& \leq & \frac{\left\| \nabla u \right\|_p^{tp^*}\left\| u \right\|_p^{(1-t)p}}{S^{\frac{tp^*}{p}}},
		\end{eqnarray*}
		taking $\theta=\frac{tp^*}{\beta}$, we get:
		\begin{eqnarray*}
			\left\| u \right\|_{\beta} & = & \left(\int_{\mathbb{R}^N}|u|^{\beta}\right)^{\frac{1}{\beta}} \leq \frac{\left\| \nabla u \right\|_p^{\frac{tp^*}{\beta}}\left\| u \right\|_p^{\frac{(1-t)p}{\beta}}}{S^{\frac{tp^*}{p\beta}}} = C_{N,p,\beta} \left\| \nabla u \right\|_p^{\theta}\left\| u \right\|_p^{1-\theta}.
		\end{eqnarray*}
		Thus we get \eqref{G_N_inequality}.
	\end{proof}
	\noindent \textbf{Scheme of the paper:} In section 2, we discuss the regularity of the solution, essential to establish the Pohozaev identity. In order to get this, we first constructed an iterative scheme as done in \cite{Nidhi2025Normlaized_MMA} to prove that a radially symmetric solution lies in $L^r(\mathbb{R}^N)$ for all $r\geq 1$. This helped us to attain the hypothesis of \cite[Theorem~1.4]{Garain2023Higher}, which gives us the required H$\ddot{\text{o}}$lder regularity. Having established regularity and hence the Pohozaev identity, we saw the non-existence of a solution in the scenario of $q=p^*$ and $\lambda<0$, it has been explained in the introduction section itself. Further, we looked for several existence results in the subsequent sections. Section 3 consists of the necessary groundwork for our main results, which involve the construction of the Pohozaev manifold, its distinct subsets ($\mathcal{M}_{\tau}^+,\mathcal{M}_{\tau}^-$ and $\mathcal{M}_{\tau}^0$) and the study of the infimum and supremum of the energy functional over these subsets. With all the necessary groundwork in our hand, in section 4, the first solution is established by proving the convergence of the Palais-Smale sequence up to a subsequence. Further, to prove the existence of a second solution, in section 5, a relation between $m_{\tau}^+$ and $m_{\tau}^-$ is deduced using \eqref{S_and_S_alpha,p} and the estimates of the Talenti function corresponding to $S$. It gave us the minimiser of the energy functional on $\mathcal{M}_{\tau}^-$ and hence the second solution. 
	A similar approach to get multiple solutions can be seen in \cite{Nidhi2025Existence,Sun2024Note}, but they all discuss the semilinear case. In the present paper, we are dealing with a more generalised quasilinear version involving  a combination of nonlinear operators as the classical $p$-Laplacian and fractional $p$-Laplacian operators, which makes our problem more technically complex and interesting to study. Note that contrary to the semilinear case, we do not know the extremal functions associated with $\mathbb{S}$, which creates significant difficulties, in particular in estimating critical energy levels that we overcome by an accurate fibering map analysis.
	
	\section{Regularity Results}
	\begin{proposition}\label{L_r_regularity}
		Suppose $u\in W^{1,p}(\mathbb{R}^N)$ be a radially symmetric solution of \eqref{prob} corresponding to some $\lambda<0$, then $u\in L^r(\mathbb{R}^N)$ for all $r\in [1,\infty)$.
	\end{proposition}
	\begin{proof}
		For $\epsilon>0$, define function $$h_{\epsilon}(t):=\sqrt{\epsilon^2+t^2}\text{ for } t\in \mathbb{R},$$
		and $$g_{\epsilon}:= h'_{\epsilon}(t)=\frac{t}{\sqrt{\epsilon^2+t^2}}.$$
		Clearly, $g_{\epsilon}\in C^1(\mathbb{R})$, with $g_{\epsilon}(0)=0$ and $|g_{\epsilon}(t)|\leq \frac{1}{\epsilon}=M_{\epsilon}$ for all $t\in \mathbb{R}$. Thus, by \cite[Theorem~2.2.3]{Kesavan2019Topics}, $g_{\epsilon}(u)\in W^{1,p}(\mathbb{R}^N)$ and hence $\psi=\phi|g_{\epsilon}(u)|^{p-2}g_{\epsilon}(u)\in W^{1,p}(\mathbb{R}^N)$ for any $0<\phi\in C_{c}^{\infty}(\mathbb{R}^N)$. Taking $\psi$ as test function we get
		\begin{equation}\label{R_2.1}
			\int_{\mathbb{R}^N}|\nabla u |^{p-2}\nabla u \nabla \psi +\ll u,\psi \gg = \lambda \int_{\mathbb{R}^N}|u|^{p-2}u \psi
			+\mu\int_{\mathbb{R}^N}|u|^{q-2}u\psi +\int_{\mathbb{R}^N}(I_{\alpha}*|u|^{p^*_{\alpha}})|u|^{p^*_{\alpha}-2}u\psi.
		\end{equation}
		Denoting
		\begin{eqnarray*}
			I_1^{\epsilon}& := &\int_{\mathbb{R}^N}|\nabla u|^{p-2}\nabla u \nabla \psi \\
			& = & (p-1)\int_{\mathbb{R}^N}|\nabla u|^p|g_{\epsilon}(u)|^{p-2}g_{\epsilon}'(u)+\int_{\mathbb{R}^N}g_{\epsilon}(u)|g_{\epsilon}(u)|^{p-2}|\nabla u|^{p-2}\nabla u \nabla \phi,
		\end{eqnarray*}
		and 
		\begin{eqnarray*}
			I_2^{\epsilon}& := & \int_{\mathbb{R}^N}\int_{\mathbb{R}^N}\frac{|h_{\epsilon}(u(x))-h_{\epsilon}(u(y))|^{p-2}(h_{\epsilon}(u(x))-h_{\epsilon}(u(y)))(\phi(x)-\phi(y))}{|x-y|^{N+sp}}dxdy,
		\end{eqnarray*}
		then we have
		\begin{eqnarray*}
			I_2^{\epsilon}	& \leq & \int_{\mathbb{R}^N}\int_{\mathbb{R}^N}\frac{|u(x)-u(y)|^{p-2}(u(x)-u(y))(\phi(x)|g_{\epsilon}(u)(x)|^{p-2}g_{\epsilon}(u)(x)-\phi(y)|g_{\epsilon}(u)(y)|^{p-2}g_{\epsilon}(u)(y))}{|x-y|^{N+sp}}\\
			& = & \ll u, \psi \gg,
		\end{eqnarray*}
		by convexity of $h_{\epsilon}$ (see \cite[Lemma~A.1]{Brasco2016second}). Thus, by \eqref{R_2.1} we get
		\begin{eqnarray*}
			I_1^{\epsilon}+I_2^{\epsilon} + (-\lambda)\int_{\mathbb{R}^N}|u|^{p-2}u\psi & \leq & \mu \int_{\mathbb{R}^N}|u|^{q-2}u\psi +\int_{\mathbb{R}^N}(I_{\alpha}*|u|^{p^*_{\alpha}})|u|^{p^*_{\alpha}-2}u\psi\\
			& \leq & \mu \int_{\mathbb{R}^N}|u|^{q-1}\phi +\int_{\mathbb{R}^N}(I_{\alpha}*|u|^{p^*_{\alpha}})|u|^{p^*_{\alpha}-1}\phi,
		\end{eqnarray*}
		since $g_{\epsilon}\leq 1$.
		Further, since $h_{\epsilon}(t)\rightarrow |t|$, $g_{\epsilon}(t)\rightarrow Sgn(t)$ and $g'_{\epsilon}(t)\rightarrow 0$ as $\epsilon\rightarrow 0$, by Fatou's lemma we get
		\begin{eqnarray}\label{R_2.2}
			&& \int_{\mathbb{R}^N}|\nabla u|^{p-2}\nabla |u| \nabla \phi + \ll |u|, \phi \gg + (-\lambda) \int_{\mathbb{R}^N} |u|^{p-1} \phi \nonumber \\
			&& \leq \liminf_{\epsilon \rightarrow 0}\left(I_1^{\epsilon}+I_2^{\epsilon} + (-\lambda)\int_{\mathbb{R}^N}|u|^{p-2}u\psi\right)\nonumber\\
			&& \leq \mu \int_{\mathbb{R}^N}|u|^{q-1}\phi +\int_{\mathbb{R}^N}(I_{\alpha}*|u|^{p^*_{\alpha}})|u|^{p^*_{\alpha}-1}\phi,
		\end{eqnarray}
		for all $0<\phi \in C_c^{\infty}(\mathbb{R}^N)$, and hence by density, \eqref{R_2.2} holds for all $0<\phi \in W^{1,p}(\mathbb{R}^N)$. Next, for $\gamma>0$, we define
		$$u_\gamma(x):=\min\{\gamma, |u(x)|\}>0,$$
		and take $\phi = u_{\gamma}^{\beta}\in W^{1,p}(\mathbb{R}^N)$ where $\beta=kp-p+1\geq 1$ for some $k\geq 1$, in \eqref{R_2.2}. Thus, using \cite[Lemma~3.1]{Biswas2023Regularity} we get
		\begin{eqnarray}\label{R_2.3}
			&& \frac{\beta p^p}{(\beta+p-1)^{p}}\left(\left\| \nabla u_{\gamma}^k\right\|_p^p+ [u_{\gamma}^k]_{s,p}^p\right)+(-\lambda) \left\| u_{\gamma}^k \right\|_p^p\nonumber\\
			&& \leq \beta \int_{\{|u(x)|\leq \gamma\}} | u_{\gamma} |^{\beta-1}|\nabla u_{\gamma}|^p +\ll |u|, u_{\gamma}^{\beta} \gg+(-\lambda) \int_{\mathbb{R}^N}|u|^{p-1}u_{\gamma}^{\beta}\nonumber\\
			&& = \int_{\mathbb{R}^N} \nabla u_{\gamma}^{\beta}\nabla |u| |\nabla u|^{p-2}+\ll |u|, u_{\gamma}^{\beta} \gg+(-\lambda) \int_{\mathbb{R}^N}|u|^{p-1}u_{\gamma}^{\beta},
		\end{eqnarray} 
		since \begin{eqnarray*}
			\beta \int_{\{|u(x)|\leq \gamma\}} | u_{\gamma} |^{\beta-1}|\nabla u_{\gamma}|^p & = & \int_{\{|u(x)|\leq \gamma\}}\beta u_{\gamma}^{\beta-1}|\nabla |u||^{p-2}\nabla |u| \nabla |u|\\
			& = & \beta \int_{\mathbb{R}^N}|\nabla |u||^{p-2} u_{\gamma}^{\beta-1} \nabla u_{\gamma}\nabla |u|\\
			& = & \int_{\mathbb{R}^N} |\nabla |u||^{p-2}\nabla u_{\gamma}^{\beta}\nabla |u|.
		\end{eqnarray*}
		Now, by \eqref{R_2.3} and \eqref{R_2.2}, we get
		\begin{eqnarray}\label{R_2.4}
			&& \frac{\beta}{k^p}\left(\left\| \nabla u_{\gamma}^k\right\|_p^p+ [u_{\gamma}^k]_{s,p}^p\right)+(-\lambda) \left\| u_{\gamma}^k \right\|_p^p\nonumber\\
			&&=\frac{\beta p^p}{(\beta+p-1)^{p}}\left(\left\| \nabla u_{\gamma}^k\right\|_p^p+ [u_{\gamma}^k]_{s,p}^p\right)+(-\lambda) \left\| u_{\gamma}^k \right\|_p^p\nonumber\\
			& & \leq   \mu \int_{\mathbb{R}^N}|u|^{q-1}u_{\gamma}^{\beta} +\int_{\mathbb{R}^N}(I_{\alpha}*|u|^{p^*_{\alpha}})|u|^{p^*_{\alpha}-1}u_{\gamma}^{\beta}.
		\end{eqnarray}
		Now, for some fixed $\delta>1$ using \autoref{prop1} and H$\ddot{\text{o}}$lder's inequality, we get
		\begin{eqnarray}\label{R_2.5}
			&&\int_{\mathbb R^N}(I_{\alpha}*|u|^{p^*_{\alpha}})|u|^{p^*_{\alpha}-1}u_{\gamma}^{\beta} 
			= \int_{\mathbb{R}^N}\frac{A_{N,\alpha}|u(x)|^{p^*_{\alpha}}|u(y)|^{p^*_{\alpha}-1}u_{\gamma}^{\beta}(y)}{|x-y|^{N-\alpha}}dxdy\nonumber\\
			&& \leq  C_1\left(\int_{\mathbb{R}^N}\left(|u(x)|^{p^*_{\alpha}-2}|u(x)u_{\gamma}^{\beta}(x)|\right)^{\frac{2N}{N+\alpha}}\right)^{\frac{N+\alpha}{N}}\nonumber\\
			&& = C_1 \left(\int_{\{|u(x)|<\delta\}}\left(|u(x)|^{p^*_{\alpha}-2}|u(x)u_{\gamma}^{\beta}(x)|\right)^{\frac{2N}{N+\alpha}}\right.\nonumber\\
			&& \left. +\int_{\{|u(x)|\geq \delta\}}\left(|u(x)|^{p^*_{\alpha}-2}|u(x)u_{\gamma}^{\beta}(x)|\right)^{\frac{2N}{N+\alpha}}\right)^{\frac{N+\alpha}{2N}}\nonumber\\
			&& \leq  C_2\left(\left(\int_{\{|u(x)|<\delta\}}\left(|u(x)|^{p^*_{\alpha}-2}|u(x)u_{\gamma}^{\beta}(x)|\right)^{\frac{2N}{N+\alpha}}\right)^{\frac{N+\alpha}{2N}}\right.\nonumber\\
			&&\left.+\left(\int_{\{|u(x)|\geq\delta\}}\left(|u(x)|^{p^*_{\alpha}-2}|u(x)u_{\gamma}^{\beta}(x)|\right)^{\frac{2N}{N+\alpha}}\right)^{\frac{N+\alpha}{2N}}\right)\nonumber\\
			& & \leq C_2\left(\left(\int_{\{|u(x)|<\delta\}}\left(|u|^{p^*_{\alpha}+\beta-1}\right)^{\frac{2N}{N+\alpha}}\right)^{\frac{N+\alpha}{2N}}+\left(\int_{\{|u(x)|\geq\delta\}}\left(|u|^{p^*_{\alpha}+\beta-1}\right)^{\frac{2N}{N+\alpha}}\right)^{\frac{N+\alpha}{2N}}\right)\nonumber\\
			&& \leq C_2\left(\delta^{p^*_{\alpha}-p}\left(\int_{\{|u(x)|<\delta\}}\left(|u|^{p+\beta-1}\right)^{\frac{2N}{N+\alpha}}\right)^{\frac{N+\alpha}{2N}}\right.\nonumber\\
			&&\left.+\left(\int_{\{|u(x)|\geq\delta\}}\left(|u|^{p^*_{\alpha}+\beta-1}\right)^{\frac{2N}{N+\alpha}}\right)^{\frac{N+\alpha}{2N}}\right)\nonumber\\
			&& =  C_2\left(\delta^{p^*_{\alpha}-p}\left(\int_{\{|u(x)|<\delta\}}\left(|u|^{kp}\right)^{\frac{2N}{N+\alpha}}\right)^{\frac{N+\alpha}{2N}}+\left(\int_{\{|u(x)|\geq\delta\}}\left(|u|^{p^*_{\alpha}-p+kp}\right)^{\frac{2N}{N+\alpha}}\right)^{\frac{N+\alpha}{2N}}\right)\nonumber\\
			&& \leq C_2 \left(\delta^{p^*_{\alpha}-p} \left\| u \right\|_{\left(\frac{2N}{N+\alpha}\right)kp}^{kp}+\left(\int_{\{|u(x)|\geq \delta\}}|u|^{p^*}\right)^{\frac{2p+\alpha-N}{2N}}\left(\int_{\{|u(x)|\geq \delta\}}|u|^{kp^*}\right)^{\frac{N-p}{N}}
			\right)\nonumber\\
			&& = C_2\delta^{p^*_{\alpha}-p} \left\| u \right\|_{\frac{2Nkp}{N+\alpha}}^{kp}+ C_2\bar{C}(\delta)^{\frac{p^*_{\alpha}-p}{p^*}}\left\| u \right\|_{kp^*}^{kp},
		\end{eqnarray}
		where $$\bar{C}(\delta)= \int_{\{|u(x)|\geq \delta\}}|u|^{p^*}.$$
		Also,
		\begin{eqnarray}\label{R_2.6}
			\int_{\mathbb{R}^N}|u|^{q-1}u_{\gamma}^\beta & = & \int_{\{|u(x)|<\delta\}} |u|^{q-1}u_{\gamma}^\beta+\int_{\{|u(x)|\geq \delta\}}|u|^{q-1}u_{\gamma}^\beta\nonumber\\
			& \leq & \delta^{q-p}\int_{\{|u(x)|<\delta\}}|u|^{p+\beta-1} + \int_{\{|u(x)|\geq \delta\}} |u|^{p^*-1+\beta}\nonumber\\
			& = & \delta^{q-p}\int_{\{|u(x)|<\delta\}}|u|^{kp} +\int_{\{|u(x)|\geq \delta\}}|u|^{p^*+kp-p} \nonumber\\
			& \leq & \delta^{q-p} \left\| u \right\|_{kp}^{kp}+ \left(\int_{\{|u(x)|\geq \delta\}}|u|^{kp^*}\right)^{\frac{p}{p^*}}\left(\int_{\{|u(x)|\geq \delta\}}|u|^{p^*}\right)^{\frac{p^*-p}{p^*}}\nonumber\\
			& \leq & \delta^{q-p}\left\| u \right\|_{kp}^{kp} +\bar{C}(\delta)^{\frac{p^*-p}{p^*}} \left\| u \right\|_{kp^*}^{kp}.
		\end{eqnarray}
		Using \eqref{R_2.5} and \eqref{R_2.6} in \eqref{R_2.4} we get
		\begin{eqnarray*}
			&&\frac{\beta}{k^p}\left(\left\| \nabla u_{\gamma}^k\right\|_p^p+ [u_{\gamma}^k]_{s,p}^p\right)+(-\lambda) \left\| u_{\gamma}^k \right\|_p^p	\\
			&& \leq \mu\left(\delta^{q-p}\left\| u \right\|_{kp}^{kp} +\bar{C}(\delta)^{\frac{p^*-p}{p^*}} \left\| u \right\|_{kp^*}^{kp}\right)+ C_2\delta^{p^*_{\alpha}-p} \left\| u \right\|_{\frac{2Nkp}{N+\alpha}}^{kp}+ C_2\bar{C}(\delta)^{\frac{p^*_{\alpha}-p}{p^*}}\left\| u \right\|_{kp^*}^{kp}.
		\end{eqnarray*}
		Taking $C_3=\min\{1,\frac{-\lambda k^p}{\beta}\}$, we get
		\begin{equation*}
			\frac{C_3 \beta}{k^p} \left\| u_{\gamma}^k\right\|^p   \leq  \left( \mu \bar{C}(\delta)^{\frac{p^*-p}{p^*}}+\bar{C}(\delta)^{\frac{p^*_{\alpha}-p}{p^*}}\right)\left\| u \right\|_{kp^*}^{kp}+\mu \delta^{q-p}\left\| u \right\|_{kp}^{kp}+C_2\delta^{p^*_{\alpha}-p} \left\| u \right\|_{\frac{2Nkp}{N+\alpha}}^{kp}.
		\end{equation*}
		Hence, by the continuous imbedding $W^{1,p}(\mathbb{R}^N)\hookrightarrow L^{p^*}(\mathbb{R}^N)$ and Fatou's lemma
		\begin{eqnarray*}
			\left\| u \right\|_{kp^*}^{kp} & = & \left\| u^k \right\|_{p^*}^{p}\leq \liminf_{\gamma \rightarrow \infty} \left\| u_{\gamma}^k\right\|_{p^*}^{p}
			\leq \liminf_{\gamma \rightarrow \infty} C_4 \left\| u_{\gamma}^k \right\|^p \\
			& \leq & \frac{C_4k^p}{C_3 \beta}\left( \left( \mu \bar{C}(\delta)^{\frac{p^*-p}{p^*}}+\bar{C}(\delta)^{\frac{p^*_{\alpha}-p}{p^*}}\right)\left\| u \right\|_{kp^*}^{kp}+\mu \delta^{q-p}\left\| u \right\|_{kp}^{kp}+C_2\delta^{p^*_{\alpha}-p} \left\| u \right\|_{\frac{2Nkp}{N+\alpha}}^{kp}\right)\\
			& \leq & C_5k^{p-1}\left( \left( \mu \bar{C}(\delta)^{\frac{p^*-p}{p^*}}+\bar{C}(\delta)^{\frac{p^*_{\alpha}-p}{p^*}}\right)\left\| u \right\|_{kp^*}^{kp}+\mu \delta^{q-p}\left\| u \right\|_{kp}^{kp}+C_2\delta^{p^*_{\alpha}-p} \left\| u \right\|_{\frac{2Nkp}{N+\alpha}}^{kp}\right).
		\end{eqnarray*}
		Choosing $\delta>1$ large enough so that $ C':=\left( \mu \bar{C}(\delta)^{\frac{p^*-p}{p^*}}+\bar{C}(\delta)^{\frac{p^*_{\alpha}-p}{p^*}}\right)C_5k^{p-1}<1$, we get
		\begin{eqnarray*}
			(1-C')\left\| u \right\|_{kp^*}^{kp} & \leq & C_5 k^{p-1}\left(\mu \delta^{q-p}\left\| u \right\|_{kp}^{kp}+C_2\delta^{p^*_{\alpha}-p} \left\| u \right\|_{\frac{2Nkp}{N+\alpha}}^{kp}\right)\\
			& \leq & \hat{C}k^{p-1}\delta^{p^*_{\alpha}-p}\left( \left\| u \right\|_{kp}^{kp}+\left\| u \right\|_{\frac{2Nkp}{N+\alpha}}^{kp}\right),	
		\end{eqnarray*}
		where $\hat{C}=C_5\max\{C_2,\mu\}$.
		Therefore
		\begin{equation}\label{R_2.7}
			\left\| u \right\|_{kp^*} \leq \tilde{C}k^{\frac{p-1}{kp}}\left(\left\| u \right\|_{\frac{2Nkp}{N+\alpha}}^{kp}+\left\| u \right\|_{kp}^{kp}\right)^{\frac{1}{kp}}
		\end{equation}
		with $\tilde{C}=\left(\frac{\hat{C}\delta^{(p^*_{\alpha}-p)}}{1-C'}\right)$. Define the sequence $\{r_n\}$ as follows:
		$$r_n=\left(\frac{N+\alpha}{2(N-p)}\right)^n>1 \text{ for all } n\in \mathbb{N}.$$
		For $N-\alpha<2p$, clearly the sequence $r_n\rightarrow \infty$ as $n\rightarrow \infty$. Now, since $u\in W^{1,p}(\mathbb{R}^N)\hookrightarrow L^r(\mathbb{R}^N)$ for all $r\in [p,p^*]$, we get
		$$\tilde{C}k^{\frac{p-1}{kp}}\left(\left\| u \right\|_{\frac{2Nkp}{N+\alpha}}^{kp}+\left\| u \right\|_{kp}^{kp}\right)^{\frac{1}{kp}}<+\infty \text{ for all } k\in [1,r_1],$$
		and hence by \eqref{R_2.7}
		\begin{equation}\label{Iteration_1}
			u\in L^r(\mathbb{R}^N) \text{ for all } r \in [p,p^*]\cup[p^*,r_1p^*]=[p,r_1p^*].
		\end{equation}
		Now, by \eqref{Iteration_1}
		$$\tilde{C}k^{\frac{p-1}{kp}}\left(\left\| u \right\|_{\frac{2Nkp}{N+\alpha}}^{kp}+\left\| u \right\|_{kp}^{kp}\right)^{\frac{1}{kp}}<+\infty \text{ for all } k\in [1,r_2],$$
		and hence by \eqref{R_2.7}
		\begin{equation*}
			u \in L^r(\mathbb{R}^N) \text{ for all } r\in [p,p^*]\cup [p^*,r_2p^*]=[p,r_2p^*].
		\end{equation*}
		Moving on in this way, we get $u\in L^r(\mathbb{R}^N)$ for all $r\in [p,r_np^*]$ and $n\in \mathbb{N}$. Hence, since $r_n\rightarrow \infty$, we get $u\in L^r(\mathbb{R}^N)$ for all $r\geq p$. Moreover, since $u\in L^r_{loc}(\mathbb{R}^N)$ for all $r\geq 1$ and it is radially symmetric, by radial Lemma \cite{Yuan2013radial} it must lie in $L^r(\mathbb{R}^N)$ for all $r\in [1,\infty)$.
	\end{proof}
	\noindent Further, the above proposition, together with \cite[Theorem~1.4]{Garain2023Higher} gives us the following H$\ddot{\text{o}}$lder regularity.
	\begin{theorem}
		Suppose $N-\alpha < 2p$ and $u\in W^{1,p}(\mathbb{R}^N)$ is a weak solution of \eqref{prob}, then $u\in C^{\delta}_{loc}(\mathbb{R}^N)$ for all $0<\delta<\Theta$, where
		$\Theta =\min\left\{1,\frac{sp}{p-1}\right\}>s.$
	\end{theorem}
	\begin{proof}
		We know that, for all $x\in \mathbb{R}^N$ 
		\begin{equation*}
			(I_{\alpha}*|u|^{p^*_{\alpha}})(x) = \int_{\mathbb{R}^N}\frac{A_{N,\alpha}|u(x-y)|^{p^*_{\alpha}}}{|y|^{N-\alpha}}dy = \int_{B_1(0)}\frac{A_{N,\alpha}|u(x-y)|^{p^*_{\alpha}}}{|y|^{N-\alpha}}dy+\int_{\mathbb{R}^N\setminus B_1(0)}\frac{A_{N,\alpha}|u(x-y)|^{p^*_{\alpha}}}{|y|^{N-\alpha}}dy.
		\end{equation*}
		For some $\gamma>\max\{\frac{N}{\alpha}, \frac{p}{p^*_{\alpha}}\}$, using \autoref{L_r_regularity} and the H$\ddot{\text{o}}$lder's inequality, we get
		\begin{equation*}
			\int_{B_1(0)}\frac{A_{N,\alpha}|u(x-y)|^{p^*_{\alpha}}}{|y|^{N-\alpha}}dy  \leq  C_\alpha \left(\int_{B_1(0)}\frac{dy}{|y|^{\frac{(N-\alpha)\gamma}{\gamma-1}}}\right)^{\frac{\gamma-1}{\gamma}}<M_1.
		\end{equation*}
		We also have:
		\begin{eqnarray*}
			&&\int_{\mathbb{R}^N\setminus B_1(0)}\frac{A_{\alpha}|u(x-y)|^{p^*_{\alpha}}}{|y|^{N-\alpha}}dy  \\
			&&\leq  A_{\alpha} \left(\int_{\mathbb{R}^N\setminus B_1(0)}|u(x-y)|^{\frac{2Np^*_{\alpha}}{N+\alpha}}\right)^{\frac{N+\alpha}{2N}}\left(\int_{\mathbb{R}^N\setminus B_1(0)}\frac{dy}{|y|^{(N-\alpha)\frac{2N}{N-\alpha}}}\right)^{\frac{N-\alpha}{2N}} <M_2.
		\end{eqnarray*}
		Therefore, there exists $M>0$ such that
		$(I_{\alpha}*|u|^{p^*_{\alpha}})(x)<M$. Thus,
		\begin{equation}\label{I_alpha_L_infinity}
			(I_{\alpha}*|u|^{p^*_{\alpha}})\in L^{\infty}(\mathbb{R}^N).
		\end{equation}	
		Now define the function $f:=\lambda |u|^{p-1}u+\mu|u|^{q-2}u+(I_{\alpha}*|u|^{p^*_{\alpha}})|u|^{p^*_{\alpha}-2}u$, then by \autoref{L_r_regularity}
		\begin{eqnarray*}
			\int_{\mathbb{R}^N}|f|^r & \leq & |\lambda|\int_{\mathbb{R}^N}|u|^{r(p-1)}+\mu \int_{\mathbb{R}^N}|u|^{r(q-1)}+\int_{\mathbb{R}^N}|(I_{\alpha}*|u|^{p^*_{\alpha}})|u|^{p^*_{\alpha}-1}|^{r}\\
			& \leq & |\lambda|\int_{\mathbb{R}^N}|u|^{r(p-1)}+\mu \int_{\mathbb{R}^N}|u|^{r(q-1)} +M^r\int_{\mathbb{R}^N}|u|^{r(p^*_{\alpha}-1)}\\
			& < & +\infty, \text{ for all } r\geq \frac{p}{p-1},
		\end{eqnarray*}
		since $q,p^*_{\alpha}>p$. Therefore $f\in L^r(\mathbb{R}^N)$ for all $r\geq \frac{p}{p-1}$ and $f\in L^r_{loc}(\mathbb{R}^N)$ for all $r\geq \frac{1}{p-1}$.
		Also, since, $u\in W^{1,p}_{loc}(\mathbb{R}^N)$, in order to use the result of Garain et al \cite[Theorem~1.4]{Garain2023Higher}, we need $u\in L^{p-1}_{sp}(\mathbb{R}^N)$ where
		$$L_{\beta}^{r}(\mathbb{R}^N)=\left\{u\in L^r_{loc}(\mathbb{R}^N): \int_{\mathbb{R}^N}\frac{|u|^r}{1+|x|^{N+\beta}}dx <+\infty\right\}, \text{ for } r \text{ and } \beta>0.$$
		Now, since $u\in L^r(\mathbb{R}^N)$ for all $r\geq p$, we get $u\in L^r_{loc}(\mathbb{R}^N)$ for all $r\in [1,\infty)$, and subsequently $u\in L^{p-1}_{loc}(\mathbb{R}^N)$, as $p\geq 2$. Further, by H$\ddot{\text{o}}$lder's inequality, we have:
		\begin{eqnarray*}
			\int_{\mathbb{R}^N}\frac{|u|^{p-1}}{1+|x|^{N+sp}} & \leq & \left(\int_{\mathbb{R}^N}\frac{dx}{(1+|x|^{N+sp})^p}\right)^{\frac{1}{p}}\left(\int_{\mathbb{R}^N}|u|^p\right)^{\frac{p-1}{p}}\\
			& = & C_{N}\left(\int_{0}^{\infty}\frac{t^{N-1}}{(1+t^{N+sp})^p}dt\right)^{\frac{p-1}{p}}\\
			& \leq & C_{N}\left( \int_{0}^{1}t^{N-1}dt+\int_{1}^{\infty}(1+t^{N+sp})^{\frac{N-1}{N+sp}-p}dt\right)^{\frac{p-1}{p}}\\
			& \leq & C_N \left( \frac{1}{N}+ \int_{1}^{\infty}\frac{dt}{(t^{N+sp})^{p-\frac{N-1}{N+sp}}}\right)^{\frac{p-1}{p}}<+\infty,
		\end{eqnarray*}
		since $Np+sp^2-N+1>1$. Hence $u\in C^{\delta}_{loc}(\mathbb{R}^N)$ for all 
		$$0<\delta<\min\left\{\frac{(p-\frac{N}{r_0})}{p-1},\frac{sp}{p-1},1\right\},$$
		here $r_0$ is such that $r_0>\frac{N}{p}$ and $f\in L^{r_0}_{loc}(\mathbb{R}^N)$. Taking $r_0=\max\{N,\frac{1}{p-1}\}+1>\frac{N}{p}$, we get
		$$\frac{(p-\frac{N}{r_0})}{p-1}>1.$$
		Thus $u\in C^{\delta}_{loc}(\mathbb{R}^N)$ for all $0<\delta<\min\{1,\frac{sp}{p-1}\}$. 
	\end{proof}
	\noindent Now, since any weak solution to \eqref{prob} lies in $C^{\delta}_{loc}(\mathbb{R}^N)$ for some $\delta>s$, by \cite[Theorem~2.5]{Anthal2025Pohozaev} and \cite[Theorem~A1]{Giacomoni2025Normalized}, it must satisfy the Pohozaev identity \eqref{Pohozaev_identity}. This identity will play an essential role in order to prove the existence of normalized solutions. Moreover, it ensures the non-existence for the case of $q=p^*$ and $\lambda<0$ as shown above in \autoref{non_existence}.
	
	\section{Preliminaries for Existence Results}
	\noindent In this section, we will establish the necessary groundwork required to deduce the final existence results.
	\begin{lemma}\label{Lemma 2.1}
		If $u\in S(\tau)$ is a solution of \eqref{prob}, corresponding to some $\lambda \in \mathbb{R}$, then 
		$u \in \mathcal{M}_{\tau}.$
	\end{lemma}
	\begin{proof}
		Since, $u\in S(\tau)$ solves \eqref{prob}, for some $\lambda\in \mathbb{R}$, then we have:
		\begin{equation}\label{2.1}
			\lambda \left\| u \right\|_p^p = \left\| u \right\|_p^p+[u]_{s,p}^p-\mu\left\| u \right\|_q^q-A(u).
		\end{equation}
		Using \eqref{2.1} in \eqref{Pohozaev_identity}, we get
		$$M(u)= \left\| \nabla u \right\|_p^p+s[u]_{s,p}^p-\mu\gamma_{p,q}\left\| u \right\|_q^q-A(u)=0$$
		where $\gamma_{p,q}= \frac{N(q-p)}{pq}$.
	\end{proof}
	\noindent Now, for any $u\in S(\tau)$, by \eqref{S_alpha,p} and Gagliardo-Nirenberg inequality \eqref{G_N_inequality}
	we have:
	\begin{equation}\label{2.3}
		E(u)  =  \frac{T(u)^p}{p}-\mu\frac{\left\| u \right\|_q^q}{q} -\frac{A(u)}{2p^*_{\alpha}}
		\geq  \frac{T(u)^p}{p}-\frac{\mu C_{N,p,q}}{q}T(u)^{q\gamma_{p,q}}\tau^{q-q\gamma_{p,q}}-\frac{T(u)^{2p^*_{\alpha}}}{2p^*_{\alpha}\mathbb{S}^{\frac{2p^*_{\alpha}}{p}}}.
	\end{equation}
	Defining,
	$$h(t):=\frac{t^p}{p}-\frac{\mu C_{N,p,q}}{q}t^{q\gamma_{p,q}}\tau^{q-q\gamma_{p,q}}-\frac{t^{2p^*_{\alpha}}}{2p^*_{\alpha}\mathbb{S}^{\frac{2p^*_{\alpha}}{p}}} \text{ for all } t>0,$$
	we get $E(u)\geq h(T(u))$. Let us discuss some properties of the function $h$, that will be useful in our analysis.
	\begin{lemma}\label{lemma 2.2}
		There exists $\tau_1>0$ such that for all $\tau\in (0,\tau_1)$, $h$ has a strict local minimum at negative level and a global maximum at a positive level. Also, we can find two positive constants $R_1>R_0$ such that $h(R_0)=0= h(R_1)$ with $h(t)>0$ if and only if $t\in (R_0,R_1)$.
	\end{lemma}
	\begin{proof}
		Since
		\begin{eqnarray*}
			h(t) & = & \frac{t^p}{p}-\frac{\mu C_{N,p,q}}{q}t^{q\gamma_{p,q}}\tau^{q-q\gamma_{p,q}}-\frac{t^{2p^*_{\alpha}}}{2p^*_{\alpha}
				\mathbb{S}	}\\
			& = &  t^{q\gamma_{p,q}}\left(\frac{t^{p-q\gamma_{p,q}}}{p}-\frac{\mu C_{N,q,p}}{q}\tau^{q(1-\gamma_{p,q})}-\frac{t^{2p^*_{\alpha}-q\gamma_{p,q}}}{2p^*_{\alpha}
				\mathbb{S}^{\frac{2p^*_{\alpha}}{p}}}\right),
		\end{eqnarray*}
		$h(t)>0$ if and only if $\bar{h}(t)>0$ where 
		$$\bar{h}(t)=\frac{t^{p-q\gamma_{p,q}}}{p}-\frac{\mu C_{N,q,p}}{q}\tau^{q(1-\gamma_{p,q})}-\frac{t^{2p^*_{\alpha}-q\gamma_{p,q}}}{2p^*_{\alpha}
			\mathbb{S}^{\frac{2p^*_{\alpha}}{p}}}.$$
		Now since 
		$$t_0=\left(\frac{2p^*_{\alpha}
			\mathbb{S}^{\frac{2p^*_{\alpha}}{p}}( p-q\gamma_{p,q})}{p(2p^*_{\alpha}-q\gamma_{p,q})}\right)^{\frac{1}{2p^*_{\alpha}-p}}$$
		is the only critical point of $\bar{h}$ corresponding to its global maximum, with $\bar{h}(0^+)=-\frac{\mu C_{N,q,p}}{q}\tau^{q(1-\gamma_{p,q})}$,
		\begin{eqnarray*}
			\bar{h}(t_0) & = & -\frac{\mu C_{N,p,q}}{q}\tau^{q(1-\gamma_{p,q})}+\frac{(2p^*_{\alpha}-p)(2p^*_{\alpha}
				\mathbb{S}^{\frac{2p^*_{\alpha}}{p}}(p-q\gamma_{p,q}))^{\frac{p-q\gamma_{p,q}}{2p^*_{\alpha}-p}}}{(p(2p^*_{\alpha}-q\gamma_{p,q}))^{\frac{2p^*_{\alpha}-q\gamma_{p,q}}{2p^*_{\alpha}-p}}}\\
			& > & 0 \text{ for all } \tau<\tau_1:=\left(\frac{q(2p^*_{\alpha}-p)(2p^*_{\alpha}
				\mathbb{S}^{\frac{2p^*_{\alpha}}{p}}(p-q\gamma_{p,q}))^{\frac{p-q\gamma_{p,q}}{2p^*_{\alpha}-p}}}{\mu C_{N,p,q}(p(2p^*_{\alpha}-q\gamma_{p,q}))^{\frac{2p^*_{\alpha}-q\gamma_{p,q}}{2p^*_{\alpha}-p}}}\right)^{\frac{1}{q(1-\gamma_{p,q})}},
		\end{eqnarray*}
		and $\bar{h}(t)\rightarrow -\infty$ as $t\rightarrow \infty$. Therefore, $\bar{h}$ should have the following curvature:
		
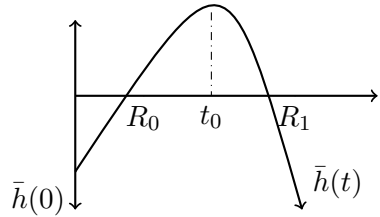
\begin{center}
			\begin{tikzpicture}
				\draw[thick,->] (0,0)--(4,0);
				\draw[thick,<->](0,1)--(0,-1.5);
				\draw[thick,->](0,-1).. controls (2,2) .. (3,-1.5) node[anchor= south west]{$\bar{h}(t)$};
				\draw (2.9,0)--(2.9,0) node[anchor=north]{$R_1$};
				\draw (0.9,0)--(0.9,0) node[anchor=north]{$R_0$};
				\draw [dash dot] (1.8,1.1)--(1.8,0) node[anchor=north]{$t_0$};
				\draw (0,-1)--(0,-1) node[anchor=north east]{$\bar{h}(0)$};
			\end{tikzpicture}
		\end{center}
		and hence 
		\begin{equation}\label{eq_2.4}
			h(t)=\left\{\begin{array}{cl}
				0, & t=R_0 \text{ or } R_1,\\
				<0, & t\in (0,R_0)\cup (R_1,\infty),\\
				>0, & t\in (R_0,R_1),\\
				0^- &  t=0^+.
			\end{array}
			\right.
		\end{equation}
		Now, \eqref{eq_2.4} suggests that $h$ has atleast two critical points. Next, we claim that $h$ has exactly two critical points.
		Suppose that $h$ has atleast three critical points, then since
		\begin{equation*}
			h'(t)  =  t^{q\gamma_{p,q}-1}\left(t^{p-q\gamma_{p,q}}-\mu C_{N,p,q}\gamma_{p,q}\tau^{q(1-\gamma_{p,q})}-\frac{t^{2p^*_{\alpha}-q\gamma_{p,q}}}{\mathbb{S}^{\frac{2p^*_{\alpha}}{p}}}\right),
		\end{equation*}
		$g_1$ must have atleast $3$ roots, where
		$$g_1(t):=t^{p-q\gamma_{p,q}}-\mu C_{N,p,q}\gamma_{p,q}\tau^{q(1-\gamma_{p,q})}-\frac{1}{\mathbb{S}^{\frac{2p^*_{\alpha}}{p}}}t^{2p^*_{\alpha}-q\gamma_{p,q}}.$$
		Now, if we define $g_2$ as follows:
		$$g_2(t):= t^{p-q\gamma_{p,q}}-\frac{t^{2p^*_{\alpha}-q\gamma_{p,q}}}{\mathbb{S}^{\frac{2p^*_{\alpha}}{p}}},$$
		then $g_2$ attains $C:=\mu C_{N,p,q}\gamma_{p,q}\tau^{q(1-\gamma_{p,q})}$ atleast thrice, and hence has atleast two critical points. But since 
		$$g_2'(t)=(p-q\gamma_{p,q})t^{p-q\gamma_{p,q}-1}-\frac{(2p^*_{\alpha}-q\gamma_{p,q})}{\mathbb{S}^{\frac{2p^*_{\alpha}}{p}}}t^{2p^*_{\alpha}-q\gamma_{p,q}-1},$$
		$\bar{t}=\left(\frac{(p-q\gamma_{p,q})\mathbb{S}^{\frac{2p^*_{\alpha}}{p}}}{2p^*_{\alpha}-q\gamma_{p,q}}\right)^{\frac{1}{2p^*_{\alpha}-p}}$ is the unique critical point of $g_2$. Thus, by contradiction, $h$ has exactly two critical points corresponding to a local minimum at negative level and global maximum at positive level and the following geometry:
		
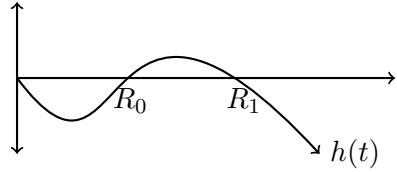
\begin{center}
			\begin{tikzpicture}
				\draw[thick,->](0,0)--(5,0);
				\draw[thick,<->] (0,1)--(0,-1);
				\draw (1.5,0)--(1.5,0) node[anchor=north]{$R_0$};
				\draw (3,0)--(3,0) node[anchor=north]{$R_1$};
				\draw[thick,->](0,0).. controls (1.5,-2) and (1,2.25) .. (4,-1) node[anchor= west]{$h(t)$};
			\end{tikzpicture}
		\end{center}
		Hence, we are done.
	\end{proof}
	\noindent For any $u\in W^{1,p}(\mathbb{R}^N)$, let us define the fiber maps $\star$ and $\circledast$, as follows:
	$$(t\star u)(x):=e^{\frac{Nt}{p}}u(e^tx)\text{ for } t\in \mathbb{R}; \text{ and } (t\circledast u)(x):=t^{\frac{N}{p}}u(tx) \text{ for } t\geq 0.$$
	Clearly, $e^t\circledast u= t\star u$. Now, defining $\psi_{u}(t):=E(t\star u)$, one can notice that $M(t\star u )=\psi_{u}'(t)$. Also we have the following results about $\psi_{u}$:
	\begin{lemma}\label{Lemma 2.3}
		Let $u\in S(\tau)$ and $\tau <\tau_1$, then $\psi_{u}$ has exactly two zeroes and two critical points, that is, we can find unique $a_u<b_u<c_u<d_u$ such that $\psi_{u}'(a_u)=0=\psi_{u}'(c_u)$ and $\psi_{u}(b_u)=0=\psi_{u}(d_u)$. Moreover, we have the following:
		\begin{enumerate}
			\item $a_u\star u \in \mathcal{M}_{\tau}^+$ and $c_u\star u\in \mathcal{M}_{\tau}^-$. If $t\star u \in \mathcal{M}_{\tau}$, then either $t=a_u$ or $t=c_u$ and hence $\mathcal{M}_{\tau}^0$ is empty.
			\item $E(c_u\star u)=\max\{E(t\star u): t\in \mathbb{R}\}>0$ and $\psi_{u}$ is strictly decreasing in $(c_u,\infty)$.
			\item $T(t\star u)\leq R_0$ for every $t<b_u$ and $$E(a_u \star u)=\min\{E(t\star u): t\in \mathbb{R} \text{ and } T(t\star u)\leq R_0\}<0.$$
			\item The maps $\Phi_1:\mathcal{M}_{\tau}\rightarrow \mathbb{R}$ and $\Phi_2:\mathcal{M}_{\tau}\rightarrow \mathbb{R}$ defined as $\Phi_1(u):=a_u$ and $\Phi_2(u):=c_u$ are of class $C^1$.
		\end{enumerate}
	\end{lemma}
	\begin{proof}
		Since	
		$$\psi_{u}(t)= E(t\star u)= \frac{e^{pt}}{p}\left\| \nabla u \right\|_p^p+\frac{e^{pst}}{p}[u]_{s,p}^p-\frac{\mu e^{q\gamma_{p,q}t}}{q}\left\| u \right\|_q^q-\frac{e^{2p^*_{\alpha}t}}{2p^*_{\alpha}}A(u),$$
		we get $$\psi_{u}'(t)= e^{2p^*_{\alpha}t}\left(e^{(p-2p^*_{\alpha})t}\left\| \nabla u \right\|_p^p+se^{(ps-2p^*_{\alpha})t}[u]_{s,p}^p-\gamma_{p,q}\mu e^{(q\gamma_{p,q}-2p^*_{\alpha})t}\left\| u \right\|_q^q-A(u)\right).$$
		If $\psi_{u}$ has more than two critical points, then the function $g$, defined for $t\in \mathbb{R}$ as:
		$$g(t):=e^{(p-2p^*_{\alpha})t}\left\| \nabla u \right\|_p^p+se^{(ps-2p^*_{\alpha})t}[u]^2-\gamma_{p,q}\mu e^{(q\gamma_{p,q}-2p^*_{\alpha})t}\left\| u \right\|_q^q,$$
		attains $A(u)$ atleast thrice and hence has atleast two critical points. Now, since	$$g'(t)= e^{(q\gamma_{p,q}-2p^*_{\alpha})t}(\bar{g}(t)-C_{p,q})$$
		where $$ \bar{g}(t)= (p-2p^*_{\alpha})e^{(p-q\gamma_{p,q})t}\left\| \nabla u \right\|_p^p+s(ps-2p^*_{\alpha})e^{(ps-q\gamma_{p,q})t}[u]_{s,p}^p$$ and $$C_{p,q}=\mu\gamma_{p,q}(q\gamma_{p,q}-2p^*_{\alpha})\left\| u \right\|_q^q,$$ $\bar{g}$ must attain $C_{p,q}$ atleast twice and hence have atleast one critical point.  However since
		$$\bar{g}'(t)= (p-2.p^*_{\alpha})(p-q\gamma_{p,q})e^{(p-q\gamma_{p,q})t}\left\| \nabla u \right\|_p^p+s(ps-2p^*_{\alpha})(ps-q\gamma_{p,q})e^{(ps-q\gamma_{p,q})t}[u]_{s,p}^p<0,$$
		for all $ t\in \mathbb{R}$, we get a contradiction. Hence $\psi_{u}$ has atmost two crtical points. Further, since $t\mapsto T(t\star u)$ is continuous and increasing map from $\mathbb{R}$ onto $(0,+\infty)$, we can find $t_1,t_2\in \mathbb{R}$ such that $R_0=T(t_1\star u)<T(t\star u)<T(t_2\star u)=R_1$ for all $t\in (t_1,t_2)$, by \eqref{2.3} and \autoref{lemma 2.2}
		$$\psi_{u}(t)=E(t\star u)\geq h(T(t\star u))>0 \text{ for all } t\in (t_1,t_2).$$
		Also, one can see that $\psi_{u}(t)\rightarrow -\infty$ as $t\rightarrow +\infty$ and $\psi_{u}(t)\rightarrow 0^-$ as $t\rightarrow -\infty$, because $q\gamma_{p,q}<ps<p<2p^*_{\alpha}$. Thus, $\psi_{u}$ can have the following curvature:
		
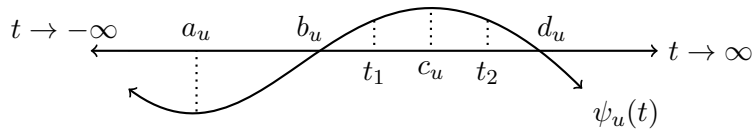
\begin{center}
			\begin{tikzpicture}\label{fig 1}
				\draw[thick,<->] (-4.5,0)--(3,0) node[anchor=west]{$t\rightarrow \infty$};
				\draw (-4,0)--(-4,0)node[anchor= south east]{$t\rightarrow -\infty$};
				\draw[thick,dotted] (-0.75,0.4)--(-0.75,0) node[anchor=north]{$t_1$};
				\draw [thick,dotted](0,0.5)--(0,0)node[anchor=north]{$c_u$};
				\draw[thick,dotted] (0.75,0.4)--(0.75,0) node[anchor=north]{$t_2$};
				\draw [thick,<->] (-4,-0.5)..controls (-2, -2) and (-1,2.5)..(2,-0.5) node[anchor=north west]{$\psi_u(t)$};
				\draw [thick,dotted](-3.1,-0.8)--(-3.1,0) node[anchor=south]{$a_u$};
				\draw (-1.6,0)--(-1.6,0) node[anchor=south]{$b_u$};
				\draw (1.6,0)--(1.6,0) node[anchor=south]{$d_u$};
			\end{tikzpicture}
		\end{center}
		Therefore, $\psi_u$ has exactly two critical points, corresponding respectively to a local minima $(a_u)$ at negative level and global maxima $(c_u)$ at positive level, and exactly two roots ($b_u$ and $d_u$).
		\begin{itemize}
			\item[Proof of 1.]  Since $a_u$ is a strict local minima of $\psi_{u}$, 
			$M(a_u\star u) = \psi_{u}'(a_u)=0$  and 
			\begin{eqnarray*}
				0  < \psi_{u}''(a_u)& = & pe^{pa_u}\left\| \nabla u \right\|_p^p+ps^2e^{psa_u}[u]_{s,p}^p-\mu q\gamma_{p,q}^2e^{q\gamma_{p,q}a_u}\left\| u \right\|_q^q-2p^*_{\alpha}e^{2p^*_{\alpha}}A(u)\\
				& = & p\left\| \nabla (a_u\star u )\right\|_p^p+ps^2[a_u \star u]_{s,p}^p-\mu q\gamma_{p,q}^2\left\| a_u\star u \right\|_q^q-2p^*_{\alpha}A(a_u\star u).
			\end{eqnarray*}		 	 	
			Thus $a_u \star u \in \mathcal{M}_{\tau}^+$. Similarly, since $c_u$ is global maxima of $\psi_u$, we will get $c_u\star u \in \mathcal{M}_{\tau}^-$. 
			Now, if $t\star u \in \mathcal{M}_{\tau}$, then clearly $t$ is a critical point of $\psi_u$, hence either $t=a_u$ or $t=c_u$. Moreover, since $\psi_{u}$ has exactly two critical points, both corresponding to its extremas, $\mathcal{M}_{\tau}^0$ must be an empty set.
			\item[Proof of 2.] It is evident by the curvature of $\psi_{u}$.
			\item[Proof of 3.] By monotonicity of the surjective map $t\mapsto T(t\star u)$ onto $(0,\infty)$, it is clear that $T(t\star u)\leq T(t_1\star u)=R_0$ for all $t<b_u\leq t_1$. Moreover, since $\psi_u$ is decreasing in $(-\infty, a_u)$ and increasing in $(a_u,t_1]$,
			\begin{equation*}
				0>E(a_u\star u)= \psi_u(a_u)  =  \min\{\psi_u(t): t\leq t_1\} = \min\{E(t\star u): T(t\star u)\leq T(t_1\star u)=R_0\}.
			\end{equation*}
			\item[Proof of 4.] By implicit function theorem, as done in the proof of Lemma 3.3 in \cite{Han2022Normalized}, clearly $\Phi_1$ and $\Phi_2$ are of class $C^1$.
		\end{itemize}
	\end{proof}	
	\begin{lemma}\label{Lemma 4}
		If $u\in \mathcal{M}_{\tau}$ is a critical point of $E|_{\mathcal{M}_{\tau}}$, then $u$ is a critical point of $E|_{S(\tau)}$.
	\end{lemma}
	\begin{proof}
		For a critical point $u $ of $E|_{\mathcal{M}_{\tau}}$, by the Lagrange multiplier's rule, there exists $\lambda_1$ and $\lambda_2\in \mathbb{R}$ such that:
		$$E'(u)(v)-\lambda_1 \int_{\mathbb{R}^N}|u|^{p-2}uv-\lambda_2 M'(u)(v)=0 \text{ for all } v\in W^{1,p}(\mathbb{R}^N),$$
		that is,
		\begin{eqnarray*}
			&&(1-p\lambda_2)\int_{\mathbb{R}^N}|\nabla u|^{p-2}\nabla u \nabla v +(1-p\lambda_2 s)\ll u,v \gg_{s,p}\\
			&& =  \mu(1-\lambda_2q\gamma_{p,q})\int_{\mathbb{R}^N}|u|^{q-2}uv+\lambda_1 \int_{\mathbb{R}^N}|u|^{p-2}uv +(1-\lambda_2 2p^*_{\alpha})\int_{\mathbb{R}^N}(I_{\alpha}*|u|^{p^*_{\alpha}})|u|^{p^*_{\alpha}-2}uv,
		\end{eqnarray*}
		for all $v\in W^{1,p}(\mathbb{R}^N)$, and hence, $u$ solves:
		\begin{eqnarray}\label{4.19}
			&& 	-(1-p\lambda_2)\Delta_p u +(1-p\lambda_2 s)(-\Delta_p)^su\nonumber\\
			&&  =\lambda_1 |u|^{p-2}u +\mu(1-\lambda_2 q\gamma_{p,q} )|u|^{q-2}u+(1-\lambda_22p^*_{\alpha})(I_{\alpha}*|u|^{p^*_{\alpha}})|u|^{p^*_{\alpha}-2}u, \text{ in } \mathbb{R}^N.
		\end{eqnarray}
		Claim: $\lambda_2=0$. \\
		Now, as done in the proof of \autoref{Lemma 2.1}, by \eqref{4.19} we have:
		\begin{equation*}
			(1-p\lambda_2)\left\| \nabla u \right\|_p^p+(1-ps\lambda_2)[u]_{s,p}^p = \lambda_1 \left\| u \right\|_p^p+\mu(1-\lambda_2 q\gamma_{p,q})\left\| u \right\|_q^q+(1-\lambda_2 2p^*_{\alpha})A(u)
		\end{equation*} 
		and
		\begin{eqnarray*}
			\lambda_1 \left\| u \right\|_p^p & = & \frac{p}{N}\left((1-p\lambda_2)\left(\frac{N-p}{p}\right)\left\| \nabla u \right\|_p^p+(1-ps\lambda_2)\left(\frac{N-ps}{p}\right)[u]_{s,p}^p\right.\\
			&&-\mu(1-\lambda_2 q\gamma_{p,q})\frac{N}{q}\left\| u \right\|_q^q
			\left. - (1-\lambda_2 2p^*_{\alpha})\left(\frac{N+\alpha}{2p^*_{\alpha}}\right)A(u)\right).
		\end{eqnarray*}
		Thus $$\lambda_2\left(p \left\| \nabla u \right\|_p^p+ps^2[u]_{s,p}^p-\mu q\gamma_{p,q}^2\left\| u \right\|_q^q-2p^*_{\alpha}A(u)\right)=0.$$
		Since $\mathcal{M}_{\tau}^0$ is an empty set, we must have $\lambda_2=0$. Therefore, $u$ is a critical point of $E|_{S(\tau)}$.
	\end{proof}
	\noindent For any $k>0$, denoting $A_k=\{u\in S(\tau): T(u)<k\}$, we define 
	$$m_{\tau}:=\displaystyle \inf_{u\in A_{R_0}}E(u),$$ 
	where $R_0$ is as deduced in \autoref{lemma 2.2}.
	Then, we have the following results for $m_{\tau}$, $m_{\tau}^-$ and $m_{\tau}^+$:
	\begin{lemma}\label{Lemma 2.4}
		$m_{\tau}^->0.$
	\end{lemma}
	\begin{proof}
		For any $u\in \mathcal{M}_{\tau}^-$ we have, $0\star u= u\in \mathcal{M}_{\tau}^-$, then by \autoref{Lemma 2.3},  $0$ is the global maxima of $\psi_u$ at a positive level and	$E(u)=\psi_{u}(0)=\max\{E(t\star u): t\in \mathbb{R}\}>0$. Hence, $m_{\tau}^-\geq 0$. Moreover, for every $u\in \mathcal{M}_{\tau}^-$, we can find some $t_u\in \mathbb{R}$ such that $T(t_u\star u)=t_0$, where $t_0$ is the global maxima of $h$ deduced in \autoref{lemma 2.2}. Thus,
		$$E(u)=\max\{E(t\star u): t\in \mathbb{R}\}\geq E(t_u\star u)\geq h(T(t_u\star u))=h(t_0)>0 \text{ for all } u\in \mathcal{M}_{\tau}^-$$
		from which it follows that $m_{\tau}^-\geq h(t_0)>0$.
	\end{proof}
	\begin{lemma}\label{Lemma 2.5}
		$\displaystyle \sup_{u\in \mathcal{M}_{\tau}^+}E(u)\leq 0 < m_{\tau}^-$ and $\mathcal{M}_{\tau}^+\subset A_{R_0}$.
	\end{lemma}
	\begin{proof}
		Clearly, for any $u\in \mathcal{M}_{\tau}^+$, $a_u=0$. Thus by \autoref{Lemma 2.3} $E(u)<0$ and hence by \autoref{Lemma 2.4} $\displaystyle \sup_{u\in \mathcal{M}_{\tau}^+}E(u)\leq 0 <m_{\tau}^-$. Furthermore, $T(u)=T(a_u\star u)<T(t_1 \star u)=R_0$, for all $u\in \mathcal{M}_{\tau}^+$, since $0=a_u<t_1$. Hence $\mathcal{M}_{\tau}^+\subset A_{R_0}$.
	\end{proof}
	\begin{lemma}\label{Lemma 2.6}
		$\displaystyle -\infty <m_{\tau}=\inf_{u\in \mathcal{M}_{\tau}}E(u)=m_{\tau}^+<0,$ and
		for $\delta>0$ small enough
		\begin{equation}\label{2.4}
			m_{\tau}<\displaystyle \inf_{\bar{A}_{R_0}\setminus A_{R_0-\delta}}E(u).
		\end{equation}
	\end{lemma}
	\begin{proof}
		For any $u\in A_{R_0}$, we have:
		$$E(u)\geq h(T(u))\geq \min_{t\in [0,R_0]}h(t)>-\infty,$$
		and hence $m_{\tau}>-\infty$. Also, since $a_u\star u \in \mathcal{M}_{\tau}^+\subset A_{R_0}$, $$-\infty< m_{\tau}=\displaystyle \inf_{u\in A_{R_0}}E(u)\leq E(a_u\star u)=\psi_u(a_u)<0.$$ 
		Further, if $u\in A_{R_0}$, then by \autoref{Lemma 2.3} $E(u)=E(0\star u)\geq E(a_u\star u)\geq m_{\tau}^+$. Hence $m_{\tau}\geq m_{\tau}^+$. Also since $\mathcal{M}_{\tau}^+\subset A_{R_0}$ we get $m_{\tau}=m_{\tau}^+$. Now, since $\mathcal{M}_{\tau}=\mathcal{M}_{\tau}^+\cup \mathcal{M}_{\tau}^+\cup\mathcal{M}_{\tau}^0$, $\mathcal{M}_{\tau}^0$ is an empty set and 
		$$m_{\tau}^+=\inf_{u\in \mathcal{M}_{\tau}^+}E(u)\leq \sup_{u\in \mathcal{M}_{\tau}^+}E(u)\leq 0<\inf_{u\in \mathcal{M}_{\tau}^-}E(u),$$
		by \autoref{Lemma 2.5}, then clearly $\displaystyle \inf_{u\in \mathcal{M}_{\tau}}E(u)=\inf_{u\in \mathcal{M}_{\tau}^+}E(u)=m_{\tau}^+$. Therefore,
		$$-\infty<m_{\tau}=\inf_{u\in \mathcal{M}_{\tau}}E(u)=m_{\tau}^+<0.$$
		Now, since $h$ is continuous, $h(R_{0})=0$, $h(t)<0$ for all $t\in (0,R_0)$ and $m_{\tau}<0$, we can find $\delta>0$ small enough so that $h(t)\geq \frac{m_{\tau}}{2}$ for all $t\in [R_0-\delta,R_0]$. Hence, for all $u\in \bar{A}_{R_0}\setminus A_{R_0-\delta}$,
		$$R_0-\delta<T(u)\leq R_0 \Rightarrow E(u)\geq h(T(u))\geq \frac{m_{\tau}}{2}>m_{\tau}.$$
		Thus, we get \eqref{2.4}.
	\end{proof}
	\section{First solution}
	In this section, using the above prerequisite results, symmetric decreasing rearrangement tool, and Ekeland variational principle, we will show the existence of a radially symmetric function $u_{\tau}^+\in \mathcal{M}_{\tau}^+$ and $\lambda_{\tau}^+<0$, such that $(u_{\tau}^+,\lambda_{\tau}^+)$ solves \eqref{prob}. The subsequent rearrangement inequalities will be beneficial for this purpose.
	\begin{remark}\label{Remark 3.1}
		For any $f\in W^{1,p}(\mathbb{R}^N)$, let $f^*$ (given by Schwarz symmetrization, see \cite{Kawohl} for further details) be its symmetric decreasing rearrangement, then we have the following:
		\begin{enumerate}
			\item $\left\| \nabla f \right\|_p \geq \left\| \nabla f^*\right\|_p$ and $\left\| f \right\|_p=\left\| f^* \right\|_p$ for all $1\leq p\leq \infty$,
			\item $A(f)\leq A(f^*)$,
			\item  $[f^*]_{s,p}^p \leq [f]_{s,p}^p$.
		\end{enumerate}
	\end{remark}		
	\begin{proof}
		We refer \cite{Baernstein2019Symmetrization, Burchard2009Short, Lieb2001Analysis} and \cite[Remark~2.1]{Nidhi2025Normalized_Asymp} for the proof of 1. and 2. Now,
		\begin{eqnarray*}
			[f^*]_{s,p}^p & = & \int_{\mathbb{R}^N} \int_{\mathbb{R}^N}\frac{|f^*(x)-f^*(y)|^p}{|x-y|^{N+sp}}dxdy = \int_{\mathbb{R}^N}\frac{1}{|z|^{N+sp}}\left(\int_{\mathbb{R}^N}|f^*(y+z)-f^*(y)|^pdy\right)dz\\
			& = & \int_{\mathbb{R}^N}\frac{\left\| \bar{f}^*-f^*\right\|_p^p}{|z|^{N+sp}}dz,
		\end{eqnarray*}
		where $\bar{f}(x):=f(x+z)$. Then by \cite[Ex~1.7]{Burchard2009Short}, we get
		\begin{equation*}
			[f^*]_{s,p}^p  =  \int_{\mathbb{R}^N}\frac{\left\| \bar{f}^*-f^*\right\|_p^p}{|x-y|^{N+sp}}dz \leq \int_{\mathbb{R}^N}\frac{\left\| \bar{f}-f\right\|_p^p}{|x-y|^{N+sp}} dz = [f]_{s,p}^p.
		\end{equation*}
	\end{proof}
	\begin{myproof}{Theorem}{\ref{Theorem 1}}
		Let $\{w_n\}_{n\in \mathbb{N}}\subset A_{R_0}$ be the minimizing sequence for $E$ on $A_{R_0}$. Then taking $w_n^*$ to be the symmetric decreasing rearrangement of $w_n$, by the rearrangement inequalities and \autoref{Remark 3.1}, it can be seen that $\{w_n^*\}_{n\in \mathbb{N}}\subset A_{R_0}$ and $E(w_n^*)\leq E(w_n)$ for each $n\in \mathbb{N}$. Thus, $\{w_n^*\}_{n\in \mathbb{N}}$ is a minimizing sequence as well. Now, for each $n\in \mathbb{N}$, by \autoref{Lemma 2.3} there exists $a_{n}\in \mathbb{R}$ such that $a_n\star w_n^*\in \mathcal{M}_{\tau}^+$ and $E(w_n^*)=E(0\star w_n^*)\geq E(a_n\star w_n^*)$.
		Taking $v_n=a_n\star w_n^*$ to be the minimizing sequence for $E$ on $\mathcal{M}_{\tau}^+$ and hence, that of $E$ on $A_{R_0}$, clearly, $v_n$ is radially decreasing, $v_n\geq 0$ and $T(v_n)<R_0-\delta$ for all $n\in \mathbb{N}$.
		Applying Ekeland variational principle 
		(see \cite[Theorem~1.1 and Corollary~1.3]{Ghoussoub}
		) we can find a sequence of radially decreasing non-negative functions, $\{u_n\}$
		such that\begin{equation}\label{3.1}
			\left\{
			\begin{array}{cc}
				E(u_n)\rightarrow m_{\tau} & \text{ as } n\rightarrow \infty,\\
				E(u_n)\leq E(v_n) & \text{ for all } n\in \mathbb{N},\\
				M(u_n) \rightarrow 0 & \text{ as } n\rightarrow \infty,\\
				E'_{S(\tau)}(u_n) \rightarrow 0 & \text{ as } n\rightarrow \infty.\\
			\end{array}
			\right.
		\end{equation}
		Here, $E'_{S(\tau)}(u_n) \rightarrow 0$ means that the sequence  $y_n=\sup\left\{\frac{E'(u_n)(w)}{\left\| w \right\|}: w\in S({\tau})\right\}$ converges to $0$. Now, by \eqref{3.1} and applying Lagrange multipliers rule, we can find a sequence $\{\lambda_n\}_{n\in \mathbb{N}}$ such that:
		\begin{equation}\label{3.2}
			E'(u_n)-\lambda_n\Phi'(u_n) \rightarrow 0, \text{ where } \Phi(u)=\frac{1}{p}\left\| u \right\|_p^p.
		\end{equation}
		Clearly, since $\{u_n\}_{n\in \mathbb{N}}\subset A_{R_0}$, it is bounded in $W^{1,p}(\mathbb{R}^N)$ and hence, weakly convergent upto a subsequence in $W^{1,p}(\mathbb{R}^N)$. Denoting the subsequence by $\{u_n\}_{n\in\mathbb{N}}$ itself, let $u_0\in W^{1,p}(\mathbb{R}^N)$ be such that $u_n\rightharpoonup u_0$ as $n\rightarrow \infty$. Clearly, $0\leq u_0\in W^{1,p}_r(\mathbb{R}^N)$ and it is radially decreasing.\\
		Claim 1: $\lambda_n \rightarrow \lambda<0$ as $n\rightarrow \infty$, up to a subsequence.\\
		Clearly,
		\begin{equation}\label{new_eq}
			o_n(1)= \left\| \nabla u_n \right\|_p^p+[u_n]_{s,p}^p-\mu \left\| u_n \right\|_q^q-A(u_n)-\lambda_n\tau^p,
		\end{equation}
		by weak convergence of $\{u_n\}$ and \eqref{3.2}.
		Then, by Fatou's lemma and the compact imbedding of $W^{1,p}_r(\mathbb{R}^N)$ in $L^q(\mathbb{R}^N)$ 
		(see \cite[Theorem~II.1]{PLLions}) 
		we have:
		$$\lambda_n \leq\frac{T(u_n)^p}{\tau^p}-\frac{\mu \left\| u_0\right\|_q^q}{\tau^p}-\frac{A(u_0)}{\tau^p}+o(1).$$
		Hence by boundedness of $\{u_n\}_{n\in \mathbb{N}}$ in $W^{1,p}(\mathbb{R}^N)$, 
		$$|\tau^p\lambda_n|\leq |T(u_n)^p|+\mu \left\| u_0 \right\|_q^q+|A(u_0)|+o(1)<+\infty.$$
		Thus $\{\lambda_n\}_{n\in \mathbb{N}}$ is bounded and hence convergent upto a subsequence. Denoting the subsequence by $\{\lambda_n\}_{n\in \mathbb{N}}$ itself, let $\lambda_0\in \mathbb{R}$ be such that $\lambda_n\rightarrow \lambda_0$ as $n\rightarrow \infty$. Now, since $u_n\in \mathcal{M}_{\tau}$, by \eqref{new_eq} and the fact that $\gamma_{p,q}<1$ we get:
		\begin{eqnarray*}
			\lambda_0\tau^p & = & \lim_{n\rightarrow \infty}\left( \left\| \nabla u_n \right\|_p^p+[u_n]_{s,p}^p-\mu\left\| u_n\right\|_q^q-A(u_n)\right)\\
			& = & \lim_{n\rightarrow \infty} \left((1-s)[u_n]_{s,p}^p+\mu\left(\gamma_{p,q}-1\right)\left\| u_n \right\|_q^q\right)< 0,
		\end{eqnarray*}
		for sufficiently large $\mu>0$.\\
		Claim 2: $u_0\neq 0$.\\
		Suppose $u_0=0$, then by the compact embedding $W^{1,p}_r(\mathbb{R}^N)\hookrightarrow L^t(\mathbb{R}^N)$ for all $t\in (p,p^*)$ and \eqref{3.1}, we get $\displaystyle \lim_{n\rightarrow \infty}\left(T_s(u_n)^p-A(u_n)\right)=0$, where $T_s(u):=(\left\| \nabla u \right\|_p^p+s[u]_{s,p}^p)^{\frac{1}{p}}$.
		Now, since $\{u_n\}_{n\in \mathbb{N}}$ is a bounded sequence in $W^{1,p}(\mathbb{R}^N)$, the sequence $\{T_s(u_n)\}_{n\in \mathbb{N}}$ turns out to be a bounded sequence in $\mathbb{R}$ and hence has a convergent subsequence. Denoting the subsequence by $\{T_s(u_n)\}_{n\in \mathbb{N}}$ itself, we get $\displaystyle \lim_{n\rightarrow \infty}A(u_n)=\lim_{n\rightarrow \infty}T_s(u_n)^p$.\\
		Suppose $T_s(u_n)^p\rightarrow l$ as $n\rightarrow\infty$, then by \eqref{S_alpha,p}
		$$l^{\frac{p}{2p^*_{\alpha}}}\mathbb{S}\leq l \Rightarrow l( l^{\frac{p-2p^*_{\alpha}}{2p^*_{\alpha}}}\mathbb{S}-1)\leq 0.$$
		Since $m_{\tau}<0$, $l=0$ will lead us to a contradiction. Indeed if $l=0$, then
		\begin{equation*}
			m_{\tau}=\lim_{n\rightarrow \infty} E(u_n)\geq \lim_{n\rightarrow \infty}\left(\frac{T_s(u_n)^p}{p}-\frac{\mu \left\| u_n \right\|_q^q}{q}-\frac{A(u_n)}{2p^*_{\alpha}}\right)=0.
		\end{equation*}
		Hence we must have $l\geq \mathbb{S}^{\frac{2p^*_{\alpha}}{2p^*_{\alpha}-p}}$. Now,
		\begin{eqnarray*}
			m_{\tau} & = & \lim_{n\rightarrow \infty}E(u_n) =\lim_{n\rightarrow \infty}\left(E(u_n)-\frac{M(u_n)}{2p^*_{\alpha}}\right)\\
			& = & \lim_{n\rightarrow \infty} \left(\left(\frac{2p^*_{\alpha}-p}{2pp^*_{\alpha}}\right)\left\| \nabla u_n \right\|_p^p+\left(\frac{2p^*_{\alpha}-sp}{2pp^*_{\alpha}}\right)[u_n]^2-\mu\left(\frac{1}{q}-\frac{\gamma_{p,q}}{2p^*_{\alpha}}\right)\left\| u_n \right\|_q^q\right)\\
			& \geq & \left(\frac{2p^*_{\alpha}-p}{2pp^*_{\alpha}}\right)\lim_{n\rightarrow \infty}T_s(u_n)^p= \left(\frac{2p^*_{\alpha}-p}{2pp^*_{\alpha}}\right)l\geq \left(\frac{2p^*_{\alpha}-p}{2pp^*_{\alpha}}\right)\mathbb{S}^{\frac{2p^*_{\alpha}}{2p^*_{\alpha}-p}}\geq 0.
		\end{eqnarray*}
		Thus, we are again lead to a contradiction. Therefore $u_0\neq 0$.\\
		Claim 3: $(u_0, \lambda_0)$ solves \eqref{prob}.\\
		Since $\lambda_0<0$, we can define the following equivalent norm on $W^{1,p}(\mathbb{R}^N)$:
		
		$$\left\| u \right\|_{\lambda_0}:=(\left\|\nabla u \right\|_p^p+[u]_{s,p}^p-\lambda_0\left\| u \right\|_p^p)^{\frac{1}{p}}.$$
		Then for any $v\in W^{1,p}(\mathbb{R}^N)$, by \eqref{3.2} we have:
		\begin{eqnarray}\label{3.4}
			0 & = & \lim_{n\rightarrow \infty}\left(E'(u_n)(v)-\lambda_n\Phi'(u_n)(v)\right)\nonumber\\
			& = & \int_{\mathbb R^N}|\nabla u_0|^{p-2}\nabla u_0\nabla v +\ll u_0, v \gg_{s,p}-\mu\int_{\mathbb{R}^N}|u_0|^{q-2}u_0 v\nonumber\\
			&& -A'(u_0)(v)- \lambda_0\int_{\mathbb{R}^N}|u_0|^{p-2}u_0v,
		\end{eqnarray}
		since the mappings, $u\mapsto \frac{\left\|u \right\|_q^q}{q}$ and $A$ defined on $W^{1,p}(\mathbb{R}^N)$ are of class $C^1$. Thus, $u_0$ solves:
		$$-\Delta_p u_0 +(-\Delta_p)^su_0 = \lambda_0|u_0|^{p-2}u_0+\mu|u_0|^{q-2}u_0+(I_{\alpha}*|u_0|^{p^*_{\alpha}})|u|^{p^*_{\alpha}-2}u_0 \text{ in } \mathbb{R}^N.$$
		Next, we will show that $\left\| u_0 \right\|_p=\tau$.
		Following the proof of \autoref{Lemma 2.1}, we have $M(u_0)=0$. Now, define $\bar{u}_n:=u_n-u_0$.
		Since $\bar{u}_n\rightharpoonup 0$ in $W^{1,p}(\mathbb{R}^N)$ and hence in $W^{1,p}_r(\mathbb{R}^N)$, then by Brezis Lieb lemma, lemma 2.4 of \cite{Moroz2013groundstates} and compact imbedding of $W^{1,p}_r(\mathbb{R}^N)$ in $L^q(\mathbb{R}^N)$, we get
		\begin{equation}\label{3.5}
			\left\{  \begin{array}{rl}
				\left\| \nabla \bar{u}_n \right\|_p^p =  \left\| \nabla u_n \right\|_p^p-\left\| u_0 \right\|_p^p+o_n(1),   &  \left[\bar{u}_n\right]_{s,p}^p  =  [u_n]_{s,p}^p-[u_0]_{s,p}^p+o_n(1),\\
				A(\bar{u}_n) =  A(u_n)-A(u_0)+o_n(1) \text{ and } & \left\| \bar{u}_n \right\|_q^q  =  o_n(1).
			\end{array}
			\right\}
		\end{equation}
		Now, by \eqref{3.5},
		\begin{eqnarray*}
			\lim_{n\rightarrow \infty}M(\bar{u}_n) & = & \lim_{n\rightarrow \infty}\left(\left\| \nabla \bar{u}_n \right\|_p^p+s[\bar{u}_n]_{s,p}^p-\mu \gamma_{p,q}\left\| \bar{u}_n\right\|_q^q-A(\bar{u}_n)\right)\\
			& = & \lim_{n\rightarrow \infty}\left(\left\| \nabla u_n \right\|_p^p+s[u_n]_{s,p}^p-A(u_n)- (\left\| \nabla u_0 \right\|_p^p+s[u_0]_{s,p}^p-A(u_0))\right) \\
			& = & \lim_{n\rightarrow \infty} \left(M(u_n)-\mu\gamma_{p,q}\left\| u_n \right\|_q^q-M(u_0)+\mu\gamma_{p,q}\left\| u_0 \right\|_q^q \right)= 0.
		\end{eqnarray*}
		Therefore, $\displaystyle \lim_{n\rightarrow \infty}\left(\left\|\nabla \bar{u}_n\right\|_p^p+s[\bar{u}_n]_{s,p}^p\right)=\lim_{n\rightarrow \infty}\left(\mu\gamma_{p,q}\left\| \bar{u}_n\right\|_q^q+A(\bar{u}_n)\right)=\lim_{n\rightarrow \infty}A(\bar{u}_n)$. Since $\{\bar{u}_n\}_{n\in \mathbb{N}}$ is bounded in $W^{1,p}(\mathbb{R}^N)$, upto subsequence $\{\left\| \nabla \bar{u}_n\right\|_p^p+s[\bar{u}_n]_{s,p}^p\}$ is convergent. Denoting the convergent subsequence as $\{\left\| \nabla \bar{u}_n\right\|_p^p+s[\bar{u}_n]_{s,p}^p\}_{n\in \mathbb{N}}$ itself, let $l\geq 0$, be such that
		\begin{equation}\label{3.6}
			l=\lim_{n\rightarrow \infty}\left(\left\| \nabla \bar{u}_n\right\|_p^p+s[\bar{u}_n]_{s,p}^p\right)=\lim_{n\rightarrow \infty}A(\bar{u}_n).
		\end{equation}
		Then, by \eqref{S_alpha,p}, we have, either $l=0$ or $l\geq \mathbb{S}^{\frac{2p^*_{\alpha}}{2^p*_{\alpha}-p}}$.\\
		Subclaim: $l=0$.\\
		Suppose $l\geq \mathbb{S}^{\frac{2p^*_{\alpha}}{2p^*_{\alpha}-p}}$, then by \eqref{3.5}, Fatou's lemma and Gagliardo-Nirenberg inequality \eqref{G_N_inequality},
		\begin{eqnarray*}
			m_{\tau} & = & \lim_{n\rightarrow \infty}E(u_n)\\
			& = & \lim_{n\rightarrow \infty}\left(\frac{\left\| \nabla \bar{u}_n\right\|_p^p+\left\|\nabla u_0\right\|_p^p}{p}+\frac{[\bar{u}_n]_{s,p}^p+[u_0]_{s,p}^p}{p}-\mu \frac{\left\| u_n \right\|_q^q}{q}-\frac{A(\bar{u}_n)+A(u_0)}{2p^*_{\alpha}}\right)\nonumber\\
			& \geq & \lim_{n\rightarrow \infty}\left(\frac{\left\| \nabla \bar{u}_n\right\|_p^p+s[\bar{u}_n]_{s,p}^p}{p}
			-\frac{A(\bar{u}_n)}{2p^*_{\alpha}}\right)+E(u_0)\nonumber\\
			& = & \left(\frac{2p^*_{\alpha}-p}{2pp^*_{\alpha}}\right)l+E(u_0)\geq \left(\frac{2p^*_{\alpha}-p}{2pp^*_{\alpha}}\right)\mathbb{S}^{\frac{2p^*_{\alpha}}{2p^*_{\alpha}-p}}+E(u_0)\nonumber\\
			& = & \left(\frac{2p^*_{\alpha}-p}{2pp^*_{\alpha}}\right)\mathbb{S}^{\frac{2p^*_{\alpha}}{2p^*_{\alpha}-p}}+E(u_0) -\frac{M(u_0)}{2p^*_{\alpha}}\nonumber\\
			& \geq & \left(\frac{2p^*_{\alpha}-p}{2pp^*_{\alpha}}\right)T(u_0)^p+\mu\left(\frac{q\gamma_{p,q}-2p^*_{\alpha}}{2qp^*_{\alpha}}\right)\left\| u_0 \right\|_q^q+\left(\frac{2p^*_{\alpha}-p}{2pp^*_{\alpha}}\right)\mathbb{S}^{\frac{2p^*_{\alpha}}{2p^*_{\alpha}-p}}\nonumber\\
			& \geq &\left(\frac{2p^*_{\alpha}-p}{2pp^*_{\alpha}}\right)T(u_0)^p+\mu\left(\frac{q\gamma_{p,q}-2p^*_{\alpha}}{2qp^*_{\alpha}}\right)C_{N,p,q}T(u_0)^{q\gamma_{p,q}}\tau^{q(1-\gamma_{p,q})}+ \left(\frac{2p^*_{\alpha}-p}{2pp^*_{\alpha}}\right)\mathbb{S}^{\frac{2p^*_{\alpha}}{2p^*_{\alpha}-p}}\nonumber\\
			& = & f(T(u_0))+\left(\frac{2p^*_{\alpha}-p}{2pp^*_{\alpha}}\right)\mathbb{S}^{\frac{2p^*_{\alpha}}{2p^*_{\alpha}-p}},
		\end{eqnarray*}
		where $$f(t)=\left(\frac{2p^*_{\alpha}-p}{2pp^*_{\alpha}}\right)t^p+\mu\left(\frac{q\gamma_{p,q}-2p^*_{\alpha}}{2qp^*_{\alpha}}\right)C_{N,p,q}t^{q\gamma_{p,q}}\tau^{q-q\gamma_{p,q}}.$$
		Now, since $$t_0=\left(\frac{\mu C_{N,p,q}\gamma_{p,q}(2p^*_{\alpha}-q\gamma_{p,q})\tau^{q(1-\gamma_{p,q})}}{2p^*_{\alpha}-p}\right)^{\frac{1}{p-q\gamma_{p,q}}}$$ is the point of global minima of $f$, 
		\begin{eqnarray*}
			m_{\tau} & \geq & f(t_0)+\left(\frac{2p^*_{\alpha}-p}{2pp^*_{\alpha}}\right)\mathbb{S}^{\frac{2p^*_{\alpha}}{2p^*_{\alpha}-p}}\\
			& = &-\left(\frac{p-q\gamma_{p,q}}{2pqp^*_{\alpha}}\right)\left(\frac{\gamma_{p,q}}{2p^*_{\alpha}-p}\right)^{\frac{q\gamma_{p,q}}{p-q\gamma_{p,q}}}\left(\mu C_{N,p,q}\tau^{q(1-\gamma_{p,q})}(2p^*_{\alpha}-q\gamma_{p,q})\right)^{\frac{p}{p-q\gamma_{p,q}}}\\
			&& +\left(\frac{2p^*_{\alpha}-p}{2pp^*_{\alpha}}\right)\mathbb{S}^{\frac{2p^*_{\alpha}}{2p^*_{\alpha}-p}}\\ 
			& > &0, \text{ for } \tau<\tau_2 = \left(\frac{(2p^*_{\alpha}-p)}{\mu C_{N,p,q}(2p^*_{\alpha}-q\gamma_{p,q})\gamma_{p,q}^{\frac{q\gamma_{p,q}}{p}}}\left(\frac{q\mathbb{S}^{\frac{2p^*_{\alpha}}{2p^*_{\alpha}-p}}}{p-q\gamma_{p,q}}\right)^{\frac{p-q\gamma_{p,q}}{p}}\right)^{\frac{1}{q(1-\gamma_{p,q})}}.
		\end{eqnarray*}
		But this contradicts \autoref{Lemma 2.6}. Therefore $l=0$. Now, by \eqref{3.5} and \eqref{3.6},  $\displaystyle \lim_{n\rightarrow \infty}A(u_n)=A(u_0)$ and $\displaystyle \lim_{n\rightarrow \infty}T(u_n)\rightarrow T(u_0)$, then taking $u_0$ as test function in \eqref{3.4} and using \eqref{3.1} we get:
		\begin{eqnarray*}
			\lambda_0\left\| u_0 \right\|_p^p & = & E'(u_0)(u_0)-\lim_{n\rightarrow \infty}\left(E'(u_n)(u_n)-\lambda_n\Phi'(u_n)(u_n)\right)= \lambda_0\lim_{n\rightarrow \infty}\left\| u_n \right\|_p^p=\lambda_0 \tau^p.
		\end{eqnarray*}
		Hence $u_0$ is a solution of \eqref{prob} and $u_n\rightarrow u_0$ strongly in $W^{1,p}(\mathbb{R}^N)$. Taking $u_{\tau}^+=u_0$ and $\lambda_{\tau}^+=\lambda_0$, we are done.
	\end{myproof}
	\section{Second Solution}
	Until now, we have seen that the infimum of $E$ on $\mathcal{M}_{\tau}^+$ is achieved and is a solution of \eqref{prob}. In this section, we will see that the infimum over $\mathcal{M}_{\tau}^-$, that is, $m_{\tau}^-$ is also achieved. Since the spaces $\mathcal{M}_{\tau}^+$ and $\mathcal{M}_{\tau}^-$ are disjoint, this corresponds to the second normalized solution. 
	Let us start with the following lemma, that will help us through the technicalities to obtain the second solution:
	\begin{lemma}\label{R_1_R_2_R_3}
		Let $u_{\tau}^+$ be the first solution and $w\in W^{1,p}(\mathbb{R}^N)$, then for a fixed $t_0>0$, we have:
		\begin{enumerate}
			\item $\left\| u_{\tau}^+ +t w \right\|_p^p \leq \left\| u_{\tau}^+ \right\|_p^p + \left\| t w \right\|_p^p+ptK_{1,w}(t_0)\left\| w \right\|_p \left\| u_{\tau}^+\right\|_p^{p-1}$, for all $t\in [\frac{1}{t_0},t_0]$.
			\item $\left\| \nabla u_{\tau}^+ +t \nabla w\right\|_p^p \leq \left\| \nabla u_{\tau}\right\|_p^p+ \left\| t \nabla w \right\|_p^p+ ptK_{2,w}(t_0)\left\| \nabla w \right\|_p\left\| \nabla u_{\tau}^+\right\|_p^{p-1},$ for all $t\in [\frac{1}{t_0},t_0]$.
			\item $[u_{\tau}^++t w]_{s,p}^p \leq [u_{\tau}^+]_{s,p}^p + [tw]_{s,p}^p+ptK_{3,w}(t_0)[w]_{s,p}[u_{\tau}^+]_{s,p}^{p-1}$, for all $t\in [\frac{1}{t_0},t_0]$
		\end{enumerate} 
		where $$K_{1,w}(t_0)= \left(1+\frac{t_0 \left\| w \right\|_p}{\left\| u_{\tau}^+\right\|_p}\right)^{p-1}-\left(t_0\frac{\left\| w \right\|_p}{\left\| u_{\tau}^+ \right\|_p}\right)^{p-1}\geq 1,$$  
		$$K_{2,w}(t_0)=\left(1+\frac{t_0 \left\|\nabla w \right\|_p}{\left\| \nabla u_{\tau}^+\right\|_p}\right)^{p-1}-\left(t_0\frac{\left\| \nabla w \right\|_p}{\left\| \nabla u_{\tau}^+ \right\|_p}\right)^{p-1}\geq 1,$$ and $$K_{3,w}(t_0)= \left(1+\frac{t_0 [w]_{s,p}}{ [u_{\tau}^+]_{s,p}}\right)^{p-1}-\left(t_0\frac{ [w ]_{s,p}}{ [u_{\tau}^+]_{s,p}}\right)^{p-1}\geq 1.$$
	\end{lemma}
	\begin{proof}
		In order to prove 1., let us define:
		$$g_1(r):=\left(1+\frac{r\left\| w\right\|_p}{\left\| u_{\tau}^+\right\|_p}\right)^{p-1}-\left(\frac{r\left\| w \right\|_p}{\left\| u_{\tau}^+ \right\|_p}\right)^{p-1}.$$
		Clearly $g_1'(r)\geq 0$ for all $p\geq 2$ and $r\geq 0$, thus $g_1$ is increasing in $[0,\infty)$ and hence 
		$$g_1(r)\leq g_1(t_0) \text{ for all } r\in [0,t_0].$$
		Now, defining
		$$f_1(t):= \left(1+\frac{t \left\| w \right\|_p}{\left\| u_{\tau}^+ \right\|_p}\right)^{p}-pg_1(t_0)\left(\frac{t\left\| w\right\|_p}{\left\| u_{\tau}^+\right\|_p}\right)-\left(\frac{t \left\| w \right\|_p}{\left\| u_{\tau}^+\right\|_p}\right)^p-1,$$
		we get
		$$f_1'(t)=p\left(\frac{\left\| w \right\|_p}{\left\| u_{\tau}^+\right\|_p}\right)\left(g_1(t)-g_1(t_0)\right)\leq 0 \text{ for } t\in \left[0,t_0\right].$$
		Thus, $f_1$ is non-increasng in $[0,t_0]$ and hence $f_1(t)\leq f_1(0)=0$ for all $t\in [0,t_0]$, that is,
		$$\left(1+\frac{t \left\| w \right\|_p}{\left\| u_{\tau}^+ \right\|_p}\right)^{p}\leq pg_1(t_0)\left(\frac{t\left\| w\right\|_p}{\left\| u_{\tau}^+\right\|_p}\right)+\left(\frac{t \left\| w \right\|_p}{\left\| u_{\tau}^+\right\|_p}\right)^p+1 \text{ for all } t\in [0,t_0].$$
		Here, 
		\begin{eqnarray*}
			g_1(t_0) & = & \left(1+\frac{t_0\left\| w \right\|_p}{\left\| u_{\tau}^+\right\|_p}\right)^{p-1}-\left(\frac{t_0\left\| w \right\|_p}{\left\| u_{\tau}^+ \right\|_p}\right)^{p-1}=K_{1,w}(t_0).\\
		\end{eqnarray*}
		Therefore, 
		\begin{eqnarray*}
			\left\| u_{\tau}^++tw \right\|_p^p & \leq & \left(\left\| u_{\tau}^+\right\|_p+t \left\| w \right\|_p\right)^p = \left\| u_{\tau}^+\right\|_p^p\left(1+\frac{t \left\| w\right\|_p}{\left\| u_{\tau}^+\right\|_p}\right)^p\\
			& \leq & p t K_{1,w}(t_0) \left\| w \right\|_p\left\| u_{\tau}^+ \right\|_p^{p-1}+t^p \left\| w \right\|_p^p+ \left\| u_{\tau}^+ \right\|_p^p 
		\end{eqnarray*}
		and hence,	we get 1. Similarly, defining
		$$g_2(r):= \left(1+\frac{r\left\| \nabla w \right\|_p}{\left\| \nabla u_{\tau}^+\right\|_p}\right)^{p-1}-\left(\frac{r\left\| \nabla w \right\|_p}{\left\| \nabla u_{\tau}^+\right\|_p}\right)^{p-1}$$
		and
		$$f_2(t):= \left(1+\frac{t \left\| \nabla w \right\|_p}{\left\| \nabla u_{\tau}^+\right\|_p}\right)^p-pg_2(t_0)\left(\frac{t\left\| \nabla w \right\|_p}{\left\| \nabla u_{\tau}^+\right\|_p}\right)-\left(\frac{t \left\| \nabla w \right\|_p}{\left\| \nabla u_{\tau}^+\right\|_p}\right)^p-1,$$
		we will get
		$$\left(1+\frac{t \left\| \nabla w \right\|_p}{\left\|\nabla u_{\tau}^+ \right\|_p}\right)^{p}\leq pg_2(t_0)\left(\frac{t\left\|\nabla w \right\|_p}{\left\| \nabla u_{\tau}^+\right\|_p}\right)+\left(\frac{t \left\|\nabla w \right\|_p}{\left\| \nabla u_{\tau}^+\right\|_p}\right)^p+1 \text{ for all } t\in [0,t_0],$$
		with
		\begin{eqnarray*}
			g_2(t_0) & = & \left(1+\frac{t_0\left\| \nabla w \right\|_p}{\left\| \nabla u_{\tau}^+\right\|_p}\right)^{p-1}-\left(\frac{t_0\left\| \nabla w \right\|_p}{\left\| \nabla u_{\tau}^+\right\|_p}\right)^{p-1}=K_{2,w}(t_0).\\
		\end{eqnarray*}
		Thus, we get 2. Next, to get 3., we define:
		$$g_3(r):=\left(1+\frac{r[w]_{s,p}}{[u_{\tau}^+]_{s,p}}\right)^{p-1}-\left(\frac{r[w]_{s,p}}{[u_{\tau}^+]_{s,p}}\right)^{p-1}, $$
		$$f_3(t):= \left(1+\frac{t[w]_{s,p}}{[u_{\tau}]_{s,p}}\right)^p-pg_3(t_0)\left(\frac{t[w]_{s,p}}{[u_{\tau}^+]_{s,p}}\right)-\left(\frac{t [w]_{s,p}}{[u_{\tau}^+]_{s,p}}\right)^p-1,$$
		and following the same argument, we derive 3.	
	\end{proof}
	\noindent The following result will play a crucial role in proving the convergence of the Palais Smale sequence, by providing us an upper subcritical bound for $m_{\tau}^-$. 

	\begin{lemma}\label{Lemma 4.1}
		For $p\geq 2$, assume that $p,N$ and $s$ satisfies \eqref{Conditions_p,N,s}
		then for all $\tau<\min\{\tau_1,\tau_2\}$,
		\begin{equation}\label{4.1}
			m_{\tau}^-=\inf_{u\in \mathcal{M}_{\tau}^-}E(u)<m_{\tau}+\left(\frac{2p^*_{\alpha}-p}{2pp^*_{\alpha}}\right)\mathbb{S}^{\frac{2p^*_{\alpha}}{2p^*_{\alpha}-p}}.
		\end{equation}
	\end{lemma}
	\begin{proof}
		Let $\phi_r\in C_c^{\infty}(\mathbb{R}^N)$ be a radially decreasing cut-off function such that 
		\begin{equation}\label{4.2}
			\left\{
			\begin{array}{cl}
				0\leq \phi_r(x) \leq 1 & \text{ for all } x\in \mathbb{R}^N,\\
				\phi_r (x) = 1 & \text{ for } x\in B_r(0),\\
				\phi_r(x)=0  & \text{ for } x\in \mathbb{R}^N\setminus B_{2r}(0).
			\end{array}
			\right.
		\end{equation}
		Then, taking $u_{\epsilon}=\phi_r U_{\epsilon}$, where $U_{\epsilon}$ is as defined in \eqref{U_epsilon}, by  \cite[lemma~5.3]{da2024MIxed}, \cite[Lemma~2.2]{Alves2003Multiplicity} and \cite[Lemma~5.2]{Kaur2010multiplicity} we have:
		\begin{equation}\label{grad_u_epsilon}
			\left\| \nabla u_{\epsilon}\right\|_p^p = S^{\frac{N}{p}}+O\left(\epsilon^{\frac{N-p}{p-1}}\right),
		\end{equation}
		\begin{equation}\label{u_epsilon_p*}
			\left\| u_{\epsilon} \right\|_{p^*}^{p^*} = S^{\frac{N}{p}}+O\left(\epsilon^{\frac{N}{p-1}}\right),
		\end{equation}
		\begin{equation}\label{[U_epsilon]}
			[u_{\epsilon}]_{s,p}^p = O\left(\epsilon^{m_{N,s}}\right) \text{ where } m_{N,p,s}= \min\left\{\frac{N-p}{p-1},p(1-s)\right\},
		\end{equation}
		\begin{equation}\label{U_epsilon_q}
			\left\| u_{\epsilon} \right\|_q^q  =  \left\{
			\begin{array}{ll}
				O\left(\epsilon^{\frac{(N-p)q}{p(p-1)}}\right)	& \text{ for } q<\frac{N(p-1)}{N-p},\\
				O\left(\epsilon^{N-\frac{q(N-p)}{p}}\right) & \text{ for } \frac{N(p-1)}{N-p}<q
			\end{array}
			\right.
		\end{equation}
		and for $q=\frac{N(p-1)}{p}$, we have
		\begin{eqnarray*}
			\left\| u_{\epsilon}\right\|_q^q & = & \int_{B_{r}(0)}|U_{\epsilon}|^q +\int_{\mathbb{R}^N\setminus B_r(0)}|\phi_rU_{\epsilon}|^q \leq \int_{B_{r}(0)}|U_{\epsilon}|^q + \int_{B_{2r}(0)\setminus B_r(0)} |U_{\epsilon}|^q\\
			& = & C_1 \epsilon^{\frac{N}{p}}\ln(1/\epsilon) + O(\epsilon^{\frac{N}{p}}). 
		\end{eqnarray*}
		Particularly,
		\begin{equation}\label{U_epsilon_p}
			\left\| u_{\epsilon}\right\|_p^p= \left\{
			\begin{array}{ll}
				O(\epsilon^p) & \text{ for } p^2<N,\\
				C_1\epsilon^{p}\ln(1/\epsilon) +O(\epsilon^p) & \text{ for }p^2=N,\\
				O(\epsilon^{\frac{N-p}{p-1}}) & \text{ for } p^2>N.
			\end{array}
			\right.
		\end{equation}
		Furthermore, \autoref{prop1}  gives us:
		\begin{equation*}
			A(u_{\epsilon})  \leq  C_{N,\alpha,p}A_{N,\alpha}\left(\left\| u_{\epsilon} \right\|_{p^*}^{p^*}\right)^{\frac{N+\alpha}{N}}.
		\end{equation*}
		Therefore by \eqref{u_epsilon_p*}, we get
		\begin{equation}\label{A(u_epsilon)}
			A(u_{\epsilon}) \leq C_{N,\alpha,p}A_{N,\alpha}S^{\frac{N+\alpha}{p}}+O\left(\epsilon^{\frac{N}{p-1}}\right).
		\end{equation}
		
		For $\zeta, t\geq 0$, define 
		$$\hat{u}_{\epsilon,t}(x):=u_{\tau}^+(x)+tu_{\epsilon}(x);\text{ and }\bar{u}_{\epsilon,t}(x):=\zeta_{\epsilon,t}^{\frac{N-p}{p}}\hat{u}(\zeta x),$$ 
		with $u_{\tau}^+$ being the radial solution deduced in \autoref{Theorem 1}. We will see that $\displaystyle m_{\tau}^-\leq \sup_{t\geq 0}E(\bar{u}_{\epsilon,t})$ and $E(\bar{u}_{\epsilon,t})<m_{\tau}+\left(\frac{2p^*_{\alpha}-p}{2pp^*_{\alpha}}\right)\mathbb{S}^{\frac{2p^*_{\alpha}}{2p^*_{\alpha}-p}}$ for all $t>0$ and small enough $\epsilon>0$.
		Clearly,
		\begin{equation}\label{4.9}
			\left\{
			\begin{array}{c}
				\left\| \nabla \bar{u}_{\epsilon,t} \right\|_p^p=\left\| \nabla\hat{u}_{\epsilon,t}\right\| _p^p; \;\;
				[\bar{u}_{\epsilon,t}]_{s,p}^p = \zeta^{p(s-1)}[\hat{u}_{\epsilon,t}]_{s,p}^p;\;\;\left\| \bar{u}_{\epsilon,t}\right\|_p^p=\zeta^{-p}\left\| \hat{u}_{\epsilon,t} \right\|_p^p\\
				\left\| \bar{u}_{\epsilon,t} \right\|_q^q=\zeta^{q\gamma_{p,q}-q}\left\| \hat{u}_{\epsilon,t}\right\|_q^q;\;\; A(\bar{u}_{\epsilon,t})= A(\hat{u}_{\epsilon,t}).
			\end{array}
			\right.
		\end{equation}
		Then, taking $\zeta=\zeta_{\epsilon,t}=\frac{\left\| \hat{u}_{\epsilon,t}\right\|_p}{\tau}$, we get $\bar{u}_{\epsilon,t}\in S(\tau)$. 
		Thus by \autoref{Lemma 2.3}, we can find $\bar{c}_{\epsilon,t}\in \mathbb{R}$ such that $\bar{c}_{\epsilon,t}\star \bar{u}_{\epsilon,t}\in \mathcal{M}_{\tau}^-$, or, $c_{\epsilon,t}\circledast \bar{u}_{\epsilon,t}\in \mathcal{M}_{\tau}^-$ where $c_{\epsilon,t}=e^{\bar{c}_{\epsilon,t}}>0$. Therefore,
		\begin{equation*}
			0  =  M(c_{\epsilon,t}\circledast \bar{u}_{\epsilon,t}) = c_{\epsilon,t}^p \left\| \nabla \bar{u}_{\epsilon,t} \right\|_p^p+s c_{\epsilon,t}^{ps}[\bar{u}_{\epsilon,t}]_{s,p}^p-\mu\gamma_{p,q} c_{\epsilon,t}^{q\gamma_{p,q}}\left\| \bar{u}_{\epsilon,t}\right\|_q^q-c_{\epsilon,t}^{2p^*_{\alpha}}A(\bar{u}_{\epsilon,t})
		\end{equation*}
		and hence,
		\begin{equation}\label{4.10}
			c_{\epsilon,t}^{p-q\gamma_{p,q}} \left\| \nabla \bar{u}_{\epsilon,t} \right\|_p^p+s c_{\epsilon,t}^{ps-q\gamma_{p,q}}[\bar{u}_{\epsilon,t}]_{s,p}^p=\mu\gamma_{p,q} \left\| \bar{u}_{\epsilon,t}\right\|_q^q+c_{\epsilon,t}^{2p^*_{\alpha}-q\gamma_{p,q}}A(\bar{u}_{\epsilon,t}).
		\end{equation}
		Now, since $0\star \hat{u}_{\epsilon,0}=u_{\tau}^+\in \mathcal{M}_{\tau}^+$, by \autoref{Lemma 2.3}, $\bar{c}_{\epsilon,0}>0$, that is, $c_{\epsilon,0}>1$. Also, by \eqref{4.10}
		\begin{equation*}
			c_{\epsilon,t}^{2p^*_{\alpha}} \leq \frac{c_{\epsilon,t}^p \left\| \nabla\bar{u}_{\epsilon,t}\right\|_p^p+sc_{\epsilon,t}^{ps}[\bar{u}_{\epsilon,t}]_{s,p}^p}{A(\bar{u}_{\epsilon,t})}.
		\end{equation*}
		Defining $$B_{\epsilon,t}:=\frac{\left\| \nabla \bar{u}_{\epsilon,t}\right\|_p^p+s[\bar{u}_{\epsilon,t}]_{s,p}^p}{A(\bar{u}_{\epsilon,t})}, $$
		we get $0<c_{\epsilon,t}\leq \max\left\{B_{\epsilon,t}^{\frac{1}{2p^*_{\alpha}-p}}, B_{\epsilon,t}^{\frac{1}{2p^*_{\alpha}-sp}}\right\}$. By \eqref{4.9}, we have:
		\begin{eqnarray*}
			B_{\epsilon,t} & = & \frac{\left\| \nabla \bar{u}_{\epsilon,t}\right\|_p^p+s[\bar{u}_{\epsilon,t}]_{s,p}^p}{A(\bar{u}_{\epsilon,t})}=\frac{\left\|\nabla \hat{u}_{\epsilon,t}\right\|_p^p+s\zeta_{\epsilon,t}^{p(s-1)}[\hat{u}_{\epsilon,t}]_{s,p}^p}{A(\hat{u}_{\epsilon,t})}\\
			& = &\frac{1}{A(\hat{u}_{\epsilon,t})}\left(\left\|\nabla \hat{u}_{\epsilon,t}\right\|_p^p+s\left(\frac{\tau}{\left\| \hat{u}_{\epsilon,t}\right\|_p^p}\right)^{p(1-s)}[\hat{u}_{\epsilon,t}]_{s,p}^p\right)\leq \frac{\left\| \nabla \hat{u}_{\epsilon,t}\right\|_p^p+s[\hat{u}_{\epsilon,t}]_{s,p}^p}{A(\hat{u}_{\epsilon,t})}\\
			& \leq & C\left(\frac{\left\| \nabla u_{\tau}^+\right\|_p^p+t^p\left\| \nabla u_{\epsilon}\right\|_p^p+s[u_{\tau}^+]_{s,p}^p+st^p[u_{\epsilon}]_{s,p}^p}{t^{2p^*_{\alpha}}A(u_{\epsilon})}\right)\rightarrow 0 \text{ as } t\rightarrow\infty,
		\end{eqnarray*}
		and hence $c_{\epsilon,t}\rightarrow 0$ as $t\rightarrow\infty$. 
		Since $c_{\epsilon,0}>1$, there exists some $t_{\epsilon}>0$ such that $c_{\epsilon,t_{\epsilon}}=1$ (and $\bar{c}_{\epsilon,t_{\epsilon}}=0$), which implies that
		\begin{equation}\label{4.11}
			m_{\tau}^-=\inf_{u\in \mathcal{M}_{\tau}^-}E(u)\leq E(c_{\epsilon,t_{\epsilon}}\circledast \bar{u}_{\epsilon,t_{\epsilon}})= E(\bar{u}_{\epsilon,t_{\epsilon}})\leq \sup_{t\geq 0}E(\bar{u}_{\epsilon,t}). 
		\end{equation}
		Now, since $\hat{u}_{\epsilon,t}\geq u_{\tau}^+$ by \eqref{4.9} and definition of $\bar{u}_{\epsilon,t}$, we have:
		\begin{equation}\label{eq4.12}
			E(\bar{u}_{\epsilon,t})  =  \frac{\left\| \nabla u_{\tau}^++t\nabla u_{\epsilon}\right\|_p^p}{p}+\frac{\zeta_{\epsilon,t}^{p(s-1)}}{p}[u_{\tau}^++tu_{\epsilon}]_{s,p}^p-\frac{\zeta_{\epsilon,t}^{q(\gamma_{p,q}-1)}\mu}{q}\left\| u_{\tau}^++tu_{\epsilon}\right\|_q^q-\frac{A(u_{\tau}^++tu_{\epsilon})}{2p^*_{\alpha}}.
		\end{equation}
		Further, by the generalised binomial theorem and the fact that $\gamma_{p,q}<1$, one gets
		\begin{equation}\label{eq4.12_a}
			\left\| \nabla u_{\tau}^+ +t \nabla u_{\epsilon}\right\|_p^p \leq \left(\left\| \nabla u_{\tau}^+\right\|_p+t \left\| \nabla u_{\epsilon}\right\|_p\right)^p = \left\| \nabla u_{\tau}^+\right\|_p^p+O(t),
		\end{equation}
		\begin{eqnarray}\label{eq4.12_b}
			\zeta_{\epsilon,t}^{p(s-1)}[u_{\tau}^++tu_{\epsilon}]_{s,p}^p & = & \left(\frac{\left\| \hat{u}_{\epsilon,t}\right\|_p}{\tau}\right)^{p(s-1)}[u_{\tau}^++tu_{\epsilon}]_{s,p}^p\leq [u_{\tau}^++tu_{\epsilon}]_{s,p}^p\nonumber\\
			& \leq & \int_{\mathbb{R}^N}\int_{\mathbb{R}^N}\frac{\left(|u_{\tau}^+(x)-u_{\tau}^+(y)|+t|u_{\epsilon}(x)-u_{\epsilon}(y)|\right)^p}{|x-y|^{N+sp}} dxdy\nonumber\\
			& = & [u_{\tau}^+]_{s,p}^p +O(t),
		\end{eqnarray}
		\begin{eqnarray}\label{eq4.12_c}
			\zeta_{\epsilon,t}^{q\gamma_{p,q}-q}\left\| u_{\tau}^++tu_{\epsilon}\right\|_q^q & \geq & \left(1+\frac{t \left\| u_{\epsilon}\right\|_p}{\tau}\right)^{q\gamma_{p,q}-q}\left\| u_{\tau}^++tu_{\epsilon}\right\|_q^q \nonumber\\
			& = & \left(1+O(t)\right)\left\| u_{\tau}^++tu_{\epsilon}\right\|_q^q \geq \left(1+O(t)\right)\left\| u_{\tau}^+\right\|_q^q\nonumber\\
			& = & \left\| u_{\tau}^+\right\|_q^q+O(t),
		\end{eqnarray}
		and
		\begin{eqnarray}\label{eq4.12_d}
			A(u_{\tau}^++tu_{\epsilon}) & = & \int_{\mathbb{R}^N}\int_{\mathbb{R}^N}\frac{A_{N,\alpha}|(u_{\tau}^++tu_{\epsilon})(x)|^{p^*_{\alpha}}|(u_{\tau}^++tu_{\epsilon})(y)|^{p^*_{\alpha}}}{|x-y|^{N-\alpha}}dxdy\nonumber\\
			& \geq & \int_{\mathbb{R}^N}\int_{\mathbb{R}^N}\frac{A_{N,\alpha}|(u_{\tau}^+)(x)|^{p^*_{\alpha}}|(u_{\tau}^+)(y)|^{p^*_{\alpha}}}{|x-y|^{N-\alpha}}dxdy = A(u_{\tau}^+).
		\end{eqnarray}
		Using \eqref{eq4.12_a}, \eqref{eq4.12_b}, \eqref{eq4.12_c} and \eqref{eq4.12_d} in \eqref{eq4.12} we get:
		\begin{eqnarray}\label{4.12}
			E(\bar{u}_{\epsilon,t}) & \leq & \frac{\left\| \nabla u_{\tau}^+\right\|_p^p}{p}+\frac{[u_{\tau}^+]_{s,p}^p}{p}-\frac{\mu \left\| u_{\tau}^+ \right\|_q^q}{q}-\frac{A(u_{\tau}^+)}{2p^*_{\alpha}}+O(t)\nonumber\\
			&& \rightarrow E(u_{\tau}^+) =m_{\tau} <0 \text{ as } t\rightarrow0^+,
		\end{eqnarray}
		uniformly for all $\epsilon>0$.
		Also, since $u_{\tau}^+$ solves \eqref{prob} for some $\lambda_{\tau}^+\in \mathbb{R}$ and it satisfies the Poho\v{z}aev identity, we get the following:
		\begin{equation*}
			\lambda_{\tau}^+ \tau^p = \mu (\gamma_{p,q}-1)\left\| u_{\tau}^+ \right\|_q^q +(1-s) [u_{\tau}^+]_{s,p}^p.
		\end{equation*}
		Thus, we get
		\begin{eqnarray}\label{eq_4.17}
			\mu \zeta_{\epsilon,t}^{q\gamma_{p,q}-q} \left\| u_{\tau}^+ +tu_{\epsilon} \right\|_q^q & \geq & \mu \zeta_{\epsilon,t}^{q\gamma_{p,q}-q}\left\| u_{\tau}^+ \right\|_q^q = \mu \left(\frac{\left\| u_{\tau}^++tu_{\epsilon}\right\|_p}{\tau}\right)^{q\gamma_{p,q}-q}\left\| u_{\tau}^+ \right\|_q^q\nonumber\\
			& \geq & \mu\left(1+\frac{t \left\| u_{\epsilon}\right\|_p}{\tau}\right)^{q\gamma_{p,q}-q}\left\| u_{\tau}^+ \right\|_q^q\nonumber\\
			& \geq & \left(1+\frac{tq(\gamma_{p,q}-1)}{\tau}\left\| u_{\epsilon}\right\|_p\right)\left\| u_{\tau}^+\right\|_q^q\nonumber\\
			& = & \left(1+\frac{tq(\gamma_{p,q}-1)}{\tau}\left\| u_{\epsilon}\right\|_p\right)\left(\frac{(1-s)}{1-\gamma_{p,q}}[u_{\tau}^+]_{s,p}^p-\frac{\lambda_{\tau}^+\tau^p}{1-\gamma_{p,q}}\right)
		\end{eqnarray}
		since $\gamma_{p,q}<1$ and $(1+z)^{\gamma}\geq 1+z\gamma$ for $\gamma< 0$ and $z\in \mathbb{R}$. Now, using \eqref{eq_4.17} in \eqref{eq4.12} together with the fact that $\zeta_{\epsilon,t}^{p(s-1)}\leq 1$ and $(a+b)^p\leq 2^{p-1}(a^p+b^p) $ for all $a,b\geq 0$, $p>1$, we get
		\begin{eqnarray}\label{eq_4.18}
			E(\bar{u_{\epsilon,t}}) & \leq & \frac{2^{p-1}}{p}\left(\left\| \nabla u_{\tau}^+\right\|_p^p +t^p \left\| \nabla u_{\epsilon}\right\|_p^p\right)+ \frac{2^{p-1}}{p}\left([u_{\tau}^+]_{s,p}^p+t^p [u_{\epsilon}]_{s,p}^p\right) -\frac{A(u_{\tau}^++tu_{\epsilon})}{2p^*_{\alpha}}\nonumber\\
			&& -\frac{1}{q}\left(1+\frac{tq(\gamma_{p,q}-1)}{\tau}\left\| u_{\epsilon}\right\|_p\right)\left(\frac{(1-s)}{1-\gamma_{p,q}}[u_{\tau}^+]_{s,p}^p-\frac{\lambda_{\tau}^+\tau^p}{1-\gamma_{p,q}}\right)\nonumber\\
			& = & \frac{2^{p-1}}{p}\left(t^p \left\| \nabla u_{\epsilon}\right\|_p^p+t^p [u_{\epsilon}]_{s,p}^p\right) +\frac{t(1-s)}{\tau}\left\| u_{\epsilon} \right\|_p[u_{\tau}^+]_{s,p}^p -\lambda_{\tau}^+ \tau^{p-1}\left\| u_{\epsilon} \right\|_p t \nonumber\\
			&& -\frac{A(u_{\tau}^++tu_{\epsilon})}{2p^*_{\alpha}} +K_{p,q,\tau},
		\end{eqnarray}
		where $$K_{p,q,\tau}= \frac{2^{p-1}}{p}\left(\left\| \nabla u_{\tau}^+\right\|_p^p+[u_{\tau}^+]_{s,p}^p\right)-\frac{(1-s)[u_{\tau}^+]_{s,p}^p}{q(1-\gamma_{p,q})}+\frac{\lambda_{\tau}^+ \tau^p}{q(1-\gamma_{p,q})}$$
		is independent of $\epsilon$ and $t$. Now, since  $A(u_{\tau}^++u_{\epsilon})\geq t^{2p^*_{\alpha}}A(u_{\epsilon})$ with
		\begin{eqnarray*}
			A(u_{\epsilon}) & = & \int_{\mathbb{R}^N} \int_{\mathbb{R}^N} \frac{A_{N,\alpha}|u_{\epsilon}(x)|^{p^*_{\alpha}}|u_{\epsilon}(y)|^{p^*_{\alpha}}}{|x-y|^{N-\alpha}}dxdy\\
			& \geq &  \int_{B_{r}(0)}\int_{B_{r}(0)}\frac{A_{N,\alpha}|U_{\epsilon}(x)|^{p^*_{\alpha}}|U_{\epsilon}(y)|^{p^*_{\alpha}}}{|x-y|^{N-\alpha}}dxdy\\
			& = &  K_{N,p}^{2p^*_{\alpha}} \epsilon^{\frac{N+\alpha}{p-1}}\int_{B_{r}(0)}\int_{B_{r}(0)} \frac{A_{N,\alpha}dxdy}{|x-y|^{N-\alpha}(\epsilon^{\frac{p}{p-1}}+|x|^{\frac{p}{p-1}})^{\frac{N+\alpha}{2}}(\epsilon^{\frac{p}{p-1}}+|y|^{\frac{p}{p-1}})^{\frac{N+\alpha}{2}}}\\
			& = & \frac{ K_{N,p}^{2p^*_{\alpha}} \epsilon^{\frac{N+\alpha}{p-1}}}{\epsilon^{\frac{p(N+\alpha)}{p-1}}}\int_{B_{r}(0)}\int_{B_{r}(0)}\frac{A_{N,\alpha}dxdy}{|x-y|^{N-\alpha}(1+|\frac{x}{\epsilon}|^{\frac{p}{p-1}})^{\frac{N+\alpha}{2}}(1+|\frac{y}{\epsilon}|^{\frac{p}{p-1}})^{\frac{N+\alpha}{2}}}\\
			& = &  K_{N,p}^{2p^*_{\alpha}}\epsilon^{-N-\alpha}\int_{|y|<\frac{r}{\epsilon}}\int_{|x|<\frac{r}{\epsilon}}\frac{A_{N,\alpha}\epsilon^{2N}dxdy}{|\epsilon x -\epsilon y|^{N-\alpha}(1+|x|^{\frac{p}{p-1}})^{\frac{N+\alpha}{2}}(1+|y|^{\frac{p}{p-1}})^{\frac{N+\alpha}{2}}}\\
			& = & \int_{|x|<\frac{r}{\epsilon}}\int_{|y|<\frac{r}{\epsilon}}\frac{A_{N,\alpha}|U_{1}(x)|^{p^*_{\alpha}}|U_1(y)|^{p^*_{\alpha}}}{|x-y|^{N-\alpha}}dxdy\rightarrow A(U_1) \text{ as } \epsilon \rightarrow 0,
		\end{eqnarray*}
		and by \eqref{eq_4.18}, we get
		\begin{eqnarray}\label{eq_4.19}
			E(\bar{u}_{\epsilon,t}) & \leq & \frac{2^{p-1}}{p}\left(t^p \left\| \nabla u_{\epsilon}\right\|_p^p+t^p [u_{\epsilon}]_{s,p}^p\right) +\frac{t(1-s)}{\tau}\left\| u_{\epsilon} \right\|_q[u_{\tau}^+]_{s,p}^p -\lambda_{\tau}^+ \tau^{p-1}\left\| u_{\epsilon} \right\|_p t \nonumber\\
			&& -t^{2p^*_{\alpha}}\frac{A(u_{\epsilon})}{2p^*_{\alpha}} +K_{p,q,\tau} \rightarrow -\infty \text{ as } t\rightarrow \infty,
		\end{eqnarray}
		uniformly for all $\epsilon>0$, since $\left\| \nabla u \right\|_p^p+[u_{\epsilon}]_{s,p}^p$ are bounded independently of $\epsilon$.
		\noindent Furthermore, since $\bar{c}_{\epsilon, t_{\epsilon}}=0$, we get $\bar{u}_{\epsilon,t_{\epsilon}}=0\star \bar{u}_{\epsilon,t_{\epsilon}}=\bar{c}_{\epsilon,t_{\epsilon}}\star \bar{u}_{\epsilon,t_{\epsilon}}\in \mathcal{M}_{\tau}^-$. Hence by definition of $\bar{c}_{\epsilon,t_{\epsilon}}$ and $2.$ of \autoref{Lemma 2.3} 
		\begin{eqnarray*}
			E(\bar{u}_{\epsilon, t_{\epsilon}}) & = & E(\bar{c}_{\epsilon,t_{\epsilon}}\star \bar{u}_{\epsilon,t_{\epsilon}})=\max\{E(t\star \bar{u}_{\epsilon,t_{\epsilon}}): t\in \mathbb{R}\}>0.
		\end{eqnarray*}
		
	\noindent Thus by \eqref{4.12} and \eqref{eq_4.19}, there exists some $t_0>0$ large enough such that $E(\bar{u}_{\epsilon,t})<0$ for $t\in (0,\frac{1}{t_0})\cup (t_0,\infty)$. Therefore, we need to estimate $E(\bar{u}_{\epsilon,t})$ in $[\frac{1}{t_0}, t_0]$. Above analysis can be summerized by the following plot:
	
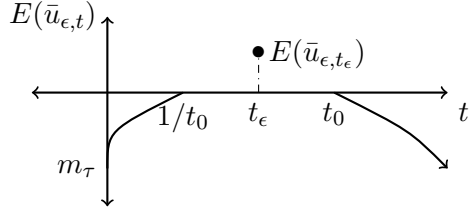
\begin{center}
		\begin{tikzpicture}
			\draw[thick,<->] (-1,0)--(4.5,0) node[anchor=north west]{$t$};
			\draw[thick,<->] (0,-1.5)--(0,1) node[anchor=east]{$E(\bar{u}_{\epsilon,t})$};
			\draw[thick] (1,0).. controls (0,-0.5)..(0,-1) node[anchor=east]{$m_{\tau}$};
			\draw[thick,->] (3,0)..controls(4,-0.5)..(4.5,-1);
			\draw [thick](2,0) node[anchor= north]{$t_{\epsilon}$};
			\draw [dash dot](2,0)--(2,0.5) node[anchor=west]{$E(\bar{u}_{\epsilon,t_{\epsilon}})$};
			\draw (2,0.75) node[anchor=north]{$\bullet$};
			\draw (1,0) node[anchor=north]{$1/t_0$};
			\draw (3,0) node[anchor=north]{$t_0$};
		\end{tikzpicture}
	\end{center}
	Now, let us study $E(\bar{u}_{\epsilon,t})$ for $t\in [1/t_0,t_0]$. 
	Next, we will show that
	$$E(\bar{u}_{\epsilon,t})<m_{\tau}+\left(\frac{2p^*_{\alpha}-p}{2pp^*_{\alpha}}\right)\mathbb{S}^{\frac{2p^*_{\alpha}}{2p^*_{\alpha}-p}} \text{ for } t\in \left[\frac{1}{t_0},t_0\right].$$
	Now, using 2. and 3. of \autoref{R_1_R_2_R_3} for $w= u_{\epsilon}\in W^{1,p}(\mathbb{R}^N)$ in \eqref{eq4.12} we get:
	\begin{eqnarray}\label{eq4.21}
		E(\bar{u}_{\epsilon,t})& = & \frac{\left\| \nabla u_{\tau}^++t\nabla u_{\epsilon}\right\|_p^p}{p}+\frac{\zeta_{\epsilon,t}^{p(s-1)}}{p}[u_{\tau}^++tu_{\epsilon}]_{s,p}^p-\frac{\zeta_{\epsilon,t}^{q(\gamma_{p,q}-1)}\mu}{q}\left\| u_{\tau}^++tu_{\epsilon}\right\|_q^q-\frac{A(u_{\tau}^++tu_{\epsilon})}{2p^*_{\alpha}}\nonumber\\
		& \leq & \frac{\left\| \nabla u_{\tau}^++t\nabla u_{\epsilon}\right\|_p^p}{p}+\frac{[u_{\tau}^++tu_{\epsilon}]_{s,p}^p}{p}-\frac{\zeta_{\epsilon,t}^{q(\gamma_{p,q}-1)}\mu}{q}\left\| u_{\tau}^++tu_{\epsilon}\right\|_q^q-\frac{A(u_{\tau}^++tu_{\epsilon})}{2p^*_{\alpha}}\nonumber\\
		& \leq & \frac{\left\| \nabla u_{\tau}^+ \right\|_p^p}{p} +\frac{\left\| \nabla (tu_{\epsilon})\right\|_p^p}{p} +K_{2,u_{\epsilon}}(t_0)t \left\| \nabla u_{\tau}^+\right\|_p^{p-1} \left\| \nabla u_{\epsilon}\right\|_p +\frac{[u_{\tau}^+]_{s,p}^p}{p}+\frac{[tu_{\epsilon}]_{s,p}^p}{p} \nonumber\\
		&&+tK_{3,u_{\epsilon}}(t_0)[u_{\epsilon}]_{s,p} [u_{\tau}^+]_{s,p}^{p-1} -\frac{\mu \zeta_{\epsilon,t}^{q(\gamma_{p,q}-1)}}{q} \left\| u_{\tau}^++t u_{\epsilon} \right\|_q^q -\frac{A(u_{\tau}^++tu_{\epsilon})}{2p^*_{\alpha}}.
	\end{eqnarray}
	Also, we have:
	\begin{eqnarray}\label{eq4.22}
		A(u_{\tau}^++tu_{\epsilon}) & = & \int_{\mathbb{R}^N} \int_{\mathbb{R}^N} \frac{A_{N,\alpha}|(u_{\tau}^++t u_{\epsilon})(x)|^{p^*_{\alpha}}|(u_{\tau}^++tu_{\epsilon})(y)|^{p^*_{\alpha}}}{|x-y|^{N-\alpha}}dxdy\nonumber\\
		& \geq & \int_{\mathbb{R}^N}\int_{\mathbb{R}^N} \frac{A_{\alpha}\left(|tu_{\epsilon}(x)|^{p^*_{\alpha}}|tu_{\epsilon}(y)|^{p^*_{\alpha}}+|u_{\tau}^+(x)|^{p^*_{\alpha}}|u_{\tau}^+(y)|^{p^*_{\alpha}}\right)}{|x-y|^{N-\alpha}}dxdy\nonumber\\
		&& +\int_{\mathbb{R}^N}\int_{\mathbb{R}^N}\frac{A_{\alpha}p^*_{\alpha}\left(|u_{\tau}^+(x)|^{p^*_{\alpha}-1}|u_{\tau}^+(y)|^{p^*_{\alpha}}|tu_{\epsilon}(x)|+|u_{\tau}^+(x)|^{p^*_{\alpha}-1}|u_{\tau}^+(y)|^{p^*_{\alpha}}|tu_{\epsilon}(x)\right)}{|x-y|^{N-\alpha}}dxdy\nonumber\\
		& = & A(tu_{\epsilon})+A(u_{\tau}^+) +2p^*_{\alpha}\int_{\mathbb{R}^N}(I_{\alpha}*|u_{\tau}^+|^{p^*_{\alpha}})|u_{\tau}^+|^{p^*_{\alpha}-2}u_{\tau}^+ (tu_{\epsilon})dxdy,
	\end{eqnarray}
	since, for any $a,b,c,d>0$ and $z>1$ we have:
	\begin{eqnarray*}
		(a+b)^z(c+d)^z & = & \left(ac+ad+bc+bd\right)^z \geq (ac+bc+ad)^z +(bd)^z\\
		& = & (bd)^z +(ac)^z\left(1+ \frac{bc+ad}{ac}\right)^z\\
		& \geq & b^zd^z + a^zc^z +z(ac)^z \left(\frac{bc+ad}{ac}\right)=b^zd^z + a^zc^z +z a^{z-1}c^zb + za^zc^{z-1}d.
	\end{eqnarray*}
	Thus, \eqref{eq4.22} transforms \eqref{eq4.21} into the following:
	\begin{eqnarray}\label{eq4.23}
		E(\bar{u}_{\epsilon,t}) & \leq & \frac{\left\| \nabla u_{\tau}^+ \right\|_p^p}{p} +\frac{\left\| \nabla (tu_{\epsilon})\right\|_p^p}{p} +tK_{2,u_{\epsilon}}(t_0) \left\| \nabla u_{\tau}^+\right\|_p^{p-1} \left\| \nabla u_{\epsilon}\right\|_p +\frac{[u_{\tau}^+]_{s,p}^p}{p}+\frac{[tu_{\epsilon}]_{s,p}^p}{p} \nonumber\\
		&&+tk_{3,u_{\epsilon}}(t_0)[u_{\epsilon}]_{s,p} [u_{\tau}^+]_{s,p}^{p-1}
		-\frac{\mu \zeta_{\epsilon,t}^{q(\gamma_{p,q}-1)}}{q} \left\| u_{\tau}^++t u_{\epsilon} \right\|_q^q -\frac{A(tu_{\epsilon})}{2p^*_{\alpha}}-\frac{A(u_{\tau}^+)}{2p^*_{\alpha}}\nonumber\\
		&&-\int_{\mathbb{R}^N}(I_{\alpha}*|u_{\tau}^+|^{p^*_{\alpha}})|u_{\tau}^+|^{p^*_{\alpha}-2}u_{\tau}^+ (tu_{\epsilon})dxdy.
	\end{eqnarray}
	Now, by 1. of \autoref{R_1_R_2_R_3} we have:
	\begin{equation*}
		\zeta_{\epsilon,t}^p  =  \frac{\left\| \hat{u}_{\epsilon,t}\right\|_p^p}{\tau^p} = \frac{\left\| u_{\tau}^++tu_{\epsilon}\right\|_p^p}{\tau^p} \leq 1+\frac{\left\| tu_{\epsilon}\right\|_p^p}{\tau^p} +ptK_{1,u_{\epsilon}}(t_0) \frac{\left\| u_{\epsilon}\right\|_p \left\| u_{\tau}^+\right\|_p^{p-1}}{\tau^p},
	\end{equation*}
	and hence
	\begin{eqnarray*}
		\zeta_{\epsilon,t}^{q(\gamma_{p,q}-1)}\left\| u_{\tau}^++tu_{\epsilon} \right\|_q^q & \geq &\left(1+\left(\frac{\left\| t u_{\epsilon}\right\|_p^p}{\tau^p}+pt K_{1,u_{\epsilon}}(t_0)\frac{\left\| u_{\epsilon}\right\|_p \left\| u_{\tau}^+\right\|_p^{p-1}}{\tau^p}\right)\right)^{\frac{q(\gamma_{p,q}-1)}{p}} \left\| u_{\tau}^++tu_{\epsilon} \right\|_q^q\\
		& \geq & \left(1+\frac{q(\gamma_{p,q}-1)}{p}\left(\frac{\left\| t u_{\epsilon}\right\|_p^p}{\tau^p}+pt K_{1,u_{\epsilon}}(t_0)\frac{\left\| u_{\epsilon}\right\|_p \left\| u_{\tau}^+\right\|_p^{p-1}}{\tau^p}\right)\right)\left\| u_{\tau}^++tu_{\epsilon} \right\|_q^q,
	\end{eqnarray*}
	since $(1+a)^{z}\geq 1+z a$ for all $z<0$ and $a\in \mathbb{R}$. Thus, \eqref{eq4.23} becomes:
	\begin{eqnarray*}
		E(\bar{u}_{\epsilon,t}) & \leq & \frac{\left\| \nabla u_{\tau}^+ \right\|_p^p}{p} +\frac{\left\| \nabla (tu_{\epsilon})\right\|_p^p}{p} +K_{2,u_{\epsilon}}(t_0)t \left\| \nabla u_{\tau}^+\right\|_p^{p-1} \left\| \nabla u_{\epsilon}\right\|_p +\frac{[u_{\tau}^+]_{s,p}^p}{p}+\frac{[tu_{\epsilon}]_{s,p}^p}{p} \nonumber\\
		&&+tK_{3,u_{\epsilon}}(t_0)[u_{\epsilon}]_{s,p} [u_{\tau}^+]_{s,p}^{p-1}-\frac{\mu}{q}\left\| u_{\tau^+}+tu_{\epsilon}\right\|_q^q+\frac{\mu(1-\gamma_{p,q})}{p\tau^p}t^p \left\| u_{\epsilon}\right\|_p^p \left\| u_{\tau}^++tu_{\epsilon}\right\|_q^q\\
		&&+K_{1,u_{\epsilon}}(t_0)\frac{\mu(1-\gamma_{p,q})t}{\tau^p}\left\| u_{\epsilon}\right\|_p \left\| u_{\tau}^+ \right\|_p^{p-1} \left\| u_{\tau}^++tu_{\epsilon}\right\|_q^q -\frac{A(tu_{\epsilon})}{2p^*_{\alpha}}-\frac{A(u_{\tau}^+)}{2p^*_{\alpha}}\nonumber\\
		&&-\int_{\mathbb{R}^N}(I_{\alpha}*|u_{\tau}^+|^{p^*_{\alpha}})|u_{\tau}^+|^{p^*_{\alpha}-2}u_{\tau}^+ (tu_{\epsilon})dxdy\\
		& \leq &  \frac{\left\| \nabla u_{\tau}^+ \right\|_p^p}{p} +\frac{\left\| \nabla (tu_{\epsilon})\right\|_p^p}{p} +K_{2,u_{\epsilon}}(t_0)t \left\| \nabla u_{\tau}^+\right\|_p^{p-1} \left\| \nabla u_{\epsilon}\right\|_p +\frac{[u_{\tau}^+]_{s,p}^p}{p}+\frac{[tu_{\epsilon}]_{s,p}^p}{p} \nonumber\\
		&&+tK_{3,u_{\epsilon}}(t_0)[u_{\epsilon}]_{s,p} [u_{\tau}^+]_{s,p}^{p-1}-\frac{\mu}{q}\left\| u_{\tau^+}\right\|_q^q -\frac{\mu}{q} \left\| t u_{\epsilon} \right\|_q^q-\mu\int_{\mathbb{R}^N}|u_{\tau}^+|^{q-2}u_{\tau}^+ (tu_{\epsilon})\\
		&&	+\frac{\mu(1-\gamma_{p,q})}{p\tau^p}t^p \left\| u_{\epsilon}\right\|_p^p \left\| u_{\tau}^++tu_{\epsilon}\right\|_q^q+\frac{\mu(1-\gamma_{p,q})t}{\tau^p}K_{1,u_{\epsilon}}(t_0)\left\| u_{\epsilon}\right\|_p \left\| u_{\tau}^+ \right\|_p^{p-1} \left\| u_{\tau}^++tu_{\epsilon}\right\|_q^q\\ 
		&& -\frac{A(tu_{\epsilon})}{2p^*_{\alpha}}-\frac{A(u_{\tau}^+)}{2p^*_{\alpha}}
		-\int_{\mathbb{R}^N}(I_{\alpha}*|u_{\tau}^+|^{p^*_{\alpha}})|u_{\tau}^+|^{p^*_{\alpha}-2}u_{\tau}^+ (tu_{\epsilon})dxdy,
	\end{eqnarray*}
	since $(a+b)^q\geq a^q+b^q+qa^{q-1}b$, for all $a,b\geq 0$ and $q>p\geq 2$. Therefore,
	\begin{eqnarray*}
		E(\bar{u}_{\epsilon,t}) &\leq & E(u_{\tau}^+) +E(tu_{\epsilon}) +tK_{2,u_{\epsilon}}(t_0) \left\| \nabla u_{\tau}^+\right\|_p^{p-1}\left\| \nabla u_{\epsilon} \right\|_p+tK_{3,u_{\epsilon}}(t_0)[u_{\epsilon}]_{s,p}[u_{\tau}^+]_{s,p}^{p-1}\nonumber\\ && -\mu\int_{\mathbb{R}^N}|u_{\tau}^+|^{q-2}u_{\tau}^+(tu_{\epsilon})\nonumber\\
		&& +\frac{\mu(1-\gamma_{p,q})}{p\tau^p}\left(t^p \left\| u_{\epsilon}\right\|_p^p+tpK_{1,u_{\epsilon}}(t_0)\left\| u_{\epsilon}\right\|_p \left\| u_{\tau}^+\right\|_p^{p-1}\right)\left\| u_{\tau}^++tu_{\epsilon} \right\|_q^q\nonumber\\
		&&-\int_{\mathbb{R}^N}(I_{\alpha}*|u_{\tau}^+|^{p^*_{\alpha}})|u_{\tau}^+|^{p^*_{\alpha}-2}u_{\tau}^+ (tu_{\epsilon})dxdy.
	\end{eqnarray*}
	Denoting
	\begin{eqnarray*}
		\mathcal{X} & := & -tK_{2,u_{\epsilon}}(t_0) \left\| \nabla u_{\tau}^+\right\|_p^{p-1}\left\| \nabla u_{\epsilon}\right\|_p-tK_{3,u_{\epsilon}}(t_0)[u_{\epsilon}]_{s,p}[u_{\tau}^+]_{s,p}^{p-1} +\mu\int_{\mathbb{R}^N}|u_{\tau}^+|^{q-2}u_{\tau}^+(tu_{\epsilon})\\
		&& +\int_{\mathbb{R}^N}(I_{\alpha}*|u_{\tau}^+|^{p^*_{\alpha}})|u_{\tau}^+|^{p^*_{\alpha}-2}u_{\tau}^+ (tu_{\epsilon})dxdy-\mu\frac{(1-\gamma_{p,q})}{\tau^p}tK_{1,u_{\epsilon}}(t_0)\left\| u_{\epsilon}\right\|_p\left\| u_{\tau}^+\right\|_p^{p-1},
	\end{eqnarray*}
	we get
	\begin{eqnarray}\label{eq4.24}
		E(\bar{u}_{\epsilon,t}) &\leq & E(u_{\tau}^+) +E(tu_{\epsilon}) -\mathcal{X}\nonumber\\
		&& +\frac{\mu(1-\gamma_{p,q})}{p\tau^p}\left(t^p \left\| u_{\epsilon}\right\|_p^p
		\right)\left\| u_{\tau}^++tu_{\epsilon} \right\|_q^q.
	\end{eqnarray}
	Now, since $u_{\tau}^+$ and $u_{\epsilon}$ are radially decreasing functions, by H$\ddot{\text{o}}$lder and Cauchy Schwarz inequalities, we have:
	\begin{eqnarray*}
		K_{2,u_{\epsilon}}(t_0)\left\| \nabla u_{\tau}^+ \right\|_p^{p-1} \left\| \nabla u_{\epsilon}\right\|_p & = & 	K_{2,u_{\epsilon}}(t_0) \left(\int_{\mathbb{R}^N}|\nabla u_{\tau}^+|^p\right)^{\frac{p-1}{p}} \left(\int_{\mathbb{R}^N}|\nabla u_{\epsilon}|^p\right)^{\frac{1}{p}}\\
		& \geq & 	K_{2,u_{\epsilon}}(t_0)\int_{\mathbb{R}^N}|\nabla u_{\tau}^+|^{p-1}|\nabla u_{\epsilon}| \\
		& \geq & 	K_{2,u_{\epsilon}}(t_0)\int_{\mathbb{R}^N}|\nabla u_{\tau}^+|^{p-2} \nabla u_{\tau}^+ \nabla u_{\epsilon}.
	\end{eqnarray*}	
	{Since $u_{\tau}^+$ and $u_{\epsilon}$ are radially decreasing functions, there exists $f_{\tau}$ and $f_{\epsilon}$ such that, they are non-increasing in $[0,\infty)$, with
		$$u_{\tau}^+(x) = f_{\tau}(|x|) \text{ and } u_{\epsilon}(x)=f_{\epsilon}(|x|).$$
		Thus, for any $x\in \mathbb{R}^N$
		\begin{eqnarray*}
			\nabla u_{\tau}^+(x) \nabla u_{\epsilon}(x) & = & f_{\tau}'(|x|) f_{\epsilon}'(|x|) \geq 0.
		\end{eqnarray*} 
	}
	Similarly, since $u_{\tau}^+$ and $u_{\epsilon}$ are non-negative functions, by H$\ddot{\text{o}}$lder inequality we have:
	\begin{eqnarray*}
		K_{1,u_{\epsilon}}(t_0) \int_{\mathbb{R}^N}|u_{\tau}^+|^{p-2}u_{\tau}^+u_{\epsilon} & \leq & K_{1,u_{\epsilon}}(t_0) \left(\int_{\mathbb{R}^N}|u_{\tau}^+|^p\right)^{\frac{p-1}{p}}\left(\int_{\mathbb{R}^N} |u_{\epsilon}|^p\right)^{\frac{1}{p}}.
	\end{eqnarray*}
	Thus,
	\begin{eqnarray*}
		K_{1,u_{\epsilon}}(t_0)\left\| u_{\tau}^+\right\|_p^{p-1}\left\| u_{\epsilon}\right\|_p & \geq & K_{1,u_{\epsilon}}(t_0)\int_{\mathbb{R}^N} | u_{\tau}^+|^{p-2}  u_{\tau}^+  u_{\epsilon}
	\end{eqnarray*}
	and
	\begin{eqnarray*}
		t[u_{\epsilon}]_{s,p}[u_{\tau}^+]_{s,p}^p & = & 		t\left(\int_{\mathbb{R}^N}\int_{\mathbb{R}^N}\frac{|u_{\epsilon}(x)-u_{\epsilon}(y)|^p}{|x-y|^{\left(\frac{N+sp}{p}\right)p}}\right)^{\frac{1}{p}}\left(\int_{\mathbb{R}^N}\int_{\mathbb{R}^N}\frac{|u_{\tau}^+(x)-u_{\tau}^+(y)|^{\left(\frac{p}{p-1}\right)(p-1)}}{|x-y|^{\left(N+sp\right)\left(\frac{p-1}{p}\right)\left(\frac{p}{p-1}\right)}}\right)^{\frac{p-1}{p}}\\
		& \geq & 		t\int_{\mathbb{R}^N}\int_{\mathbb{R}^N}\frac{|u_{\epsilon}(x)-u_{\epsilon}(y)||u_{\tau}^+(x)-u_{\tau}^+(y)|^{p-1}}{|x-y|^{\left(\frac{N+sp}{p}\right)+\left(\frac{(p-1)(N+sp)}{p}\right)}} \geq \ll u_{\tau}^+, tu_{\epsilon}\gg.
	\end{eqnarray*}
	Therefore, since $K_{2,u_{\epsilon}}(t_0)$ and $K_{3,u_{\epsilon}}(t_0)\geq 1$,
	\begin{eqnarray*}
		\mathcal{X} & \leq & -K_{2,u_{\epsilon}}(t_0)\int_{\mathbb{R}^N}|\nabla u_{\tau}^+|^{p-2} \nabla u_{\tau}^+ \nabla(t u_{\epsilon})-		K_{3,u_{\epsilon}}(t_0)\ll u_{\tau}^+, tu_{\epsilon}\gg+\mu\int_{\mathbb{R}^N}|u_{\tau}^+|^{q-2}u_{\tau}^+(tu_{\epsilon})\\
		&& +\int_{\mathbb{R}^N}(I_{\alpha}*|u_{\tau}^+|^{p^*_{\alpha}})|u_{\tau}^+|^{p^*_{\alpha}-2}u_{\tau}^+ (tu_{\epsilon})dxdy-\mu\frac{(1-\gamma_{p,q})}{\tau^p}tK_{1,u_{\epsilon}}(t_0)\int_{\mathbb{R}^N} | u_{\tau}^+|^{p-2}  u_{\tau}^+  u_{\epsilon}\\
		& \leq &-  E'(u_{\tau}^+)(tu_{\epsilon})-\mu\frac{(1-\gamma_{p,q})}{\tau^p}tK_{1,u_{\epsilon}}(t_0)\int_{\mathbb{R}^N} | u_{\tau}^+|^{p-2}  u_{\tau}^+  u_{\epsilon}\\
		& = & -\left(\lambda+\mu\frac{(1-\gamma_{p,q})}{\tau^p}tK_{1,u_{\epsilon}}(t_0) \right) \int_{\mathbb{R}^N}|u_{\tau}^+|^{p-2} u_{\tau}^+(tu_{\epsilon}) \\
		& \leq & t\left(|\lambda|+\mu\frac{(1-\gamma_{p,q})}{\tau^p}K_{1,u_{\epsilon}}(t_0)\right) \int_{\mathbb{R}^N} |u_{\tau}^+(x)|^{p-1}|u_{\epsilon}(x)|dx\\
		& = & t\left(|\lambda|+\mu\frac{(1-\gamma_{p,q})}{\tau^p}K_{1,u_{\epsilon}}(t_0)\right)\int_{B_{2r}(0)} |u_{\tau}^+(x)|^{p-1}|u_{\epsilon}(x)|dx.
	\end{eqnarray*}
	In addition defining $\bar{u}(x):= u(x+y)$ for some fixed $y\in \mathbb{R}^N$ such that $|x-y|>\left|\frac{y}{2}\right|>1$ for all $x\in B_{2r}(0)$, by radial lemma \cite[Proposition~1.1]{Sickel2014radial} we get:
	\begin{eqnarray*}
		\int_{B_{2r}(0)} |u_{\tau}^+(x)|^{p-1}|u_{\epsilon}(x)|dx & = & \int_{B_{2r}(0)} |\bar{u}(x-y)|^{p-1}|u_{\epsilon}(x)|dx\\
		& \leq & C\int_{B_{2r}(0)} \frac{|u_{\epsilon}(x)|}{|x-y|^{\frac{(N-1)(p-1)}{p}}}\leq C \int_{B_{2r}(0)}|u_{\epsilon}(x)|\left|\frac{2}{y}\right|^{\frac{(N-1)(p-1)}{p}}dx\\
		& \leq & C_{N,p}'\int_{B_{2r}(0)}\frac{\epsilon^{\frac{N-p}{p(p-1)}}dx}{(\epsilon^{\frac{p}{p-1}}+|x|^{\frac{p}{p-1}})^{\frac{N-p}{p}}}=\bar{C} \epsilon^{\frac{N-p}{p(p-1)}}.
	\end{eqnarray*}
	Thus $\mathcal{X}=(1+O(\left\| u_{\epsilon}\right\|_p^{p-1}))O(\epsilon^{\frac{N-p}{p(p-1)}})$
	since
	\begin{eqnarray*}
		K_{1,u_{\epsilon}} & \leq & \left(1+ \frac{t_0 \left\| u_{\epsilon}\right\|_p}{\left\| u_{\tau}\right\|_p}\right)^{p-1}\leq \tilde{C}\left(1+\left(\frac{t_0}{\left\| u_{\tau}^+\right\|_p}\right)^{p-1}\left\| u_{\epsilon}\right\|_p^{p-1}\right).
	\end{eqnarray*}
	Hence, \eqref{eq4.24} becomes
	\begin{eqnarray}\label{eq4.25}
		E(\bar{u}_{\epsilon,t}) &\leq & E(u_{\tau}^+) +E(tu_{\epsilon}) -(1+O(\left\| u_{\epsilon}\right\|_p^{p-1}))O\left(\epsilon^{\frac{N-p}{p(p-1)}}\right)\nonumber\\
		&& +\frac{\mu(1-\gamma_{p,q})}{p\tau^p}\left(t^p \left\| u_{\epsilon}\right\|_p^p
		\right)\left\| u_{\tau}^++tu_{\epsilon} \right\|_q^q
		\leq  E(u_{\tau}^+) +E(tu_{\epsilon}) 
	\end{eqnarray}
	since
	\begin{eqnarray*}
		\frac{\mu(1-\gamma_{p,q})}{p\tau^p}t^p \left\| u_{\epsilon}\right\|_p^p\left\| u_{\tau}^++tu_{\epsilon}\right\|_q^q & = & \left(\left\| u_{\tau}^+\right\|_q^q +O\left(\left\| u_{\epsilon}\right\|_q^q\right)+O\left(\left\| u_{\epsilon}\right\|_q\right)\right)\frac{\mu(1-\gamma_{p,q})}{p\tau^p}t^p \left\| u_{\epsilon}\right\|_p^p\\
		& = & \left(\left\| u_{\tau}^+\right\|_q^q +O\left(\left\| u_{\epsilon}\right\|_q^q\right)+O\left(\left\| u_{\epsilon}\right\|_q\right)\right)O\left(\left\| u_{\epsilon}\right\|_p^p\right).
	\end{eqnarray*}
	Thus,
	\begin{eqnarray*}
		&&\frac{\mu(1-\gamma_{p,q})}{p\tau^p}t^p \left\| u_{\epsilon}\right\|_p^p\left\| u_{\tau}^++tu_{\epsilon}\right\|_q^q -(1+O(\left\| u_{\epsilon}\right\|_p^{p-1}))O\left(\epsilon^{\frac{N-p}{p(p-1)}}\right)\\
		&&= O(\left\| u_{\epsilon}\right\|_p^p)-(1+O(\left\| u_{\epsilon}\right\|_p^{p-1}))O\left(\epsilon^{\frac{N-p}{p(p-1)}}\right)	\\
		&& =\left\{ 
		\begin{array}{ll}
			O(\epsilon^p) -(1+O(\epsilon^{p-1}))O(\epsilon^{\frac{N-p}{p(p-1)}}) & \text{ for } p^2<N,\\
			\epsilon^p \ln(1/\epsilon) +O(\epsilon^p) -(1+\epsilon^{p-1}(\ln(1/\epsilon))^{\frac{p-1}{p}}+O(\epsilon^{p-1}))O(\epsilon) & \text{ for } p^2=N,\\
			O(\epsilon^{\frac{N-p}{p-1}}) -(1+O(\epsilon^{\frac{N-p}{p}}))O(\epsilon^{\frac{N-p}{p(p-1)}}) & \text{ for } p^2>N,
		\end{array}
		\right.\\
		&& < 0 \text{ whenever, either } p^2\geq N \text{ or } p^2< N < \frac{p^3+p}{2},
	\end{eqnarray*}
	since $p>1$. Note that for $p^2<N<\frac{p^3+p}{2}$, we get $p^3-p^2+p > p^3-N+p >2N-N=N$, and hence $p>\frac{N-p}{p(p-1)}$ and for $p^2\geq N$, we will directly get the result together with $p>1$.
	Therefore, { for $0<\frac{1}{t_0}\leq t \leq t_0$ and $u_{\epsilon}\neq 0$ we have:}  
	\begin{eqnarray*}
		E(\bar{u}_{\epsilon,t}) & \leq & m_{\tau}+E(tu_{\epsilon})\\
		& = & {m_{\tau}+\frac{t^pT(u_{\epsilon})}{p}-\mu \frac{t^q\left\| u_{\epsilon}\right\|_q^q}{q}-\frac{t^{2p^*_{\alpha}}A(u)}{2p^*_{\alpha}}}\\
		& {<}& {m_{\tau}+\frac{t^pT(u_{\epsilon})}{p}-\frac{t^{2p^*_{\alpha}}A(u)}{2p^*_{\alpha}}}\\
		& = & m_{\tau} +f_{u_{\epsilon}}(t)
		\text{ where }f_{u}(t):=\frac{t^pT(u)^p}{p}-\frac{t^{2p^*_{\alpha}}A(u)}{2p^*_{\alpha}}.
	\end{eqnarray*}
	{ Now,} since $f_u$ has global maxima at $t_u=\left(\frac{T(u)^p}{A(u)}\right)^{\frac{1}{2p^*_{\alpha}-p}}$, by \eqref{grad_u_epsilon}, \eqref{[U_epsilon]} and \eqref{A(u_epsilon)}, we get also:
	\begin{eqnarray}\label{eq_5.27}
		E(\bar{u}_{\epsilon,t}) & { <} & m_{\tau} +f_{u_{\epsilon}}(t_{u_{\epsilon}}) =m_{\tau}+\left(\frac{2p^*_{\alpha}-p}{2pp^*_{\alpha}}\right)\left(\frac{T(u_{\epsilon})^p}{A(u_{\epsilon})^{\frac{p}{2p^*_{\alpha}}}}\right)^{\frac{2p^*_{\alpha}}{2p^*_{\alpha}-p}}\nonumber\\
		& \leq & m_{\tau} +  \left(\frac{2p^*_{\alpha}-p}{2pp^*_{\alpha}}\right)\left(\frac{S^{\frac{N}{p}}+O(\epsilon^{\frac{N-p}{p-1}})+O(\epsilon^{m_{N,p,s}})}{\left((A_{\alpha}C_{N,\alpha,p})^{\frac{N}{2}}S_{\alpha}^{\frac{N+\alpha}{p}}+O(\epsilon^{\frac{N}{p-1}})\right)^{\frac{p}{2p^*_{\alpha}}}}\right)^{\frac{2p^*_{\alpha}}{2p^*_{\alpha}-p}}\nonumber\\
		& \leq & m_{\tau}+\left(\frac{2p^*_{\alpha}-p}{2pp^*_{\alpha}}\right)\left(\frac{S}{(A_{\alpha}C_{N,\alpha,p})^{\frac{p}{2p^*_{\alpha}}}}\right)^{\frac{2p^*_{\alpha}}{2p^*_{\alpha}-p}} \nonumber\\
		& \leq & m_{\tau}+\left(\frac{2p^*_{\alpha}-p}{2pp^*_{\alpha}}\right)\mathbb{S}^{\frac{2p^*_{\alpha}}{2p^*_{\alpha}-p}},
	\end{eqnarray}
	for $\epsilon$ small and uniformly in $t\in [1/t_0,t_0]$. The second last inequality above holds, since 
	\begin{eqnarray*}
		m_{N,p,s} & = & \min\left\{p(1-s),\frac{N-p}{p-1}\right\}=\left\{
		\begin{array}{cc}
			\min\left\{p(1-s),\frac{N-p}{p-1}\right\} & \text{ for } p^2>N\\
			p(1-s)& \text{ for } p^2\leq N
		\end{array}\right.\\
		& \geq &  \frac{N-p}{p(p-1)} \text{ for } \left\{\begin{array}{l}
			p^2>N \text{ if, either } \frac{N-p}{p(p-1)}<p(1-s)<\frac{N-p}{p-1} \text{ or } p(1-s) \geq \frac{N-p}{p-1},\\
			p^2\leq N \text{ if } p^2\leq N \leq \min\{\frac{p^3+p}{2}, p\left(p(1-s)(p-1)+1\right)\} .
		\end{array}
		\right.
	\end{eqnarray*}
	Thus, if \eqref{Conditions_p,N,s} is satisfied, then \eqref{eq_5.27} holds. Therefore, by \eqref{4.11} we are done. 
\end{proof}
\noindent Now, if $u\in \mathcal{M}_{\tau}^{\pm}$ for $0<\tau<\min\{\tau_1,\tau_2\}$, then $v_{\beta}:=\frac{\beta}{\tau}u \in S(\beta)$, for $\beta>0$. Now, by \autoref{Lemma 2.3}, for every $0<\beta<\min\{\tau_1,\tau_2\}$ we can find $\bar{t}_{\pm}(\beta)\in \mathbb{R}$, such that 
$$\bar{t}_{\pm}(\beta)\star v_{\beta}\in \mathcal{M}_{\beta}^{\pm} \Rightarrow t_{\pm}(\beta)\circledast v_{\beta}\in \mathcal{M}_{\beta}^{\pm}\text{ where } t_{\pm}(\beta)=e^{\bar{t}_{\pm}(\beta)}>0.$$
Here $t_{\pm}(\tau)=1$, since $v_{\tau}=u\in \mathcal{M}_{\tau}^{\pm}$. Moreover, we have the following results for $t_{\pm}$:
\begin{lemma}\label{Lemma5.3}
	For $N\geq 3$, $p<q<p+\frac{p^2}{N}$ and $0<\tau<\min\{\tau_1,\tau_2\}$, $t_{\pm}$ is differentiable at $\tau$ with
	\begin{equation}\label{t_tau'}
		t'_{\pm}(\tau)= \frac{\mu q \gamma_{p,q}\left\| u \right\|_{q}^q+2p^*_{\alpha}A(u)-sp[u]_{s,p}^p-p \left\| \nabla u \right\|_p^p}{\tau\left(p\left\| \nabla u \right\|_p^p+ps^2[u]_{s,p}^p-\mu q \gamma_{p,q}^2 \left\| u \right\|_q^q-2p^*_{\alpha}A(u)\right)}.
	\end{equation}
	Moreover, for sufficiently large $\mu>0$, we have
	$$E(t_{\pm}(\beta)\circledast v_{\beta})<E(u) \text{ whenever } \tau<\beta<\min\{\tau_1,\tau_2\}.$$
\end{lemma}
\begin{proof}
	Since $t_{\pm}(\beta)\circledast v_{\beta}\in \mathcal{M}_{\tau}^{\pm}$, we have
	\begin{eqnarray}\label{5.27}
		0 & = & M(t_{\pm}(\beta)\circledast v_{\beta})\nonumber\\
		& = & (t_{\pm}(\beta))^p\left\| \nabla v_{\beta} \right\|_p^p +s(t_{\pm}(\beta))^{sp}[v_{\beta}]_{s,p}^p -\mu \gamma_{p,q}(t_{\pm}(\beta))^{q\gamma_{p,q}}\left\| v_{\beta} \right\|_q^q -(t_{\pm}(\beta))^{2p^*_{\alpha}}A(v_{\beta})\nonumber\\
		& = & \left(t_{\pm}(\beta)\frac{\beta}{\tau}\right)^p\left\| \nabla u \right\|_p^p +s\left(\frac{(t_{\pm}(\beta)^s\beta)}{\tau}\right)^p[u]_{s,p}^p-\mu \gamma_{p,q}\left((t_{\pm}(\beta))^{\gamma_{p,q}}\frac{\beta}{\tau}\right)^q\left\| u \right\|_q^q\nonumber\\
		&& -\left(\frac{t_{\pm}(\beta)\beta}{\tau}\right)^{2p^*_{\alpha}}A(u).
	\end{eqnarray}
	Now, define the map $\Phi:(0,\min\{\tau_1,\tau_2\})\times (0,\infty)\rightarrow \mathbb{R}$ as
	$$\Phi(\beta,t):=\left(\frac{\beta t}{\tau}\right)^p\left\| \nabla u \right\|_p^p +s\left(\frac{\beta t^s}{\tau}\right)^p[u]_{s,p}^p-\mu\gamma_{p,q}\left(\frac{\beta t^{\gamma_{p,q}}}{\tau}\right)^{q}\left\| u \right\|_q^q-\left(\frac{\beta t}{\tau}\right)^{2p^*_{\alpha}}A(u).$$
	Clearly, by \eqref{5.27}, $\Phi(\beta, t_{\pm}(\beta))=0$ for all $0<\beta<\min\{\tau_1,\tau_2\}$. In addition,
	\begin{equation*}
		\frac{\partial \Phi}{\partial t}(\tau,1) = p \left\| \nabla u \right\|_p^p+sp^2 [u]_{s,p}^p -\mu q\gamma_{p,q}^2 \left\| u \right\|_q^q-2p^*_{\alpha}A(u) \neq 0,
	\end{equation*}
	since $\mathcal{M}_{\tau}^0$ is empty. Thus, by the inverse function theorem, $\beta \mapsto t_{\pm}(\beta)$ is differentiable at $\tau$ with
	\begin{equation*}
		t'_{\pm}(\tau)  =  -\frac{\frac{\partial \Phi}{\partial \beta}(\tau,1)}{\frac{\partial \Phi}{\partial t}(\tau,1)}
		=  \frac{\mu q \gamma_{p,q}\left\| u \right\|_{q}^q+2p^*_{\alpha}A(u)-sp[u]_{s,p}^p-p \left\| \nabla u \right\|_p^p}{\tau\left(p\left\| \nabla u \right\|_p^p+ps^2[u]_{s,p}^p-\mu q \gamma_{p,q}^2 \left\| u \right\|_q^q-2p^*_{\alpha}A(u)\right)}.
	\end{equation*}
	Moreover,
	\begin{equation}\label{5.28}
		1+\tau t_{\pm}'(\tau) = \frac{sp(s-1)[u]_{s,p}^p +\mu q \gamma_{p,q}(1-\gamma_{p,q})\left\| u \right\|_q^q}{p \left\| \nabla u \right\|_p^p +s^2p [u]_{s,p}^p- \mu q \gamma_{p,q}^2\left\| u \right\|_q^q-2p^*_{\alpha}A(u)}.
	\end{equation}
	Now, since $M(t_{\pm}(\beta)\circledast v_{\beta})=0$, we have:
	\begin{eqnarray}\label{5.29}
		E(t_{\pm}(\beta)\circledast v_{\beta}) & = & \frac{\left\| \nabla (t_{\pm}(\beta)\circledast v_{\beta})\right\|_p^p}{p}+\frac{[t_{\pm}(\beta)\circledast v_{\beta}]_{s,p}^p}{p}-\frac{A(t_{\pm}(\beta)\circledast v_{\beta})}{2p^*_{\alpha}}\nonumber\\
		&& -\frac{1}{q\gamma_{p,q}}\left(\left\| \nabla (t_{\pm}(\beta)\circledast v_{\beta})\right\|_p^p+s[t_{\pm}(\beta)\circledast v_{\beta}]_{s,p}^p-A(t_{\pm}(\beta)\circledast v_{\beta})\right)	\nonumber\\
		& = & \left(\frac{1}{p}-\frac{1}{q\gamma_{p,q}}\right)\left\| \nabla (t_{\pm}(\beta)\circledast v_{\beta}) \right\|_p^p+\left(\frac{1}{p}-\frac{s}{q\gamma_{p,q}}\right)[t_{\pm}(\beta)\circledast v_{\beta}]_{s,p}^p\nonumber\\
		&& +\left(\frac{1}{q\gamma_{p,q}}-\frac{1}{2p^*_{\alpha}}\right)A(t_{\pm}(\beta)\circledast v_{\beta})\nonumber\\
		& = & \left(\frac{1}{p}-\frac{1}{q\gamma_{p,q}}\right)\left(\frac{t_{\pm}(\beta)\beta}{\tau}\right)^p\left\| \nabla u \right\|_p^p+\left(\frac{1}{p}-\frac{s}{q\gamma_{p,q}}\right)\left(\frac{(t_{\pm}(\beta))^s\beta}{\tau}\right)^p[u]_{s,p}^p\nonumber\\
		&& +\left(\frac{1}{q\gamma_{p,q}}-\frac{1}{2p^*_{\alpha}}\right)\left(\frac{t_{\pm}(\beta)\beta}{\tau}\right)^{2p^*_{\alpha}}A(u).
	\end{eqnarray}
	Here, since $t_{\pm}(\tau)=1$,
	\begin{eqnarray*}
		\left(\frac{t_{\pm}(\beta)\beta}{\tau}\right)^p & = & \left(1+\frac{t_{\pm}(\beta)\beta}{\tau} -1\right)^p=\left(1+\frac{t_{\pm}(\beta)\beta}{\tau}-t_{\pm}(\beta)\right)^p\\
		& = & \left(1+(\beta-\tau)\left(\frac{t_{\pm}(\beta)\beta-\tau t_{\pm}(\tau)}{\tau(\beta-\tau)}\right)\right)^p\\
		& = & 1+ \frac{p(\beta -\tau)}{\tau}\left(\frac{\beta t_{\pm}(\beta)-\tau t_{\pm}(\tau)}{(\beta-\tau)}\right)+\sum_{i=2}^{\infty}\binom{n}{i}\left(\frac{\beta t_{\pm}(\beta)-\tau t_{\pm}(\tau)}{\tau(\beta-\tau)}\right)^i(\beta -\tau)^{i}\\
		& = & 1+\frac{p(\beta-\tau)}{\tau}\left(\frac{d}{d\beta}(\beta t_{\pm}(\beta))\right)_{\beta=\tau} +O((\beta-\tau)^2)\\
		& = & 1+p(\beta-\tau)\left(\frac{1+\tau t_{\pm}'(\tau)}{\tau}\right)+O((\beta-\tau)^2).
	\end{eqnarray*}
	Similarly, 
	\begin{eqnarray*}
		\left(\frac{(t_{\pm}(\beta))^s\beta}{\tau}\right)^p & = & 1+\frac{p(\beta-\tau)}{\tau}\left(\frac{d}{d\beta}(\beta t^s_{\pm}(\beta))\right)_{\beta=\tau} +O((\beta-\tau)^2)\\
		& = & 1+p(\beta-\tau)\left(\frac{1+s\tau t_{\pm}'(\tau)}{\tau}\right)+O((\beta-\tau)^2)
	\end{eqnarray*}
	and 
	\begin{eqnarray*}
		\left(\frac{t_{\pm}(\beta)\beta}{\tau}\right)^{2p^*_{\alpha}} & = & 1+2p^*_{\alpha}(\beta-\tau)\left(\frac{1+\tau t_{\pm}'(\tau)}{\tau}\right)+O((\beta-\tau)^2).
	\end{eqnarray*}
	Thus, \eqref{5.29} becomes
	\begin{eqnarray*}
		E(t_{\pm}(\beta)\circledast v_{\beta}) & = & \left(\frac{1}{p}-\frac{1}{q\gamma_{p,q}}\right)\left\| \nabla u \right\|_p^p +\left(\frac{1}{p}-\frac{s}{q\gamma_{p,q}}\right)[u]_{s,p}^p +\left(\frac{1}{q\gamma_{p,q}}-\frac{1}{2p^*_{\alpha}}\right)A(u)\\
		&& +p(\beta-\tau)\left(\frac{1}{p}-\frac{1}{q\gamma_{p,q}}\right)\left(\frac{1+\tau t'_{\pm}(\tau)}{\tau}\right)\left\| \nabla u \right\|_p^p\\
		&& +p(\beta-\tau)\left(\frac{1}{p}-\frac{s}{q\gamma_{p,q}}\right)\left(\frac{1+\tau st'_{\pm}(\tau)}{\tau}\right)[u]_{s,p}^p\\
		&& + 2p^*_{\alpha}(\beta-\tau)\left(\frac{1}{q\gamma_{p,q}}-\frac{1}{2p^*_{\alpha}}\right)\left(\frac{1+\tau t'_{\pm}(\tau)}{\tau}\right)A(u) +O((\beta-\tau)^2) \\
		& = & \left(\frac{\left\| \nabla u \right\|_p}{p}+\frac{[u]_{s,p}^p}{p}-\frac{A(u)}{2p^*_{\alpha}}\right)-\left(\frac{A(u)-s[u]_{s,p}^p-\left\| \nabla u \right\|_p^p}{q\gamma_{p,q}}\right)+O((\beta-\tau)^2)\\
		&& +\frac{p(\beta-\tau)}{\tau}\left(\left(\frac{1}{p}-\frac{1}{q\gamma_{p,q}}\right)\left(1+\tau t'_{\pm}(\tau)\right)\left\| \nabla u \right\|_p^p\right.\\
		&& +\left(\frac{1}{p}-\frac{s}{q\gamma_{p,q}}\right)\left(1+\tau st'_{\pm}(\tau)\right)[u]_{s,p}^p\\
		&&\left. -\left(\frac{1}{p}-\frac{2p^*_{\alpha}}{pq\gamma_{p,q}}\right)\left(1+\tau t_{\pm}'(\tau)\right)A(u)\right)\\
		& = & E(u) +O((\beta-\tau)^2)+\frac{p(\beta-\tau)}{\tau}\left( \left(\frac{1}{p}+\frac{\tau t_{\pm}'(\tau)}{p}\right)\left(\left\| \nabla u \right\|_p^p+s[u]_{s,p}^p-A(u)\right)\right.\\
		&& -\frac{1}{q\gamma_{p,q}}\left( \left\| \nabla u \right\|_p^p+s[u]_{s,p}^p -\frac{2p^*_{\alpha}}{p}A(u)\right) +\frac{(1-s)}{p}[u]_{s,p}^p\\
		&&\left. -\frac{\tau t_{\pm}'(\tau)}{q\gamma_{p,q}}\left( \left\| \nabla u \right\|_p^p +s^2[u]_{s,p}^p -\frac{2p^*_{\alpha}}{p}A(u)\right)\right)\\
		& = & E(u) +\frac{p(\beta-\tau)}{\tau}\left( \frac{(1+\tau t'_{\pm}(\tau))}{p}\left(\mu \gamma_{p,q} \left\| u \right\|_q^q\right)+\frac{(1-s)}{p}\right.\\
		&& -\frac{1}{q\gamma_{p,q}}\left(\mu \gamma_{p,q} \left\| u \right\|_q^q+A(u)-\frac{2p^*_{\alpha}}{p}A(u)\right)\\
		&&\left. -\frac{\tau t'_{\pm}(\tau)}{q\gamma_{p,q}}\left( (s^2-s)[u]_{s,p}^p+\mu \gamma_{p,q} \left\| u \right\|_q^q +\left(1-\frac{2p^*_{\alpha}}{p}\right)A(u)\right)\right)+O((\beta-\tau)^2)\\
		& = & E(u) +\frac{p(\beta-\tau)}{\tau}\left(\mu \gamma_{p,q} \left\| u \right\|_q^q\left(\frac{1+\tau t_{\pm}'(\tau)}{p}-\frac{1+\tau t_{\pm}'(\tau)}{p}\right)\right.\\
		&& +A(u)\left(-\frac{1}{q\gamma_{p,q}}+\frac{2p^*_{\alpha}}{pq\gamma_{p,q}}-\frac{\tau t'_{\pm}(\tau)}{q\gamma_{p,q}}+\frac{2p^*_{\alpha}\tau t'_{\pm}(\tau)}{pq\gamma_{p,q}}\right)\\
		&& \left.+[u]_{s,p}^p\left(\frac{(1-s)}{p}+\frac{s(1-s)\tau t'_{\pm}(\tau)}{q\gamma_{p,q}}\right)\right) +O((\beta-\tau)^2)\\
		& = & E(u) +\frac{(1-s)(q\gamma_{p,q}-sp)(\beta-\tau)}{q\gamma_{p,q}\tau}[u]_{s,p}^p +O((\beta-\tau)^2)\\
		&& \frac{(\beta-\tau)(1+\tau t'_{\pm}(\tau))}{\tau q \gamma_{p,q}}\left(\mu q \gamma_{p,q}^2\left\| u \right\|_q^q+2p^*_{\alpha}A(u)-ps^2[u]_{s,p}^p\right.\\
		&& \left. p\left(\mu \gamma_{p,q}\left\| u \right\|_q^q+A(u)-s[u]_{s,p}^p\right)\right),
	\end{eqnarray*}
	since $M(u)=0$. Hence by \eqref{5.28}, we get
	\begin{eqnarray*}
		E(t_{\pm}(\beta)\circledast v_{\beta}) & = & E(u)- \frac{\mu (\beta-\tau)(1-\gamma_{p,q})}{\tau}\left\| u \right\|_q^q+\frac{(\beta-\tau)(1-s)}{\tau}[u]_{s,p}^p +O((\beta-\tau)^2).
	\end{eqnarray*}
	Then, for $\mu>0$, large enough we have:
	$$\frac{\partial E}{\partial\beta}(t_{\pm}(\beta)\circledast v_{\beta})_{\beta=\tau}= -\frac{\mu(1-\gamma_{p,q})}{\tau}\left\| u \right\|_q^q+\frac{(1-s)}{\tau}[u]_{s,p}^p<0, \text{ for all }\tau<\min\{\tau_1,\tau_2\}$$
	that is, $E(t_{\pm}(\beta)\circledast v_{\beta})$ is strictly decreasing corresponding to $\beta$. Therefore, for all $\tau<\beta <\min\{\tau_1,\tau_2\}$
	$$E(t_{\pm}(\beta)\circledast v_{\beta}) < E(t_{\pm}(\tau)\circledast v_{\tau})=E(u).$$
\end{proof}
\noindent Denoting $\mathcal{M}_{r,\tau}^-:=\mathcal{M}_{\tau}^-\cap W^{1,p}_r(\mathbb{R}^N)$, we get $m_{r,\tau}^-:=\displaystyle \inf_{u\in \mathcal{M}_{r,\tau}^-}E(u)=\inf_{u\in \mathcal{M}_{\tau}^-}E(u)=m_{\tau}^-$, by symmetrization and the fact that $\mathcal{M}_{r,\tau}^-\subset\mathcal{M}_{\tau}^-$. Now, let us prove our final result:
	
	\begin{myproof}{Theorem}{\ref{Theorem 2}}
		Let $\{\bar{u}_n\}_{n\in \mathbb{N}}$ be the minimizing sequence for $E$ on $\mathcal{M}_{r,\tau}^-$, then by Ekeland variational principle, \cite[Theorem~1.1]{Ghoussoub}, we can find a sequence $\{u_n\}_{n\in \mathbb{N}}\subset \mathcal{M}_{r,\tau}^-$ such that
		\begin{equation}\label{4.20}
			\left\{
			\begin{array}{cl}
				\left\| \bar{u}_n-u_n\right\|_{W^{1,p}(\mathbb{R}^N)}\rightarrow 0 & \text{ as } n\rightarrow\infty,\\
				E(u_n) \rightarrow m_{r,\tau}^- & \text{ as } n\rightarrow \infty,\\
				M(u_n)\rightarrow 0 & \text{ as } n\rightarrow \infty,\\
				E'|_{\mathcal{M}_{r,\tau}^-}(u_n)\rightarrow 0 & \text{ as } n\rightarrow \infty.
			\end{array}
			\right.
		\end{equation}
		Now, by \eqref{4.20} we have
		\begin{eqnarray}\label{4.21}
			m_{r,\tau}^- & = & \lim_{n\rightarrow \infty}E(u_n)=\lim_{n\rightarrow \infty}\left(E(u_n)-\frac{M(u_n)}{p}\right)\nonumber\\
			& = & \lim_{n\rightarrow \infty}\left(\frac{\mu}{q}\left(\frac{q\gamma_{p,q}}{p}-1\right)\left\| u_n \right\|_q^q+\frac{(1-s)}{p}[u_n]_{s,p}^p+\left(\frac{2p^*_{\alpha}-p}{2pp^*_{\alpha}}\right)A(u_n)\right).
		\end{eqnarray}
		Now, since $E(u_n)\leq m_{r,\tau}^-+1$, for large $n\in \mathbb{N}$, by Gagliardo-Nirenberg inequality \eqref{G_N_inequality}
		\begin{eqnarray*}
			\frac{(2p^*_{\alpha}-p)}{2pp^*_{\alpha}}T(u_n)^p & \leq &  \frac{(2p^*_{\alpha}-p)}{2pp^*_{\alpha}}\left\| \nabla u_n \right\|_p^p+\frac{(2p^*_{\alpha}-sp)}{2pp^*_{\alpha}}[u_n]_{s,p}^p\\
			& = & E(u_n)-\frac{1}{2p^*_{\alpha}}M(u_n)+\frac{\mu}{q}\left(1-\frac{qq\gamma_{p,q}}{2p^*_{\alpha}}\right) \left\| u_n \right\|_q^q\\
			& \leq & m_{r,\tau}^-+1+\frac{C_{N,p,q}(2p^*_{\alpha}-q\gamma_{p,q})}{2p^*_{\alpha}}\tau^{q(1-\gamma_{p,q})}T(u_n)^{q\gamma_{p,q}}.
		\end{eqnarray*}
		Then, either $T(u_n)\leq 1$ or 
		\begin{equation*}
			\frac{(2p^*_{\alpha}-p)}{2pp^*_{\alpha}}T(u_n)^p  \leq  m_{r,\tau}^-+1+\frac{C_{N,p,q}(2p^*_{\alpha}-q\gamma_{p,q})}{2p^*_{\alpha}}\tau^{q(1-\gamma_{p,q})}T(u_n)^{p}.
		\end{equation*}
		Hence for $$\tau < \tau_3=\left(\frac{(2p^*_{\alpha}-p)}{\mu p(2p^*_{\alpha}-q\gamma_{p,q})C_{N,p,q}}\right)^{\frac{1}{q(1-\gamma_{p,q})}},$$
		we get
		$$T(u_n)^p \leq \frac{m_{r,\tau}^-+1}{\left(\frac{(2p^*_{\alpha}-p)}{2pp^*_{\alpha}}-\frac{C_{N,p,q}(2p^*_{\alpha}-q\gamma_{p,q})}{2p^*_{\alpha}}\tau^{q(1-\gamma_{p,q})}\right)}.$$
		Thus, $\{u_n\}_{n\in \mathbb{N}}$ is bounded and hence weakly convergent upto a subsequence in $W^{1,p}(\mathbb{R}^N)$. Denoting the weakly convergent subsequence as $\{u_n\}_{n\in \mathbb{N}}$ itself, let $u_0\in W^{1,p}_r(\mathbb{R}^N)$ be such that $u_n\rightharpoonup u_0$, weakly as $n\rightarrow \infty$. Thanks to the compact embedding $W^{1,p}_r(\mathbb{R})\hookrightarrow L^r(\mathbb{R}^N)$, for all $r\in (p,p^*)$, we get $u_n\rightarrow u_0$ in $L^q(\mathbb{R}^N)$. Next, we claim that $u_0\neq 0$.\\
		Suppose $u_0=0$, then 
		$$0 = \lim_{n\rightarrow \infty}M(u_n) = \lim_{n\rightarrow \infty}\left(\left\| \nabla u_n \right\|_p^p+s[u_n]_{s,p}^p-A(u_n)\right),$$
		and hence $\displaystyle \lim_{n\rightarrow \infty}\left(\left\| \nabla u_n \right\|_p^p+s[u_n]_{s,p}^p\right)= \lim_{n\rightarrow \infty}A(u_n)$. Since $\{u_n\}$ is bounded in $W^{1,p}(\mathbb{R}^N)$, the sequence $\{\left\| \nabla u_n\right\|_p^p+s[u_n]_{s,p}^p\}$ is convergent upto a subsequence in $\mathbb{R}$.
		Now, let $$l=\displaystyle \lim_{n\rightarrow \infty}\left(\left\| \nabla u_n \right\|_p^p+s[u_n]_{s,p}^p\right)= \lim_{n\rightarrow \infty}A(u_n),$$
		then by \eqref{S_alpha,p}, we get $l(\mathbb{S}-l^{\frac{2p^*_{\alpha}-p}{2p^*_{\alpha}}})\leq 0$. Thus, either $l=0$ or $l\geq \mathbb{S}^{\frac{2p^*_{\alpha}}{2p^*_{\alpha}-p}}$. For $l\geq \mathbb{S}^{\frac{2p^*_{\alpha}}{2p^*_{\alpha}-p}}$,  by \eqref{4.21} we get:
		\begin{eqnarray*}
			m_{\tau}^-=m_{r,\tau}^- & = & \lim_{n\rightarrow \infty}\left(\frac{\mu}{q}\left(\frac{q\gamma_{p,q}}{p}-1\right)\left\| u_n \right\|_q^q+\frac{(1-s)}{p}[u_n]_{s,p}^p+\left(\frac{2p^*_{\alpha}-p}{2pp^*_{\alpha}}\right)A(u_n)\right)\\
			& \geq & \lim_{n\rightarrow \infty}\left(\frac{2^p*_{\alpha}-p}{2pp^*_{\alpha}}\right)A(u_n)\geq \left(\frac{2p^*_{\alpha}-p}{2pp^*_{\alpha}}\right)\mathbb{S}^{\frac{2p^*_{\alpha}}{2p^*_{\alpha}-p}}>m_{\tau}+\left(\frac{2p^*_{\alpha}-p}{2pp^*_{\alpha}}\right)\mathbb{S}^{\frac{2p^*_{\alpha}}{2p^*_{\alpha}-p}},
		\end{eqnarray*}
		but this contradicts \autoref{Lemma 4.1}. Also, if $l=0$, we will end up with $m_{r,\tau}^-=0$, but since $0<m_{\tau}^-=m_{r,\tau}^-$, we get a contradiction. Therefore, $u_0\neq 0$. Now, define $v_n:=u_n-u_0$, clearly $v_n\rightharpoonup 0$ in $W^{1,p}(\mathbb{R}^N)$ as $n\rightarrow \infty$.\\
		Case 1: $\left\| v_n\right\|_{W^{1,p}(\mathbb{R}^N)}\rightarrow 0 $.\\
		In this case, we get strong convergence of $\{u_n\}_{n\in \mathbb{N}}$ in $W^{1,p}(\mathbb{R}^N)$, and hence $u_0\in \mathcal{M}_{r,\tau}^-$ with $E(u_0)=m_{\tau}^-$ and hence $E'_{\mathcal{M}_{\tau}}(u_0)=0$. Thus, by \autoref{Lemma 4}, $u_0$ solves \eqref{prob} for some $\lambda_0\in \mathbb{R}$, and since $M(u_0)=0$, we have:
		\begin{equation*}
			\lambda_0\tau^p  =  \left\| \nabla u_0 \right\|_p^p+[u_0]_{s,p}^p-\mu \left\| u_0 \right\|_q^q-A(u_0)= (1-s)[u_0]_{s,p}^p+\mu(\gamma_{p,q}-1)\left\| u_0 \right\|_q^q<0,
		\end{equation*}
		for sufficiently large $\mu>0$. Hence, taking $u_{\tau}^-=u_0$ and $\lambda_{\tau}^-=\lambda_0$, we are done.\\
		Case 2: $\displaystyle \lim_{n\rightarrow \infty}\left\| v_n \right\|_{W^{1,p}(\mathbb{R}^N)}\neq 0$, that is, $\left\| v_n \right\|_{W^{1,p}(\mathbb{R}^N)} \geq \tilde{C}>0$ for large $n\in \mathbb{N}$.\\
		Let $\left\| u_0 \right\|_p=r_0$, then by Fatou's lemma, we have $0<r_0\leq \tau$. Now, either $A(v_n)\rightarrow 0$ or there exists a constant $\bar{C}>0$ such that $A(v_n)\geq \bar{C}$ for large $n\in \mathbb{N}$. Let us analyse the two subcases separately:\\
		Subcase 1: $A(v_n)\rightarrow 0$ as $n\rightarrow \infty$.\\
		Since $u_0\in S(r_0)$, by \autoref{Lemma 2.3}, there exists $c_0>0$ such that $c_0\circledast u_0 \in \mathcal{M}_{r,r_0}^-$. Thus, by \cite[lemma~2.4]{Moroz2013groundstates}, compact embedding of $W^{1,p}_r(\mathbb{R}^N)$ in $L^q(\mathbb{R}^N)$, Fatou's lemma and \autoref{Lemma 2.3} we get
		\begin{eqnarray}\label{4.22}
			m_{\tau}^- & = & \lim_{n\rightarrow \infty}E(u_n) \geq \lim_{n\rightarrow \infty}E(c_0 \circledast u_n) \nonumber\\
			& = & \lim_{n\rightarrow \infty}\left(\frac{c_0^p\left\| \nabla u_n \right\|_p^p}{p}+\frac{c_0^{sp}[u_n]_{s,p}^p}{p}-\frac{\mu c_0^{q\gamma_{p,q}}\left\| u_n \right\|_q^q}{q}-\frac{c_0^{2p^*_{\alpha}}A(u_n)}{2p^*_{\alpha}}\right)\nonumber\\
			& \geq & \frac{c_0^p\left\| \nabla u_0 \right\|_p^p}{p}+\frac{c_0^{sp}[u_0]_{s,p}^p}{p}-\frac{\mu c_0^{q\gamma_{p,q}}\left\| u_0 \right\|_q^q}{q}-\frac{c_0^{2p^*_{\alpha}}A(u_0)}{2p^*_{\alpha}}\nonumber\\
			& = & E(c_0 \circledast u_0)\geq m_{r_0}^-.
		\end{eqnarray}
		Since $0<r_0\leq \tau$, for any $u\in \mathcal{M}_{r_0}^-$, by \autoref{Lemma5.3} we also can find $v\in \mathcal{M}_{\tau}^-$ such that $E(u)>E(v)\geq \displaystyle \inf_{u\in \mathcal{M}_{\tau}^-}E(u)$ and hence $m_{r_0}^-\geq m_{\tau}^-$. Therefore, $m_{\tau}^-=m_{r_0}^-$. Now, we claim that $r_0=\tau$ and hence $u_{\tau}^-=c_0\circledast u_0$ is the required solution to \eqref{prob} corresponding to some $\lambda_{\tau}^-$ with $\lambda_{\tau}^-<0$ for sufficiently large $\mu>0$ as done in case 1.\\
		Suppose if $0<r_0<\tau<\min\{\tau_0,\tau_1\}$, then by \autoref{Lemma5.3}, there exists $\bar{v}\in \mathcal{M}_{\tau}^-$ such that $E(c_0\circledast u_0)>E(\bar{v})$, then by \eqref{4.22} we have
		$$m_{r_0}^-=E(c_0\circledast u_0)>E(\bar{v})\geq m_{\tau}^-,$$
		from which by using $m_{r_0}^-=m_{\tau}^-$, we get contradiction, thus $r_0=\tau$.\\
		Subcase 2: $A(v_n)\geq \bar{C}>0$ for large $n\in \mathbb{N}$.\\
		For every $n\in \mathbb{N}$, define $$s_n:=\left(\frac{\left\| \nabla v_n\right\|_p^p}{A(v_n)}\right)^{\frac{1}{(2p^*_{\alpha}-p)}}.$$
		Clearly, by boundedness of $\{\frac{1}{A(v_n)}\}_{n\in \mathbb{N}}$ and $\{u_n\}_{n\in \mathbb{N}}$ in $W^{1,p}(\mathbb{R}^N)$, $\{s_n\}_{n\in \mathbb{N}}$ is a bounded sequence in $\mathbb{R}$. Now, since $u_0\in S(r_0)$, by \autoref{Lemma 2.3} there exists $c_0>0$ such that $c_0\circledast u_0\in \mathcal{M}_{r_0}^-$. We claim that $s_n\geq c_0$ upto subsequence.\\
		Suppose $s_n<c_0$ for all $n\in \mathbb{N}$, defining 
		$$E_0(u):=\frac{\left\| \nabla u \right\|_p^p}{p}-\frac{A(u)}{2p^*_{\alpha}},$$
		by \autoref{Lemma 2.3}, Brezis Lieb lemma and \cite[lemma~2.4]{Moroz2013groundstates} we get
		\begin{eqnarray}\label{4.23}
			m_{\tau}^- & = & \lim_{n\rightarrow \infty}E(u_n) \geq \lim_{n\rightarrow \infty}E(s_n\circledast u_n) =\lim_{n\rightarrow \infty}\left(E(s_n\circledast u_0)+E(s_n \circledast v_n)\right)\nonumber\\
			& \geq & \lim_{n\rightarrow \infty}\left(E(s_n\circledast u_0)+E_0(s_n \circledast v_n)\right)\geq m_{r_0}^++ \lim_{n\rightarrow \infty}E_0(s_n\circledast v_n).
		\end{eqnarray}
		Now, by \eqref{S_alpha,p}
		$$E_0(s_n\circledast v_n) = \left(\frac{2p^*_{\alpha}-p}{2pp^*_{\alpha}}\right)\left(\frac{\left\| \nabla v_n \right\|_p^p }{A(v_n)^{\frac{p}{2p^*_{\alpha}}}}\right)^{\frac{2p^*_{\alpha}}{2p^*_{\alpha}-p}}\geq \left(\frac{2p^*_{\alpha}-p}{2pp^*_{\alpha}}\right)\mathbb{S}^{\frac{2p^*_{\alpha}}{2p^*_{\alpha}-p}}.$$
		Thus, by \autoref{Lemma5.3} 
		$$m_{\tau}^-\geq m_{r_0}^++\left(\frac{2p^*_{\alpha}-p}{2pp^*_{\alpha}}\right)\mathbb{S}^{\frac{2p^*_{\alpha}}{2p^*_{\alpha}-p}}\geq m_{\tau}^++\left(\frac{2p^*_{\alpha}-p}{2pp^*_{\alpha}}\right)\mathbb{S}^{\frac{2p^*_{\alpha}}{2p^*_{\alpha}-p}}$$
		which yields a contradiction to \autoref{Lemma 4.1}. Thus, there exists a subsequence (denoted as $\{s_n\}_{n\in \mathbb{N}}$ itself), such that $s_n\geq c_0$ for all $n\in \mathbb{N}$. Now, again proceeding as in \eqref{4.23}
		\begin{equation*}
			m_{\tau}^-  =  \lim_{n\rightarrow \infty}E(u_n) \geq \lim_{n\rightarrow \infty}E(c_0\circledast u_n)\geq \lim_{n\rightarrow \infty}\left(E(c_0\circledast u_0)+E_0(c_0\circledast v_n)\right) \geq E(c_0\circledast u_0),
		\end{equation*}
		because $c_0\leq s_n$ which implies that $$\frac{c_0^{2p^*_{\alpha}}A(v_n)}{\left\| \nabla v_n \right\|_p^p}\leq c_0^p,$$ and hence
		$$E_0(c_0\circledast v_n)\geq \left(\frac{2p^*_{\alpha}-p}{2pp^*_{\alpha}}\right)c_0^{2p^*_{\alpha}}A(v_n)\geq0.$$
		Therefore, $E(c_0\circledast u_0) \leq m_{\tau}^-$. Also, since $c_0\circledast u_0\in \mathcal{M}_{r_0}^-$, by \autoref{Lemma5.3},
		$$m_{\tau}^-\geq E(c_0\circledast u_0)\geq m_{r_0}^-\geq m_{\tau}^-.$$
		Hence $E(c_0\circledast u_0) = m_{\tau}^-$, thus taking $u_{\tau}^-=c_0\circledast u_0$ we get the required result.
	\end{myproof}
	\begin{remark}
		We would like to highlight that, for $p+\frac{p^2}{N}<q<p^*$, the work of \cite{Nidhi2025Normalized_Asymp} can be extended to achieve a mountain pass type solution. However, the case $p+\frac{sp^2}{N}\leq q\leq p+\frac{p^2}{N}$ remains unresolved. In this case, one observes that $\psi_{u}\rightarrow 0^+$ and $t\rightarrow -\infty$ and $\psi_{u}\rightarrow -\infty$ as $t \rightarrow +\infty$, ensuring the existence of at least one critical point corresponding to a global maximum. Moreover, the behaviour of $\psi_{u}''$ indicates the presence of a point of inflection as well. Thus, in this case $\mathcal{M}_{\tau}^0$ is non-empty and consequently the above approach cannot be directly followed.
	\end{remark}
	\section*{Acknowledgement}
	The author, Nidhi Nidhi(PMRF ID - 1402685), is supported by the Ministry of Education, Government of India, under the Prime Minister’s Research Fellows (PMRF) scheme. \\
	\noindent The author K. Sreenadh thanks the Department of Science and Technology (DST) India, for providing support under the Improvement of S\&T Infrastructure (FIST) programme. (Project No. SR/FST/MS-1/2019/45).

\end{document}